\documentclass[a4paper,12pt]{article}
\usepackage[utf8x]{inputenc}

\usepackage{amsmath}
\usepackage{amsfonts}
\usepackage{amssymb}
\usepackage{amsthm} 
\usepackage{graphicx} 
\usepackage[all]{xy} 

\usepackage{fullpage}

\usepackage{setspace}
\newcommand{\Z}{\ensuremath{\mathbb{Z}}}

\newcommand{\N}{\ensuremath{\mathbb{N}}}
\newcommand{\C}{\ensuremath{\mathbb{C}}}

\newcommand{\A}{\ensuremath{\mathcal{A}}}
\newcommand{\B}{\ensuremath{\mathcal{B}}}
\newcommand{\D}{\ensuremath{\mathcal{D}}}
\newcommand{\E}{\ensuremath{\mathcal{E}}}
\renewcommand{\P}{\ensuremath{\mathbf{P}}}
\newcommand{\QQ}{\ensuremath{\mathbf{Q}}}
\newcommand{\PP}{\ensuremath{\mathcal{P}}}
\newcommand{\U}{\ensuremath{\mathcal{U}}}
\newcommand{\M}{\ensuremath{\mathcal{M}}}
\renewcommand{\O}{\ensuremath{\mathcal{O}}}
\renewcommand{\phi}{\ensuremath{\varphi}}
\newcommand{\G}{\ensuremath{\mathcal{G}}}
\newcommand{\Con}{\ensuremath{\mathrm{Con}}}
\newcommand{\Nuc}{\ensuremath{\mathrm{Nuc}}}
\newcommand{\Eq}{\ensuremath{\mathrm{Eq}}}
\newcommand{\Sub}{\ensuremath{\mathrm{Sub}}}
\newcommand{\CC}{\ensuremath{\mathcal{C}}}
\newcommand{\CCC}{\ensuremath{\mathcal{C}}}

\newcommand{\op}{\ensuremath{\mathrm{op}}}

\newcommand{\Hom}{\ensuremath{\mathrm{Hom}}}
\newcommand{\Sets}{\ensuremath{\mathbf{Sets}}}
\newcommand{\Sh}{\ensuremath{\mathrm{Sh}}}
\newcommand{\up}{\ensuremath{{\uparrow\,}}}
\newcommand{\down}{\ensuremath{{\downarrow\,}}}
\newcommand{\embeds}{\hookrightarrow}

\newcommand{\Idl}{\mathrm{Idl}}
\newcommand{\F}{\ensuremath{\mathcal{F}}}
\newcommand{\y}{\ensuremath{\mathbf{y}}}
\newcommand{\atom}{\ensuremath{\mathrm{atom}}}
\newcommand{\dense}{\ensuremath{\mathrm{dense}}}

\theoremstyle{definition} 
\newtheorem{theorem}{Theorem}[section]
\newtheorem{definition}[theorem]{Definition}
\newtheorem{lemma}[theorem]{Lemma}
\newtheorem{proposition}[theorem]{Proposition}
\newtheorem{corollary}[theorem]{Corollary}
\newtheorem{example}[theorem]{Example}

\newtheorem*{warning}{Warning}

\newtheorem{convention}{Convention}

\newtheorem*{notation*}{Notation}

\title{Grothendieck topologies on posets\footnote{The author would like to thank Mai Gehrke, Sam van Gool, Klaas Landsman, Frank Roumen and Sander Wolters for their comments and suggestions. In particular Sam van Gool made a significant contribution by pointing out that several notions in frame and locale theory such as nuclei are equivalent to the notion of a Grothendieck topology on a poset. This research has been financially supported by the Netherlands Organisation for Scientific Research (NWO) under TOP-GO grant no. 613.001.013 (The logic of composite quantum systems).}}
\author{A.J. Lindenhovius}

\begin{document}

\maketitle
 \tableofcontents

\begin{abstract}
We investigate Grothendieck topologies (in the sense of sheaf theory) on a poset $\P$ that are generated by some subset of $\P$. We show that such Grothendieck topologies exhaust all possibilities if and only if $\P$ is Artinian. If $\P$ is not Artinian, other families of Grothendieck topologies on $\P$ exist that are not generated by some subset of $\P$, but even those are related to the Grothendieck topologies generated by subsets. Furthermore, we investigate several notions of equivalences of Grothendieck topologies, and using a posetal version of the Comparison Lemma, a sheaf-theoretic result known as the Comparison Lemma, going back to Grothendieck et al \cite{SGA4}, we calculate the sheaves with respect to most of the Grothendieck topologies we have found.
\end{abstract}

\section*{Introduction and motivation}

In this paper, we aim to classify Grothendieck topologies on a poset. This is motivated by the research project of describing quantum theory in terms of topos theory (see e.g., \cite{BI}, \cite{DI}, \cite{HLS}). Here the central object of research are the topoi $\Sets^{\CCC(\A)}$ and $\Sets^{\CCC(\A)^\op}$, where $\CCC(\A)$ is the poset of the commutative C*-subalgebras of some C*-algebra $\A$ ordered by inclusion. $\A$ describes some quantum system, whereas the elements of $\CCC(\A)$ are interpreted as ``classical snapshots of reality''. The motivation behind this approach is Niels Bohr's doctrine of classical concepts, which, roughly speaking, states that a measurement provides a classical snapshot of quantum reality, and knowledge of all classical snapshots should provide a picture of quantum reality that is as complete as (humanly) possible. In terms of C*-algebras, this means we should be able to reconstruct the algebra $\A$ in some way from the posetal structure of $\CCC(\A)$. Partial results in this direction can be found in \cite{DH} and \cite{Hamhalter}, where, at least for certain classes of C*-algebras and von Neumann algebras, it is proven that the Jordan structure of $\A$ may be recovered from the poset $\CCC(\A)$ of commutative subalgebras of $\A$. Since there is a C*-algebra $\A$ that is not isomorphic to its opposite C*-algebra $\A^0$, although $\CCC(\A)\cong\CCC(\A^0)$ as posets (see \cite{Connes}), it is known that the C*-algebraic structure of a C*-algebra cannot always be completely recovered from the posetal structure of $\CCC(\A)$. Therefore, extra information is needed, which might be found in sheaf theory, for instance in the form of a Grothendieck topology on $\CCC(\A)$.

The simplest C*-algebras are the finite-dimensional algebras $\B(\C^n)$ of $n\times n$-matrices, where $n$ is a finite positive integer. However, even the posets $\CCC_n=\CCC(\B(\C^n))$ are already far from trivial (see for instance \cite[Example 5.3.5]{Heunen}), hence interesting. Familiar with the posetal structure of $\CCC_n$, one easily sees that these posets are Artinian, i.e., posets for which every non-empty downwards directed subset contains a least element. Moreover, we show in Appendix A that $\CC(\A)$ is Artinian if and only if $\A$ is finite dimensional. The Artinian property is extremely powerful, since it allows one to use a generalization of the principle of induction called \emph{Artinian induction}\index{Artinian induction}.

Two main sources of this paper are \cite{J&T} and \cite{EGP}. In the first, several order-theoretic notions are related to the notion of Grothendieck topologies, whilst in the second, the relation is between these order-theoretic notions and the Artinian property is investigated. This article aims to combine both articles. We define a family of Grothendieck topologies on a poset $\P$ generated by a subset of $\P$ and show that these Grothendieck topologies in fact exhaust the possible Grothendieck topologies on the given poset if and only $\P$ is Artinian, where we use Artinian induction for the 'if' direction.

A poset equipped with a Grothendieck topology is called a \emph{site}. If the extra structure on a $\CCC(\A)$ is indeed given by a Grothendieck topology, we also need a notion of 'morphisms of sites'. Finally, we calculate the sheaves corresponding to all classes of Grothendieck topologies we have found. For this calculation, we make use of a posetal version of a sheaf-theoretic result known as the Comparison Lemma, going back to Grothendieck et al \cite{SGA4} (see also \cite{Elephant2}), which relates sheaves on a given category to sheaves on some subcategory.

In the appendix, we show that the notion of a Grothendieck topology on $\P$ is equivalent with the notion of a frame quotient of the frame $\D(\P)$ of down-sets of $\P$, which in turn is equivalent with the notion of a nuclei on $\D(\P)$, a congruence on $\D(\P)$, a sublocale of $\D(\P)$ if we consider $\D(\P)$ as a locale. Since all these notions are equivalent, this gives a description of nuclei, congruences, and sublocales of $\D(\P)$ that correspond to Grothendieck topologies corresponding to subsets of $\P$.

\section{Preliminaries on order theory}\label{Order Theory}\label{Noetherian induction}
This section can be skipped if one is familiar with the basic concepts of order theory. We refer to \cite{DP} for a detailed exposition of order theory.

\begin{definition}
 A \emph{poset} $(\P,\leq)$ is a set $\P$ equipped with a \emph{(partial) order}\index{order!partial} $\leq$. That is, $\leq$, is binary relation, which is reflexive, antisymmetric and transitive. We often write $\P$ instead of $(\P,\leq)$ if it is clear which order is used. If either $p\leq q$ or $q\leq p$ for each $p,q\in\P$, we say that $\leq$ is a \emph{linear order}\index{order!linear}, and we call $\P$ a \emph{linearly ordered set}\index{linearly ordered set}. Given a poset $\P$ with order $\leq$, we define the opposite poset $\P^\op$ as the poset with the same underlying set $\P$, but where $p\leq q$ if and only if $q\leq p$ in the original order. .
\end{definition}

\begin{definition}
 Let $(\P,\leq)$ be a poset and $M\subseteq \P$ a subset. We say that $M$ is an \emph{up-set}\index{up-set} if for each $x\in M$ and $y\in \P$ we have $x\leq y$ implies $y\in M$. Similarly, $M$ is called a \emph{down-set}\index{down-set} if for each $x\in M$ and $y\in \P$ we have $x\geq y$ implies $y\in M$. Given an element $x\in \P$, we define the up-set and down-set generated by $x$ by $\up x=\{y\in \P:y\geq x\}$ and $\down x=\{y\in \P:y\leq x\}$, respectively. We can also define the up-set generated by a subset $M$ of $\P$ by $\up M=\{x\in \P:x\geq m$ for some $m\in M\}=\bigcup_{m\in M}\up m$, and similarly, we define the down-set generated by $M$ by $\down M=\bigcup_{m\in M}\down m$. We denote the collection of all up-sets of a poset $\P$ by $\U(\P)$ and the set of all down-sets by $\D(\P)$. If we want to emphasize that we use the order $\leq$, we write $\U(\P,\leq)$ instead of $\U(\P)$.
\end{definition}

\begin{lemma}\label{cosievegen}
Let $\P$ be a poset and $M\subseteq \P$. Then $\down M$ is the smallest down-set containing $M$, and $M=\down M$ if and only if $M$ is an down-set.
\end{lemma}

\begin{definition}
 Let $M$ be a non-empty subset of a poset $\P$.
\begin{enumerate}
 \item An element $x\in M$ such that $\up x\cap M=\{x\}$ is called a \emph{maximal element}\index{maximal element}. The set of maximal elements of $M$ is denoted by $\max M$. The set of all \emph{minimal elements}\index{minimal element} $\min M$ is defined dually.
 \item If there is an element $x\in M$ such that $x\geq y$ for each $y\in M$, we call $x$ the \emph{greatest element}\index{greatest element} of $M$, which is necesarrily unique. The \emph{least element}\index{least element} of $M$ is defined dually. If $\P$ itself has a greatest element, it is denoted by $1$. Dually, the least element of $\P$ is denoted by $0$.
 \item An element $x\in \P$ is called an \emph{upper bound}\index{upper bound} of $M$ if $y\leq x$ for each $y\in M$. If the set of upper bounds of $M$ has a least element, it is called the \emph{join}\index{join} of $M$, which is necesarrily unique and is denoted by $\bigvee M$. Dually, we can define \emph{lower bounds}\index{lower bound} and the \emph{meet}\index{meet} of $M$, which is denoted by $\bigwedge M$. The (binary) join and the (binary) meet of $\{x,y\}$ are denoted by $x\vee y$ and $x\wedge y$, respectively.
 \item $M$ is called upwards (downwards) \emph{directed}\index{directed subset of a poset} if for every $x,y\in M$ there is an upper (lower) bound $z\in M$.
\item $M$ is called a \emph{filter}\index{filter of a poset} if $M\in\U(\P)$ and $M$ is downwards directed. We denote the set of filters of $\P$ by $\F(\P)$. A filter is called \emph{principal}\index{principal filter} if it is equal to $\up p$ for some $p\in \P$.
\item $M$ is called an \emph{ideal}\index{ideal of a poset} if $M\in\D(\P)$ and $M$ is upwards directed. We denote the set of ideals of $\P$ by $\Idl(\P)$. An ideal is called \emph{principal}\index{principal ideal} if it is equal to $\down p$ for some $p\in \P$.
\end{enumerate}
\end{definition}
If we consider $\P$ as a category, meets are exactly products, joins are coproducts, and the greatest (least) element of $\P$ is exactly the terminal (initial) object. These concepts are unique in posets, whereas in arbitrary categories they are only unique up to isomorphism. This follows from the fact that there is at most one morphism between two elements in a poset, hence isomorphic elements in a poset are automatically equal.

\begin{definition}
Let $\P$ be a poset. Then $\P$ is called
\begin{enumerate}
 \item \emph{Artinian}\index{Artinian poset} (\emph{Noetherian})\index{Noetherian poset} if every non-empty subset contains a minimal (maximal) element;
 \item a \emph{meet-semilattice}\index{meet-semilattice} (\emph{join-semilattice})\index{join-semilattice} if all binary meets (joins) exist;
 \item a \emph{lattice}\index{lattice} if all binary meets and binary joins exist;
 \item a \emph{complete lattice}\index{complete lattice} if all meets and joins exist.
\item a \emph{distributive lattice}\index{distributive lattice} if it is a lattice such that the \emph{distributive law}\index{distributive law}
 \begin{equation}
  x\wedge(y\vee z)=(x\wedge y)\vee(x\wedge z)
 \end{equation}
 holds for each $x,y,z\in \P$
 \item a \emph{frame} if all joins and all finite meets exists and the \emph{infinite distributive law}\index{infinite distributive law}
 \begin{equation}\label{eq:infdist}
  x\wedge\bigvee_i y_i=\bigvee_i(x\wedge y_i)
 \end{equation}
 holds for each element $x$ and each family $y_i$ in $\P$.
\end{enumerate}
\end{definition}

Remark that in a meet-semilattice (and therefore also in a lattice and in a frame), we have $x\wedge y=x$ if and only if $x\leq y$.

\begin{lemma}\label{lem:completelattice}
 Let $\P$ be a poset. If $\P$ has all meets, then $\P$ becomes a complete lattice with join operation defined by
\begin{equation*}
 \bigvee S=\bigwedge\{p\in\P:s\leq p\ \forall s\in S\}
\end{equation*}
for each $S\subseteq\P$. Dually, if $\P$ has all joins, then $\P$ becomes a complete lattice with meet operation defined by
\begin{equation*}
\bigwedge S=\bigvee\{p\in\P:p\leq s\ \forall s\in S\}
\end{equation*}
for each $S\subseteq\P$. In particular a frame is a complete lattice.
\end{lemma}
As a consequence, if $\P$ is a complete lattice, $\P^\op$ is a complete lattice as well.

Finally we define morphisms between posets as follows.
\begin{definition}
 Let $\P,\mathbf{Q}$ be posets and $f:\P\to \mathbf{Q}$ a map. Then $f$ is called
\begin{enumerate}
 \item an \emph{order morphism} if $x\leq y$ implies $f(x)\leq f(y)$ for each $x,y\in \P$;
 \item a \emph{embedding of posets} if $f(x)\leq f(y)$ if and only if $x\leq y$ for each $x,y\in\P$;
 \item a \emph{meet-semilattice morphism} if $\P$ and $\mathbf{Q}$ are both meet-semilattices and\\ $f(x\wedge y)=f(x)\wedge f(y)$ for each $x,y\in \P$;
 \item a \emph{lattice morphism} if $\P$ and $\mathbf{Q}$ are both lattices and $f$ is a meet-semilattice morphism such that $f(x\vee y)=f(x)\vee f(y)$ for each $x,y\in \P$;
 \item a \emph{frame morphism} if $\P$ and $\mathbf{Q}$ are both frames and $f\left(\bigvee M\right)=\bigvee f(M)$,\\ $f\left(\bigwedge N\right)=\bigwedge f(N)$ for each $M\subseteq \P$ and each finite $N\subseteq \P$.
\end{enumerate}
An order morphism, meet-semilattice morphism, lattice morphism, frame morphism is called an \emph{order isomorphism}, \emph{meet-semilattice isomorphism}, \emph{lattice isormphism}, \emph{frame isomorphism}\index{isomorphism}, respectively, if there is a morphism $g:\mathbf{Q}\to\P$ of the same type such that $f\circ g=1_{\mathbf{Q}}$ and $g\circ f=1_\P$.
\end{definition}
Clearly an embedding of posets $f$ is injective. If $f(x)=f(y)$, then $f(x)\leq f(y)$, so $x\leq y$, and in a similar way, we find $y\leq x$, so $x=y$.
The converse does not always hold. Consider for instance the poset $\P=\{p_1,p_2,p_3\}$ with $p_1,p_2<p_3$ and $\mathbf{Q}=\{q_1,q_2,q_3\}$ with $q_1<q_2<q_3$. Then $f:\P\to\mathbf{Q}$ defined by $f(p_i)=q_i$ for $q=1,2,3$ is clearly an injective order morphism (it is even bijective), but $p_1\nleq p_2$, whereas $f(p_1)\leq  f(p_2)$.

\begin{lemma}
 Let $f:\P\to \mathbf{Q}$ be an order isomorphism. Then $f$ respects meets and joins. Moreover, if $\P$ and $\mathbf{Q}$ are both meet-semilattices, lattices or frames, then $f$ is a meet-semilattice morphism, lattice morphism, or frame morphism, respectively, if and only if $f$ is an order isomorphism.
\end{lemma}

\begin{definition}
 Let $\P$, $\mathbf Q$ be posets and $f :\P \to \QQ$, $g : \QQ \to \P$ order morphism. If $f(p) \leq q$ if and only if $p \leq g(q)$ holds for each $p \in\P$ and $q \in \QQ$, then $g$ is called the \emph{upper adjoint}\index{adjoint} of $f$ and $f$ is called the \emph{lower adjoint} of $g$. If the upper adjoint of $f$ exists, it is sometimes denoted by $f_*$. Similarly, the lower adjoint of $g$, if it exists, is sometimes denoted by $g^*$.
\end{definition}
If we consider posets as categories, an upper adjoint is exactly a right adjoint.

The following lemma can be found in \cite[Chapter 7]{DP}. Categorical generalizations of these facts are also well-known, but we will not need them here.
\begin{lemma}\label{lem:adjointequivalent}
Let $\P$, $\mathbf Q$ be posets and $f :\P \to \QQ$, $g : \QQ \to \P$ order morphisms. The following two statements are equivalent:
\begin{enumerate}
\item[(a)] $g$ is the upper adjoint of $f$.
\item[(b)] Both $1_{\P} \leq g \circ f$ and $f \circ g \leq 1_{\QQ}$. 
\end{enumerate}
If these statements hold, then we also have the following:
\begin{enumerate}
\item[(i)] Both $f \circ g \circ f = f$ and $g \circ f \circ g = g$.
\item[(ii)] $f$ is surjective if and only if $g$ is injective if and only if $f \circ g = 1_{\QQ}$.
\item[(iii)] $f$ is injective if and only if $g$ is surjective if and only if $g\circ f=1_{\P}$.
\item[(iv)] If $g' : \QQ \to \P$ is order-preserving and (a) or (b) holds when $g$ is replaced by $g'$, then $g' = g$.
\item[(v)] If $f' : \P \to \QQ$ is order-preserving and (a) or (b) holds when $f$ is replaced by $f'$, then $f' = f$.
\end{enumerate}
\end{lemma}

\begin{lemma}\label{lem:reflection}
 Let $g:\QQ\to\P$ be an embedding of posets. If $g$ has a lower adjoint $f:\P\to\QQ$, then $f$ is the left-inverse of $g$.
\end{lemma}

\begin{lemma}\label{lem:adjoint}
 Let $f:\P\to\QQ$ be an order morphism with an upper adjoint $g:\QQ\to\P$. Then $f$ preserves all existing joins, whereas $g$ preserves all existing meets. If $\P$ and $\QQ$ are frames, then an order morphism (not necessarily a frame morphism) $f:\P\to\QQ$ has an upper (lower) adjoint $g:\QQ\to\P$ if it preserves all joins (meets).
\end{lemma}


\begin{lemma}\label{lem:isomorphismisadjoint}
 Let $f:\P\to\mathbf{R}$ and $g:\mathbf{R}\to\QQ$ be order morphism with upper adjoints $f_*$ and $g_*$, respectively. Then $g\circ f$ has an upper adjoint $(g\circ f)_*$ equal to $f_*\circ g_*$. Moreover, if $f$ is an order isomorphism, then its upper adjoint $f_*$ is its inverse $f^{-1}$.
\end{lemma}

\begin{lemma}\label{lem:frameadjoint}
 Let $f:F\to G$ be a frame morphism. Then $f$ has an adjoint $g:G\to F$ given by $g(y)=\bigvee\{x\in F:f(x)\leq y\}$.
\end{lemma}

\begin{example}\label{ex:defheyimp}
Let $F$ be a frame. Then for fixed $x\in F$, we define $f_x:F\to F$ by $f_x(z)=x\wedge z$. Then $f_x\left(\bigvee_{i}z_i\right)=x\wedge\left(\bigvee_iz_i\right)=\bigvee_i(x\wedge z_i)=\bigvee_i f_x(z_i)$ by the infinite distributivity law, so $f$ preserves joins. So $f$ has an upper adjoint given by $g_x(y)=\bigvee\{z:x\wedge z\leq y\}$.
\end{example}

\begin{definition}\label{def:heyimplicationinframes}
 The element $g_x(y)$ as defined in the previous example is denoted by $x\to y$ and is called the \emph{relative pseudo-complement}\index{relative pseudo-complement} of $x$ with respect to $y$. We refer to the map $F\times F\to F$ given by $(x,y)\mapsto(x\to y)$ as the \emph{Heyting implication}\index{Heyting implication} or \emph{Heyting operator}\index{Heyting operator}. So we have
\begin{equation}\label{eq:HeyImp}
 x\wedge z\leq y \ \mathrm{if\ and\ only\ if\ }z\leq(x\to y).
\end{equation}
For each $x\in F$ we denote the element $x\to 0$ by $\neg x$, which we call the \emph{negation} of $x$.
\end{definition}

In the next lemma we state some useful identities for the Heyting implication and especially the negation.

\begin{lemma}\cite[Chapter I.8]{M&M}\label{lem:negationidentities}
 Let $F$ be a frame. Then for each $x,y\in F$ we have
\begin{enumerate}
 \item[(i)] $x\wedge(x\to y)=x\wedge y$;
 \item[(ii)] $x\wedge\neg x=0$;
 \item[(iii)] $y\leq\neg x$ if and only if $x\wedge y=0$;
 \item[(iv)] $x\leq y$ implies $\neg y\leq\neg x$;
 \item[(v)] $x\leq\neg\neg x$;
 \item[(vi)] $\neg x=\neg\neg\neg x$;
 \item[(vii)] $\neg\neg(x\wedge y)  =  (\neg\neg x)\wedge(\neg\neg y).$
\end{enumerate}
\end{lemma}

There are several categories which have frames as objects, but different classes of morphisms.
\begin{definition}
 The category $\mathbf{Frm}$ of frames is the category with frames as objects and frame morphisms as morphisms. The category $\mathbf{cHA}$ of \emph{complete Heyting algebras}\index{complete Heyting algebra} is defined as the category with frames as objects and as morphisms exactly the frame morphisms $f:H\to G$ which also satisfy $f(x\to y)=f(x)\to f(y)$ for each $x,y\in H$. Finally, the category $\mathbf{Loc}$ of \emph{locales}\index{locale} is defined as the opposite category of $\mathbf{Frm}$. That is, the objects of $\mathbf{Loc}$ are frames, and $f:L\to K$ is a locale morphism if it is the opposite of a frame morphism $K\to L$. If we want to stress that we work in the $\mathbf{Loc}$, we call the objects locales rather than frames.
\end{definition}

\begin{lemma}\label{lem:equivalentdefinitionsArtinian}
Let $\P$ be a poset. Then the following statements are equivalent:
\begin{enumerate}
\item $\P$ is Artinian;
\item All non-empty downwards directed subsets of $\P$ have a least element.
\item $\P$ satisfies the \emph{descending chain condition}\index{descending chain condition}. That is, if we have a sequence of
elements $x_1\geq x_2\geq \ldots$ in $\P$ (a descending chain), then the sequence stabilizes. That is, there is an
$n\in\N$ such that $x_k=x_{n}$ for all $k>n$.
\end{enumerate}
\end{lemma}
\begin{proof}
Assume $\P$ is Artinian and let $M$ a non-empty downwards directed subset of $\P$. Then $M$ must have a minimal element $x$. Now, if $y\in M$, then there must be an $z\in M$ such that $z\leq x,y$. Since $x$ is minimal, it follows that $z=x$, so $x\leq y$, whence $x$ is the least element of $M$.

Now, assume that every non-empty downwards directed subset of $\P$ has a least element. If $x_1\geq x_2\geq x_3\geq \ldots$ is a descending chain, then $M=\{x_i\}_{i\in\N}$ is clearly a directed subset, so it has a least element, say $x_n$. So we must have $x_k=x_n$ for all $k>n$, hence $\P$ satisfies the descending chain condition.

Finally, we show by contraposition that it follows from the descending chain condition that $\P$ must be Artinian.
So assume that $\P$ does not satisfy the descending chain
condition. Using the Axiom of Dependent Choice, we can construct a sequence $x_1\geq x_2\geq \ldots$ that
does not terminate. The set \mbox{$M=\{x_n:n\in\N\}$} is then a non-empty
subset of $\P$ without a minimal element. Thus $\P$ is not
Artinian.
\end{proof}

Since $\P$ is Noetherian if and only if $\P^\op$ is Artinian, we obtain an equivalent characterization of Noetherian posets.
\begin{lemma}\label{lem:equivalentdefinitionsNoetherian}
Let $\P$ be a poset. Then the following statements are equivalent:
\begin{enumerate}
\item $\P$ is Noetherian;
\item All non-empty upwards directed subsets of $\P$ have a greatest element.
\item $\P$ satisfies the \emph{ascending chain condition}\index{ascending chain condition}. That is, if we have a sequence of
elements $x_1\leq x_2\leq \ldots$ in $\P$ (an ascending chain), then the sequence stabilizes. That is, there is an
$n\in\N$ such that $x_k=x_{n}$ for all $k>n$.
\end{enumerate}
\end{lemma}

\begin{proposition}[Principle of Artinian induction]\index{Artinian induction}\label{prop:ArtinianInduction}
Let $\P$ be an Artinian poset and $\PP$ a property such that:
\begin{enumerate}
\item Induction basis: $\PP(x)$ is true for each minimal $x\in \P$;
\item Induction step: $\PP(y)$ is true for all $y<x$ implies that $\PP(x)$ is true.
\end{enumerate}
Then $\PP(x)$ is true for each $x\in \P$.
\end{proposition}
\begin{proof}
Assume that $M=\{x\in \P:\PP(x)$ is not true$\}$ is non-empty. Since
$\P$ is Artinian, this means that $M$ has a minimal element $x$.
Hence $\PP(y)$ is true for all elements $y<x$, so $\PP(x)$ is true by
the induction step, contradicting the definition of $M$.
\end{proof}

\section{Grothendieck topologies on posets}
\begin{definition}
 Let $\P$ be a poset. Given an element $p\in\P$, a subset $S\subseteq\P$ is called a \emph{sieve} on $p$ if $S\in\D(\down p)$, where $\down p=\{q\in\P: q\leq p\}$. Equivalently, $S$ is a sieve on $p$ if $q\leq p$ for each $q\in S$ and $r\in S$ if $r\leq q$ for some $q\in S$. Then a \emph{Grothendieck topology}\index{Grothendieck topology} $J$ on $\P$ is a map $p\mapsto J(p)$ that assigns to each element $p\in\P$ a collection $J(p)$ of sieves on $p$ such that
\begin{enumerate}
 \item the maximal sieve $\down p$ is an element of $J(p)$;
 \item if $S\in J(p)$ and $q\leq p$, then $S\cap\down q\in J(q)$;
 \item if $S\in J(p)$ and $R$ is any sieve on $p$ such that $R\cap\down q\in J(q)$ for each $q\in S$, then $R\in J(p)$.
\end{enumerate}
The second and third axioms are called the \emph{stability axiom}\index{stability axiom} and the \emph{transitivity axiom}\index{transitivity axiom}, respectively. Given $p\in\P$, we refer to elements of $J(p)$ as $J$-\emph{covers}\index{$J$-cover} of $p$. A pair $(\P,J)$ consisting of a poset $\P$ and a Grothendieck topology $J$ on $\P$ is called a \emph{site}\index{site}. We denote the set of Grothendieck topologies on $\P$ by $\G(\P)$. This set can be ordered pointwisely: for any two Grothendieck topologies $J$ and $K$ on $\P$ we define $J\leq K$ if $J(p)\subseteq K(p)$ for each $p\in\P$.
\end{definition}

Notice that $J$ becomes a functor $\P^\op\to\Sets$ if we define $J(q\leq p)S=S\cap\down q$ for each $S\in J(p)$. Indeed, if $r\leq q\leq p$ and $S\in J(p)$, we have
\begin{equation*}
J(r\leq p)S=S\cap\down r=(S\cap\down q)\cap\down r=J(r\leq q)J(q\leq p)S.
\end{equation*}
Thus the stability axiom exactly expresses that fact that $J\in\Sets^{\P\op}$. In fact, it is enough to require that $S\in J(p)$ for some $p\in\P$ implies $S\cap\down q\in J(q)$ for each $q<p$. Indeed, if $q=p$, then $S\cap\down q=S$, which was already assumed to be in $J(p)$.

\begin{example}
 If $\P=\O(X)$, the set of open subsets of some topological space $X$ ordered by inclusion, the Grothendieck topology corresponding to the usual notion of covering is given by $J(U)=\{S\in\D(\down U):\bigcup S=U\}$.
\end{example}

\begin{example}\label{ex:examplesGT}
Let $\P$ be a poset. Then
\begin{itemize}
 \item The \emph{indiscrete}\index{Grothendieck topology!indiscrete} Grothendieck topology on $\P$ is given by $J_{\mathrm{ind}}(p)=\{\down p\}$.
 \item The \emph{discrete}\index{Grothendieck topology!discrete} Grothendieck topology on $\P$ is given by $J_{\mathrm{dis}}(p)=\D(\down p)$.
\item The \emph{atomic}\index{Grothendieck topology!atomic} Grothendieck topology on $\P$ can only be defined if $\P$ is downwards directed, and is given by $J_{\mathrm{atom}}(p)=\D(\down p)\setminus\{\emptyset\}$.
\end{itemize}
Notice that the existence of (finite) meets is sufficient for a poset to be downwards directed. An example of a poset that is not downwards directed and where the stability axiom for $J_{\mathrm{atom}}$ fails is as follows. Let $\P_3=\{x,y,z\}$ with $y\leq x$ and $z\leq x$. Then $\down y\in J_{\mathrm{atom}}(x)$ and since $z\leq x$, we should have $\down y\cap\down z\in J(x)$. But this means that $\emptyset\in J_{\mathrm{atom}}(x)$, a contradiction. On the other hand, if $\P$ is downwards directed, the stability axiom always holds. Indeed, let $S\in J_{\mathrm{atom}}(x)$ and $z\leq x$. Even with $z\notin S$, we have $S\cap\down z\neq\emptyset$, so $S\cap\down z\in J_{\mathrm{atom}}(z)$, since if we choose an arbitrary $y\in S$, there must be a $w\in\P$ such that $w\leq z$ and $w\leq y$. The latter inequality implies $w\in S$, so $w\in S\cap\down z$.
\end{example}

\begin{convention}
If we want to emphasize the poset on which a Grothendieck topology is defined, we add the poset as a superscript. So $J_\atom^\P$ is the atomic topology defined on the poset $\P$.
\end{convention}

The following lemma will be very useful; for arbitrary categories it can be found in \cite[pp. 110-111]{M&M}. We give a direct proof for posets.
\begin{lemma}\label{lem:filter}
 Let $J$ be a Grothendieck topology on $\P$. Then $J(p)$ is a filter of sieves on $p$ in the sense that:
 \begin{itemize}
  \item $S\in J(p)$ implies $R\in J(p)$ for each sieve $R$ on $p$ containing
  $S$;
  \item $S,R\in J(p)$ implies $S\cap R\in J(p)$.
  \end{itemize}
\end{lemma}
\begin{proof}
Let $S\in J(p)$ and $R\in\D(\down p)$ such that $S\subseteq R$. Then if $q\in S$, we have $q\in R$, so $R\cap\down q=\down q\in J(q)$. It follows now from the transitivity axiom that $R\in J(p)$.

If $S,R\in J(p)$, and let $q\in R$. Then $$(S\cap\down R)\cap\down q=S\cap(R\cap\down q)=S\cap\down q\in J(q)$$ by the stability axiom. So $(S\cap R)\cap\down q\in J(q)$ for each $q\in R$, hence by the transitivity axiom, it follows that $S\cap R\in J(p)$.
\end{proof}

We see that a Grothendieck topology is pointwise closed under finite intersections. In general, a Grothendieck topology is not closed under arbitrary intersections, which leads to the following definition.

\begin{definition}
 Let $J$ be a Grothendieck topology on a poset $\P$. We say that $J$ is \emph{complete}\index{complete Grothendieck topology} if for each $p\in\P$ and each family $\{S_i\}_{i\in I}$ in $J(p)$, for some index set $I$, we have $\bigcap_{i\in I}S_i\in J(p)$.
\end{definition}

\begin{lemma}
 For each $p\in\P$, denote $\bigcap\{S\in J(p)\}$ by $S_p$. Then $J$ is complete if and only if $S_p\in J(p)$ for each $p\in\P$.
\end{lemma}
\begin{proof}
If $J$ is complete, it follows directly from the definitions of $S_p$ and of completeness that $S_p\in J(p)$ for each $p\in\P$.
 Conversely, assume that $S_p\in J(p)$ for each $p\in\P$. For an arbitrary $p\in\P$ let $\{S_i\}_{i\in I}$ a family of sieves in $J(p)$. Then $$S_p=\bigcap\{S\in J(p)\}\subseteq\bigcap_{i\in I}S_i,$$ so by Lemma \ref{lem:filter}, we find that $\bigcap_{i\in I}S_i\in J(p)$.
\end{proof}

\begin{proposition}\label{prop:GTisCompleteLattice}
 Let $\P$ be a poset. Then $\G(\P)$ is a complete lattice where the meet $\bigwedge_{i\in I}J_i$ of a each collection $\{J_i\}_{i\in I}$ of Grothendieck topologies on $\P$ is defined by pointwise intersection: $\left(\bigwedge_{i\in I} J_i\right)(p)=\bigcap_{i\in I}J_i(p)$ for any collection $\{J_i\}_{i\in I}$ of Grothendieck topologies on $\P$.
\end{proposition}
\begin{proof}
 Let $\{J_{i}\}_{i\in I}$ be a collection of Grothendieck topologies on $\P$. We shall prove that $\bigwedge_{i\in I}J_i$ is a Grothendieck topology on $\P$. Let $p\in \P$, then $\down p\in J_i(p)$ for each $i\in I$, so $\down p\in\bigcap_{i\in I}J_i(p)$. For stability, assume $S\in\bigcap_{i\in I}J_i(p)$ and let $q\leq p$. Thus $S\in J_i(p)$ for each $i\in I$, so by the stability axiom for each $J_i$, we find $S\cap\down q\in J_i(q)$ for each $i\in I$. So $S\cap\down q\in\bigcap_{i\in I}J_i(q)$. Finally, let $S\in\bigcap_{i\in I}J_i(p)$ and $R\in\D(\down p)$ be such that $R\cap\down q\in \bigcap_{i\in I}J_i(q)$ for each $q\in S$. Then for each $i\in I$ we have $S\in J_i(p)$ and $R\cap\down q\in J_i(q)$ for each $q\in S$. So for each $i\in I$, by the transitivity axiom for $J_i$ we find that $R\in J_i(p)$. Thus $R\in\bigcap_{i\in I}J_i(p)$. We conclude that $\bigwedge J_i$ is indeed a Grothendieck topology on $\P$. Now, for each $k\in I$, we have $\bigwedge_{i\in I}J_i\leq J_k$, since $\bigcap_{i\in I}J_i(p)\subseteq J_k(p)$ for each $p\in\P$. If $K$ is another Grothendieck topology on $\P$ such that $J_k\leq K$ for each $k\in I$, then for each $p\in\P$ $$\bigcap_{i\in I}J_i(p)\subseteq J_k(p)\subseteq K(p).$$ Hence $\bigwedge_{i\in I}J_i\leq K$. This shows that pointwise intersection indeed defines a meet operation on $\G(\P)$, and by Lemma \ref{lem:completelattice}, it follows that $\G(\P)$ is a complete lattice.
\end{proof}

We will now describe a special class of Grothendieck topologies on a poset that are generated by subsets of the poset.
\begin{proposition}\label{prop:defJX}
Let $\P$ be a poset and $X$ a subset of $\P$. Then
\begin{equation}\label{eq:inducedtopology1}
J_X(p)  = \{S\in\D(\down p): X\cap\down p\subseteq S\}
\end{equation}
is a complete Grothendieck topology on $\P$. Moreover, if $X\subset Y\subseteq\P$, then $J_Y\leq J_X$.
\end{proposition}
\begin{proof}
 We have $\down p\in J_X(p)$, for $\down p$ contains $X\cap\down p$. If
\mbox{$S\in J_X(p)$}, i.e. $X\cap\down p\subseteq S$, and if $q<p$, then
$\down q\subset \down p$, so
\begin{equation*}
X\cap \down q=X\cap \down p\cap \down q \subseteq S\cap
\down q.
\end{equation*}
Hence we find that $S\cap\down q\in J_X(q)$, so
stability holds.

For transitivity, let $S\in J_X(p)$, so $X\cap\down p\subseteq S$. Let
$R$ be a sieve on $p$ such that for each $q\in S$ we have $R\cap \down q\in J_X(q)$, so $X\cap \down q\subseteq R\cap\down q$. Since $S$ is a sieve, we have $S=\down S=\bigcup_{q\in S}\down q$, whence we
find
\begin{equation*}
X\cap S  = X\cap\bigcup_{q\in S}\down q=\bigcup_{q\in S}(X\cap \down q)
\subseteq  \bigcup_{q\in S}(R\cap \down q)=R\cap\bigcup_{q\in
S}\down q=R\cap S,
\end{equation*}
from which $$
X\cap\down p=X\cap (X\cap \down p)\subseteq X\cap S\subseteq
R\cap S\subseteq R$$ follows. Thus $R\in J_X(p)$ and the transitivity
axiom holds.

We next have to show that $J_X$ is complete. So let $p\in\P$ and let $\{S_i\}_{i\in I}\subseteq J_X(p)$ be a collection of covers with index set $I$. This means that $X\cap\down p\subseteq S_i$ for each $i\in I$. But this implies that $X\cap\down p\subseteq\bigcap_{i\in I}S_i$, hence $\bigcap_{i\in I}S_i\in J_X(p)$.

Finally, let $Y\subseteq\P$ such that $X\subseteq Y$. Let $p\in\P$ and $S\in J_Y(p)$. Then $Y\cap\down p\subseteq S$, and since $X\subseteq Y$ this implies $X\cap\down p\subseteq S$. So $S\in J_X(p)$. We conclude that $J_Y\leq J_X$.
\end{proof}

We call $J_X$ the \emph{subset Grothendieck topology}\index{subset Grothendieck topology} generated by the subset $X$ of $\P$. It is easy to see that the indiscrete Grothendieck topology is exactly $J_\P$, whereas the discrete Grothendieck topology is $J_\emptyset$.


\begin{lemma}\label{lem:defXJ}
 Let $J$ be a Grothendieck topology on a poset $\P$ and define
$X_J\subseteq \P$ by
\begin{equation}\label{mJ}
X_J=\big\{p\in \P:J(p)=\{\down p\}\big\}.
\end{equation}
If $K$ is another Grothendieck topology on $\P$ such that $K\leq J$, then $X_J\subseteq X_K$.
\end{lemma}
\begin{proof}
 Let $p\in X_J$. Then $K(p)\subseteq J(p)=\{\down p\}$, and since $\down p\in K(p)$ by definition of a Grothendieck topology, this implies $K(p)=\{\down p\}$. So $p\in X_K$.
\end{proof}

\begin{lemma}\label{lem:sublemma}
 Let $\P$ be a poset and $J$ a Grothendieck topology on $\P$. If $p\in \P$ such that $J(p)=\{\down p\}$, then for each sieve $S$ on $p$ such that $X_J\cap\down p\subseteq S$, we have $S\in J(p)$.
\end{lemma}
\begin{proof}
 By definition of $X_J$ we have $p\in X_J$, so each sieve $S$ on $p$ containing $X_J\cap
\down p$ contains $p$. Since the only sieve on $p$ containing $p$ must be equal to $\down p$, we find that
$X_J\cap\down p\subseteq S$ implies $S\in J(p)$.
\end{proof}

\begin{proposition}\label{prop:correspsubsetsandtopologies}
Let $\P$ be a poset. Then
\begin{enumerate}
\item $Y=X_{J_Y}$ for each $Y\subseteq\P$;
\item $J_Y\leq J_Z$ if and only if $Z\subseteq Y$ for each $Y,Z\subseteq\P$;
\item $K\leq J_{X_K}$ for each Grothendieck topology $K$ on $\P$;
\item If $\P$ is Artinian, then the map $\PP(\P)^\op\to\G(\P)$ given by $X\mapsto J_X$ is an order isomorphism with inverse $J\mapsto X_J$;
\item If $\P$ is Artinian, then every Grothendieck topology on $\P$ is complete.
\end{enumerate}
\end{proposition}
\begin{proof}\
\begin{enumerate}
\item Let $x\in Y$. Then $x\in Y\cap \down x$, so if $S\in
J_Y(x)$, that is $S$ is a sieve on $x$ containing $Y\cap\down x$, we
must have $x\in S$. The only sieve on $x$ containing $x$ is $\down x$,
so we find that $J_Y(x)=\{\down x\}$. By definition of $X_{J_Y}$ we find
that $x\in X_{J_Y}$.

Conversely, let $x\in X_{J_Y}$. This means that $J_Y(x)=\{\down x\}$, or
equivalently, that the only sieve on $x$ containing $Y\cap \down x$ is
$\down x$. Now, assume that $x\notin Y$. Then $x\notin Y\cap\down x$, so
$\down x\setminus\{x\}$ (possibly empty) is a sieve on $x$ which
clearly contains $X\cap \down x$, but which is clearly not equal to
$\down x$ in any case. So we must have $x\in Y$.

\item Let $Y,Z\subseteq\P$. We already found in Proposition \ref{prop:defJX} that $Z\subseteq Y$ implies $J_Y\leq J_Z$. So assume that $J_Y\leq J_Z$. By Lemma \ref{lem:defXJ}, this implies that $X_{J_Z}\subseteq X_{J_Y}$. But by the first statement of this proposition, this is exactly $Z\subseteq Y$.
\item Let $S\in K(p)$ and $q\in X_K\cap\down p$. So $q\leq p$ and $K(q)=\{\down q\}$.
Since $S\in K(p)$ and $q\leq p$, we have $\down q\cap S\in K(q)$ by
stability. In other words $S\cap\down q=\down q$. Thus $\down q\subseteq S$ and
so certainly $q\in S$. We see that $X_K\cap\down p\subseteq S$, hence
$S\in J_{X_K}$.
\item It follows from Proposition \ref{prop:defJX} and Lemma \ref{lem:defXJ} that $X\mapsto J_X$ and $J\mapsto X_J$ are order morphisms.
Since the first statement of this proposition is equivalent with saying that $J\mapsto X_J$ is a left inverse of $X\mapsto J_X$, we only have to show that it is also a right inverse. In other words, we have to show that each
Grothendieck topology $K$ on $\P$ equals $J_{X_K}$. So
let $K$ be a Grothendieck topology on $\P$. We shall
show using Artinian induction (see Lemma \ref{prop:ArtinianInduction}) that for any $p\in \P$, each sieve
$S$ on $p$ containing $X_K\cap\down p$ is an element of $K(p)$.

Let $p$ be a minimal element. Then the only possible sieves on $p$
are $\down p=\{p\}$ and $\emptyset$. Hence there are only two options
for $K(p)$, namely either $K(p)=\{\down p\}$ or
$K(p)=\{\emptyset,\down p\}$. In the first case, we find by Lemma \ref{lem:sublemma} that
$X_K\cap\down p\subseteq S$ implies $S\in K(p)$. If
$K(p)=\{\emptyset,\down p\}$, then $K(p)$ contains all possible sieves on
$p$, so we automatically have that $S\in K(p)$ for any sieve $S$ on $p$ such that
$X_K\cap \down p\subseteq S$.

For the induction step, assume that $p\in \P$ is not minimal and assume that for each $q<p$ the condition $X_K\cap
\down q\subseteq R$ implies $R\in K(q)$ for each sieve $R$ on $q$. If
$K(p)=\{\down p\}$, we again apply Lemma \ref{lem:sublemma} to conclude that
$S\in K(p)$ for all sieves $S$ on $p$ such that $X_K\cap\down p\subseteq S$. If
$K(p)\neq\{\down p\}$, we must have $\down p\setminus\{p\}\in K(p)$, which is non-empty, since $p$ is not minimal. If $S$ is a sieve on $p$ such that $X_K\cap
\down p\subseteq S$, we have for each $q\in \down p\setminus\{p\}$, that is, for each $q<p$, that
\begin{equation}\label{eq:identityXJ}
X_K\cap \down q=X_K\cap \down p\cap \down q\subseteq S\cap\down q.
\end{equation}
Our induction assumption on $q<p$ implies now that $S\cap\down q\in K(q)$ for each $q\in \down p\setminus\{p\}$, hence by the transitivity axiom we find
$S\in K(p)$.

So for all sieves $S$ on $p$ such that $X_K\cap\down p\subseteq S$, we found
that $S\in K(p)$, hence we have $J_{X_K}\subseteq K$. The other inclusion follows from the third statement of this proposition.
\item Since all subset Grothendieck topologies are complete, this follows from the fourth statement of this proposition.
\end{enumerate}
\end{proof}

\begin{theorem}\label{thm:equivalenceArtinianandSubsettopologies}
Let $\P$ be a poset. Then $\P$ is Artinian if and only if all Grothendieck topologies on $\P$ are subset Grothendieck topologies.
\end{theorem}
\begin{proof}
 The previous proposition states that all Grothendieck topologies on $\P$ are subset Grothendieck topologies if $\P$ is Artinian. For the other direction, we first introduce another Grothendieck topology. Let $\P$ be a poset and $X\subseteq\P$, and define $L_X$ for each $p\in\P$ by $$L_X(p)=\{S\in\D(\down p): x\in X\cap\down p\implies S\cap\down x\cap X\neq\emptyset\}.$$ To see that this is a Grothendieck topology on $\P$, assume that $x\in X\cap\down p$, then clearly $\down p\cap\down x\cap X\neq\emptyset$, so $\down p\in L_X(p)$. If $S\in L_X(p)$ and $q\leq p$, assume that $x\in X\cap\down q$. Then $x\in X\cap\down p$, so $S\cap\down x\cap X\neq\emptyset$. Since $x\leq q$, we have $\down x=\down q\cap\down x$, hence $S\cap\down q\cap\down x\cap X\neq\emptyset$. We conclude that $S\cap\down q\in L_X(q)$. Finally, let $S\in L_X(p)$ and $R\in\D(\down p)$ such that $R\cap\down q\in L_X(q)$ for each $q\in S$. Let $x\in X\cap\down p$. Then $S\cap\down x\cap X\neq\emptyset$, so there is some $q\in S\cap\down x\cap X$. Since $q\in S$, we find $R\cap\down q\in L_X(q)$. Since $q\in X$, we find $q\in X\cap\down q$, so $(R\cap\down q)\cap\down q\cap X\neq\emptyset$. Since $q\leq x$, we find $\down q\subseteq\down x$, whence $R\cap\down x\cap X\neq\emptyset$. So $R\in L_X(p)$.

Now assume that $\P$ is non-Artinian. Then $\P$ contains a non-empty subset $X$ without a minimal element. We show that $L_X\neq J_Y$ for each $Y\subseteq\P$. First take $Y=\emptyset$. Then $\emptyset\in J_Y(p)$ for each $p\in\P$. However, since $X$ is assumed to be non-empty, there is some $p\in X$. Then $p\in X\cap\down p$, but $\emptyset\cap\down p\cap X=\emptyset$, so $\emptyset\notin L_X(p)$. We conclude that $L_X\neq J_Y$ if $Y=\emptyset$.

Assume that $Y$ is non-empty. Then there is some $p\in Y$, and $J_Y(p)=\{\down p\}$. Assume that $X\cap\down p=\emptyset$. Then $x\in X\cap\down p\implies \emptyset\cap\down x\cap X\neq\emptyset$ holds, so $\emptyset\in L_X(p)$. Thus $L_X\neq J_Y$ in this case. Assume that $X\cap\down p\neq\emptyset$. Hence there is some $x\in X\cap\down p$. Even if $p\in X$, we can assume that $x$ is strictly smaller than $p$, since $X$ does not contain a minimal element, so $X\cap\down p\setminus\{p\}\neq\emptyset$. Let $S=\down x$, then $S\cap\down x\cap X\neq\emptyset$, so $\down x\in L_X(p)$. We conclude that $L_X(p)\neq\{\down p\}$, so $L_X\neq J_Y$ in all cases.
\end{proof}

The Grothendieck topology $L_X$ in the proof somewhat falls out of the sky. In the next section we explore techniques, which helped us to find this Grothendieck topology.

\section{Non-Artinian and downwards directed posets}

If a poset is non-Artinian, it is much harder to find all Grothendieck topologies. This section is devoted to the question how to find a Grothendieck topology $J$ on a non-Artinian poset such that $J$ is not a subset Grothendieck topology. It turns out that this gives a wide class of new Grothendieck topologies.

\begin{proposition}\label{prop:densetopologydef}
 Let $\P$ be a poset. Then $J_\dense$ defined by $$J_\dense(p)=\{S\in\D(\down p):\down p\subseteq\up S\},$$ for each $p\in\P$ is a Grothendieck topology on $\P$, called the \emph{dense topology}\index{dense topology} \cite[p. 115]{M&M}.

If $\P$ is downwards directed, the dense topology is exactly the atomic topology on $\P$, and is not complete if $\P$ does not contain a least element.
\end{proposition}
\begin{proof}
 Clearly $\down p\subseteq\up(\down p)$, so $\down p\in J_\dense(p)$. Let $S\in J_\dense(p)$ and $q\leq p$. Then $\down p\subseteq\up S$, so if $r\in\down q$, then $r\in\down p$, so there must be an $s\in S$ such that $r\leq s$. Then $s\in S\cap\down q$, since $q\geq r$. We conclude that $\down q\subseteq\up(S\cap\down q)$, so $S\cap\down q\in J_\dense(q)$. Finally, let $S\in J_\dense(p)$ and $R\in\D(\down p)$ such that $R\cap\down q\in J_\dense(q)$ for each $q\in S$. So $\down p\subseteq\up S$ and $\down q\subseteq\up(R\cap\down q)$ for each $q\in S$. Now, let $r\in\down p$. We aim to prove that $r\in\up R$. Since $\down p\subseteq \up S$, there is a $q\in S$ such that $q\leq r$. Moreover, $\down q\subseteq(R\cap\down q)$ holds, so there is a $t\in R\cap\down q)$ such that $t\leq q$. Then $r\in\up R$, for $r\geq q\geq t$. We conlude that $R\in J_\dense(p)$.

If we assume that $\P$ is downwards directed, let $S$ be a non-empty sieve on $p\in\P$. So $S$ contains an $s\leq p$, since $S\in\D(\down p)$. Let $q\in\down p$. Since $\P$ is downwards directed, there is an $r\leq q,s$, which implies that $r\in S$, whence $q\in\up S$. So $S\in J_\dense(p)$. The other way round, if $S\in J_\dense(p)$, we should have $p\in\down p\subseteq\up S$, which is only possible if $S$ is non-empty. So $S\in J_\atom(p)$.

If $\P$ does not contain a least element, $\min(\P)=\emptyset$, since if $x\in\min(\P)$, and $y\in\P$, there is some $z\in\P$ such that $z\leq x,y$ for $\P$ is downwards directed. But $x\in\min(\P)$, so $z\leq x$ implies $z=x$, hence $x\leq y$ contradicting that $\P$ does not contain a least element. Now, let $p\in\P$. Then $\down q\in J_\atom(q)$ for each $q\leq p$. Moreover, if $x\leq p$, then there is a $q<x$, since otherwise $x$ would be minimal. Hence $x\notin\bigcap_{q\leq p}\down q$, so $\bigcap_{q\leq p}\down q=\emptyset\notin J_\atom(p)$. Hence $J_\atom$ is not complete.
\end{proof}

\begin{proposition}\label{prop:conditiondensetopologynotsubsettopology}
 Let $\P$ be a poset. Then $J_\dense=J_{\min(\P)}$ if and only if $\up\min(\P)=\P$. In particular, $J_\dense=J_{\min(\P)}$ if $\P$ is Artinian. If $\up\min(\P)\neq\P$, then $J_\dense$ is not equal to any subset Grothendieck topology.
\end{proposition}
\begin{proof}
 Assume that $\up\min(\P)=\P$ and let $p\in\P$. If $S\in J_{\min(\P)}(p)$, we have by definition $\min(\P)\cap\down p\subseteq S$. Let $q\leq p$. Since $\up\min(\P)=\P$, there is an $m\in\min(\P)$ such that $m\leq q$. Then $$m\in\min(\P)\cap\down q\subseteq\min(\P)\cap\down p.$$ Thus $m\in S$, whence $q\in\up S$. So we find that $\down p\subseteq\up S$, hence $S\in J_\dense(p)$.

Conversely, if $S\in J_\dense(p)$, we have $\down p\subseteq\up S$. Now, $\min(\P)\cap\down p$ is non-empty, since $\up\min(\P)=\P$, so let $m\in\min(\P)\cap\down p$. Then $m\in\up S$, so there is an $s\in S$ such that $s\leq m$. But $m\in\min(\P)$, so $s=m$, and we conclude that $\min(\P)\cap\down p\subseteq S$, hence $S\in J_{\min(\P)}(p)$.

If $\P$ is Artinian, there are two ways to show that $J_\dense=J_{\min(\P)}$. Firstly, if $p\in\P$, the set $\down p$ is non-empty, so by the Artinian property, it contains a minimal element $m$. Now, if $q\leq m$, we have $q\in\down p$, so $m\leq q$ by the minimality of $m$. So $q=m$, which implies that $m\in\min(\P)$. Thus $p\in\up\min(\P)$, so $\up\min(\P)=\P$ and we conclude that $J_\dense=J_{\min(\P)}$.

 Another path we could take is using Proposition \ref{prop:correspsubsetsandtopologies}, which says that $J_\dense$ is equal to $J_Y$, with $$Y=X_{J_{\dense}}=\{p\in\P:J_\dense(p)=\{\down p\}\},$$ if $J_\dense$ is a subset topology, which is always the case if $\P$ is Artinian. Then we find $J_{\dense(p)}=\{\down p\}$ if and only if $\down p\setminus\{p\}\notin J_\dense(p)$ (using Lemma \ref{lem:filter}). Thus $p\in Y$ if and only if $\down p\nsubseteq\up(\down p\setminus\{p\})$ if and only if $\down p\setminus\{p\}=\emptyset$ if and only if $p\in\min(\P)$, so $Y=\min(\P)$. Proposition \ref{prop:correspsubsetsandtopologies} implies $X_{J_Z}=Z$ for each $Z\subseteq\P$. So if $J_\dense=J_Z$ for some $Z$, this implies that $Z=\min(\P)$. Thus $J_\dense\neq J_Z$ for each $Z\subseteq\P$ such that $Z\neq\min(\P)$, hence $J_\dense$ is not equal to any subset topology if $\up\min(\P)\neq\P$.
\end{proof}

The next example shows that there exists a non-Artinian poset $\P$ with $\up\min(\P)=\P$, so with $J_\dense=J_{\min(\P)}$.
\begin{example}
Let $\P=X\cup Y$ with $X=\{x_n\}_{n\in\N}$ and $Y=\{y_n\}_{n\in\N}$ with the ordering $x_n\geq y_n,x_{n+1}$ for each $n\in\N$. We can visualize this in the following diagram

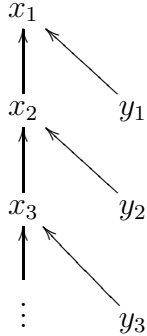
\begin{equation*}
\xymatrix{x_1\\
x_2\ar[u] & y_1\ar[ul]\\
x_3\ar[u] & y_2\ar[ul]\\
\vdots\ar[u]& y_3\ar[ul]\\
}
\end{equation*}
where $a\to b$ represents the inequality $a\leq b$ for each $a,b\in\P$. Then $\P$ is clearly non-Artinian, since it contains the chain $x_1\geq x_2\geq x_3\geq\ldots$, which does not terminate. However, we have $\min(\P)=Y$, whence $\up\min(\P)=\P$.
\end{example}

So the dense topology fails to be a universal counterexample of a Grothendieck topology that is not a subset Grothendieck topology on non-Artinian posets. Given a non-Artinian poset with non-empty subset $X$ without minimal elements, we found in the proof of Theorem \ref{thm:equivalenceArtinianandSubsettopologies}
a Grothendieck topology $L_X$ on $\P$ that is never a subset Grothendieck topology. This Grothendieck topology was found by ``extending'' the dense Grothendieck topology on $X$ (which is on $X$ not a subset Grothendieck topology by Proposition \ref{prop:conditiondensetopologynotsubsettopology}) to a Grothendieck topology on $\P$.

In order to explain in detail how $L_X$ was found, we have to introduce a notion of an extension of a Grothendieck topology to a larger poset. Since a structure on some object is called an extension of a structure on a smaller object if its restriction equals the structure on the smaller object, we also have to define how to restrict Grothendieck topologies to subsets. First we have to introduce some notation in order to discriminate between down-sets of $X$ and down-sets of $\P$.
\begin{definition}\label{def:subposetdownset}
 Let $\P$ be a poset and $X\subseteq\P$ a subset. For any other subset $A\subseteq\P$, we shall use the notation $\bar\down A=\down A\cap X$. In the same way, $\bar\up A$ is defined as the set $\up A\cap X$.
\end{definition}
 Given this definition, if $x\in X$, the set $\down x$ is the set of elements in $\P$ below $x$, whereas $\bar\down x$ is the set of elements of $X$ below $x$. Hence a sieve on $x\in X$ with respect to the subposet $X$ of $\P$ is precisely an element of $\D(\bar\down x)$.

\begin{lemma}\label{lem:inducedtopology}
 Let $J$ be a Grothendieck topology on a poset $\P$ and let $X\subseteq\P$ be a subset such that $J_X\leq J$. Then $\bar J$ defined by $$\bar J(x)=\{S\cap X:S\in J(x)\}$$ for each $x\in X$ is a Grothendieck topology on $X$ called the \emph{induced Grothendieck topology}\index{induced Grothendieck topology}. Furthermore, let $x\in X$. Then for each $\bar S\in\D(\bar\down x)$ we have $\bar S=\down \bar S\cap X$. Moreover, we have $\bar S\in \bar J(x)$ if and only if $\down\bar S\in J(x)$, i.e.,
     $$\bar J(x)=\{\bar S\in\D(\bar\down x):\down\bar S\in J(x)\}$$
     for each $x\in X$.
\end{lemma}
\begin{proof}
Let $x\in X$ and $\bar S\in\D(\bar \down x)$. Then $\bar S\subseteq X$, so $\bar S\subseteq(\down\bar S)\cap X$. On the other hand, if $y\in(\down\bar S)\cap X$, we have $y\in X$ and $y\leq s$ for some $s\in\bar S$. Since $\bar S\in\D(X)$, we obtain $y\in\bar S$. So $(\down\bar S)\cap X=\bar S$. Then if $\down\bar S\in J(x)$, it follows immediately that $\bar S\in \bar J(x)$. Now assume $\bar S\in\bar J(x)$. Then there is an $S\in J(x)$ such that $\bar S=S\cap X$. Let  $y\in S$, so $\down y\subseteq S$, then we have $$\down(X\cap\down y)=\down(X\cap\down y)\cap\down y\subseteq\down(X\cap S)\cap\down y=\down\bar S\cap \down y.$$ Since $J_X\leq J$ and so $\down(X\cap\down y)\in J(y)$, we find by Lemma \ref{lem:filter} that $\down\bar S\cap \down y\in J(y)$ for each $y\in S$. Hence by the transitivity axiom for $J$ it follows that $\down\bar S\in J(x)$.

In order to show that $\bar J$ is a Grothendieck topology on $X$, let $x\in X$. Then $\bar\down x=\down x\cap X$, and since $\down x\in J(x)$, we have $\bar\down x\in \bar J(x)$. If $\bar S\in \bar J(x)$ and $y\leq x$ in $X$, then $\bar S=S\cap X$ for some $S\in J(x)$. Since $J$ is a Grothendieck topology, we have $S\cap\down y\in J(y)$. Now, $\bar S\cap\bar\down y=S\cap X\cap\down y$, so $\bar S\cap\bar\down y\in \bar J(y)$. Finally, let $\bar S\in \bar J(x)$ and $\bar R\in\D(\bar\down x)$ such that $\bar R\cap\bar \down y\in \bar J(y)$ for each $y\in \bar S$. This means that $\down\bar S\in J(x)$ and $\down(\bar R\cap\bar\down y)\in J(y)$ for each $y\in \bar S$. Now, let $R=\down\bar R$. If $y\in\bar S$, we have $\bar R\cap\bar\down y\subseteq\down R\cap\down y$, so $\down(\bar R\cap\bar\down y)\subseteq\down\bar R\cap\down y$. Since $\down(\bar R\cap\bar\down y)\in J(y)$, Lemma \ref{lem:filter} assures that $\down\bar R\cap\down y\in J(y)$ for each $y\in\bar S$. Now assume $z\in\down\bar S$. Then $z\leq y$ for some $y\in\bar S$, and since $\down\bar R\cap\down y\in J(y)$, it follows that $$\down\bar R\cap\down z=(\down\bar R\cap\down y)\cap\down z\in J(z).$$ Again by the transitivity axiom for $J$ it follows that $\down\bar R\in J(x)$, so $\bar R\in\bar J(x)$.
\end{proof}

If we regard induced Grothendieck topologies as restrictings of Grothendieck topologies to subsets, we aim to find a Grothendieck topology $J_K$ on $\P$ such that $\bar J_K=K$ given a Grothendieck topology $K$ on $X$. First we recall Definition \ref{def:subposetdownset}, where, for a fixed subset $X$ of a poset $\P$, for each subset $A$ of $\P$, we denoted the set $X\cap\down A$ by $\bar\down A$. So given $x\in X$, this means that $\down x$ is the set of elements in $\P$ below $x$, whereas $\bar\down x$ means the set of elements of $X$ below $x$. Hence a sieve on $x\in X$ with respect to the subposet $X$ of $\P$ is precisely an element of $\D(\bar\down x)$. Using this notation, we can formulate the following definition.
\begin{definition}\label{def:extensionGrothTop}
Let $K$ be a Grothendieck topology of a subset $X$ of a poset $\P$. Then we define the \emph{extension}\index{extension of a Grothendieck topology} $J_K$ of $K$ on $\P$ by
\begin{equation*}
J_K(p)=\{S\in\D(\down p): S\cap\bar\down x\in K(x)\ \forall x\in\bar\down p\}.
\end{equation*}
\end{definition}

We perform a sanity check that this is a Grothendieck topology. Let $x\in\bar\down p$, so $x\in X$ and $x\leq p$. Then $$\down p\cap\bar\down x=\down p\cap X\cap\down x=X\cap\down x=\bar\down x\in K(x).$$ So $\down p\in J_K(p)$. For stability, let $S\in J_K(p)$, so $S\cap \bar\down x\in K(x)$ for each $x\in\bar\down p$. Let $q\leq p$. If $x\in\bar\down q$, then $x\in\bar\down p$ and $\down x\cap\down q=\down x$, hence $$S\cap\down q\cap\bar\down x=S\cap\down q\cap X\cap\down x=S\cap X\cap\down x=S\cap\bar\down x\in K(x),$$ so $S\cap\down q\in J_K(q)$. Finally, for transitivity, let $S\in J_K(p)$ and $R\in\D(\down p)$ be such that $R\cap\down q\in J_K(q)$ for each $q\in S$. Let $x\in\bar\down p$. Since $S\in J_K$ by definition of $J_K$ we find that $S\cap\bar\down x\in K(x)$. Then for each $y\in S\cap\bar\down x$, we have $y\in S$, so $R\cap\down y\in J_K(y)$. This means that $R\cap\down y\cap\bar\down z\in K(z)$ for each $z\in\bar\down y$. Since also $y\in\bar\down x\subseteq X$, we have $y\in\bar\down y$, so we are allowed to choose $z=y$. Hence we find for each $y\in\bar\down x$ that $$R\cap\bar\down x\cap\bar\down y=R\cap\bar\down y=R\cap\down y\cap\bar\down y \in K(y).$$ Since $\bar\down x\in K(x)$, the transitivity axiom for $K$ implies that $R\cap\bar\down x\in K(x)$. We conclude that $R\in J_K(p)$.

The next proposition assures that $J_K$ is indeed an extension of $K$.

\begin{proposition}\label{prop:extensionofGrothTop}
 Let $X$ be a subset of a poset $\P$ and $K$ a Grothendieck topology on $X$. Then:
 \begin{enumerate}
 \item[(i)] $J_X\leq J_K$, where $J_K$ is the extension of $K$ on $
P$;
 \item[(ii)] $\bar J_K$, defined in Lemma \ref{lem:inducedtopology} for each $x\in X$ by
 \begin{eqnarray*}
 \bar J_K(x) & = &  \{S\cap X:S\in J_K(x)\}\\
 & = & \{\bar S\in\D(\bar\down x):\down\bar S\in J_K(x)\},
 \end{eqnarray*}
is a well-defined Grothendieck topology on $X$;
\item[(iii)] $\bar J_K=K$;
\item[(iv)] The assignment $\G(X)\to\G(\P)$ given by $K\mapsto J_K$ is injective;
\item[(v)] For $Y\subseteq X$, where the subset Grothendieck topology on $X$ induced by $Y\subseteq X$ is denoted by $J_Y^X$, we have $J_{J_Y^X}=J_Y$ and $J^X_Y=\bar J_Y$;
\item[(vi)] If $J_K$ is a subset Grothendieck topology on $\P$, then $K$ is a subset Grothendieck topology on $X$;
\item[(vii)] If $J_K$ is a complete Grothendieck topology on $\P$, then $K$ is a complete Grothendieck topology on $X$;
\item[(viii)] Denote the dense topology on $X$ by $J_\dense^X$, then $L_X$ defined by
     \begin{equation*}
    L_X(p)=\{S\in\D(\down p): x\in X\cap\down p\implies S\cap\down x\cap X\neq\emptyset\}
    \end{equation*}
    for each $p\in\P$ is a well-defined Grothendieck topology on $\P$ such that $J_X\leq L_X$, $J_{J_\dense^X}=L_X$ and $\bar L_X=J^X_\dense$.
\end{enumerate}
\end{proposition}
\begin{proof}\
\begin{enumerate}
\item[(i)]
 Let $p\in\P$ and $S\in J_X(p)$. So $X\cap\down p\subseteq S$. If $x\in X$ and $x\leq p$, then $$\bar\down x=X\cap\down x=(X\cap\down p)\cap (X\cap\down x)\subseteq S\cap\bar\down x,$$ and since $\bar\down x\in K(x)$, it follows that $S\cap \bar\down x\in K(x)$. We conclude that $S\in J_K(p)$, so $J_X\leq J_K$.
\item[(ii)] Since the inequality in (i) holds, Lemma \ref{lem:inducedtopology} garantuees that $\bar J_K$ is a well-defined Grothendieck topology.
\item[(iii)] Let $x\in X$. Since $\down\bar S\in J_K(x)$ if and only $\down\bar S\cap\down y\in K(y)$ for each $y\in\bar\down x$, and $$\down\bar S\cap\bar\down y=\down\bar S\cap X\cap\down y=\bar S\cap\bar\down y,$$ where we again used Lemma \ref{lem:inducedtopology}, we find $$\bar J_K(x)=\{\bar S\in\D(\bar\down x): \bar S\cap\bar\down y\in K(y)\ \forall y\in\bar\down x\}$$ for each $x\in X$. Let $x\in X$. Since $\bar\down x\in K(x)$, it follows immediately by the transitivity axiom for $K$ that $\bar S\in K(x)$ if $\bar S\in\bar J_K(x)$. Conversely, if $\bar S\in K(x)$, it follows immediately by the stability axiom for $K$ that $\bar S\in\bar J_K(x)$. Hence $\bar J_K=K$.
\item[(iv)] Let $K_1,K_2$ be Grothendieck topologies on $X$ such that $J_{K_1}=J_{K_2}$. Then by (iii), $$K_1=\bar J_{K_1}=\bar J_{K_2}=K_2,$$ hence $K\mapsto J_K$ is an injective map.
\item[(v)] By a direct calculation, we show that $J_K=J_Y$, where we take $K=J_Y^X$, which is more explicitly given by
 $$J_Y^X(x)=\{\bar S\in\D(\bar\down x):Y\cap\bar\down x\subseteq\bar S\}$$ for each $x\in X$. Then for each $p\in\P$,
\begin{eqnarray*}
J_K(p) & = & \{S\in\D(\down p): Y\cap\bar\down x\subseteq S\cap\bar\down x\ \forall x\in\bar\down p\}\\
& = & \{S\in\D(\down p): Y\cap\down x\subseteq S\cap X\cap\down x\ \forall x\in X\cap\down p\},
\end{eqnarray*}
where we used $Y\subseteq X$, so $Y\cap\bar\down x=Y\cap\down x$. Let $p\in\P$ and $S\in J_Y(p)$. If $x\in X\cap\down p$, then the stability axiom for $J_Y$ implies that $Y\cap\down x\subseteq S\cap\down x$, and since $Y\subseteq X$, it follows that $Y\cap\down x\subseteq S\cap X\cap\down x$. So $S\in J_K(p)$. Now assume that $S\in J_K(p)$ and let $x\in Y\cap\down p$. Since $Y\subseteq X$, this implies $x\in X\cap\down p$, so $Y\cap\down x\subseteq S\cap X\cap\down x$ for $S\in J_K(p)$. But $x\in Y$, so this implies that $x\in S$. Thus $Y\cap\down p\subseteq S$, whence $S\in J_Y(p)$.

By (iii), $J^X_Y=\bar J_Y$ follows, but we shall also give a direct proof. We have
    \begin{eqnarray*}
    \bar J_Y(x) &=&\{\bar S\in\D(\bar\down x):\down\bar S\in J_Y(x)\}\\
    & = & \{\bar S\in\D(\bar\down x):Y\cap\down x\subseteq\down\bar S\}
    \end{eqnarray*}
    for each $x\in X$. Now let $x\in X$ and assume that $Y\cap\down x\subseteq \down\bar S$. Then we have $$Y\cap\bar\down x=Y\cap\down x\cap X\subseteq\down \bar S\cap X=\bar S,$$
    where we used Lemma \ref{lem:inducedtopology} in the last equality. Hence $\bar J_Y(x)\subseteq J_Y^X(x)$. Conversely, let $\bar S\in J_Y^X(x)$. Since $Y\subseteq X$, we find that $$Y\cap\down x=Y\cap X\cap\down x=Y\cap \bar\down x\subseteq\bar S\subseteq\down\bar S,$$ so $\down\bar S\in J_Y(x)$. Thus $\bar S\in\bar J_Y(x)$.
\item[(vi)]
Assume that $J_K$ is a subset Grothendieck topology on $\P$, so there is a subset $Y$ of $\P$ such that $J_K=J_Y$. By (i), we have $J_X\leq J_Y$. By Proposition \ref{prop:correspsubsetsandtopologies}(ii) this implies $Y\subseteq X$. Thus by (iii) and (v), we find $$K=\bar J_K=\bar J_Y=J_Y^X.$$
So $K$ is a subset Grothendieck topology on $X$.
\item[(vii)] Let $J$ be a Grothendieck topology on $\P$ such that $J_X\leq J$. Then $\bar J$ is well defined. Now, let $x\in X$ and let $\{\bar S_i\}_{i\in I}\subseteq\bar J(x)$ be a collection of covers of $x$. Then for each $i\in I$, there is some $S_i\in J(x)$ such that $\bar S_i=S_i\cap X$. Since $J$ is complete, we have $\bigcap_{i\in I}S_i\in J(x)$, hence $$\bigcap_{i\in I}\bar S_i=\bigcap_{i\in I}(S_i\cap X)=\left(\bigcap_{i\in I}S_i\right)\cap X\in\bar J(x).$$ Now, let $K$ be a Grothendieck topology on $X$ such that $J_K$ is complete. By (i), we have $J_X\leq J_K$, so by (iii) it follows that $K$ is complete.
\item[(viii)]Let $X$ be a subset of $\P$. It was already shown in Theorem \ref{thm:equivalenceArtinianandSubsettopologies} that $L_X$ is a well-defined Grothendieck topology on $\P$.

Another method is to take $K=J^X_\dense$ and to show that $J_K=L_X$. It follows then by (ii) that $L_X$ is well-defined. Moreover, by (i) and (iii) it follows that $J_X\leq L_X$ and $\bar L_X=J^X_\dense$, respectively.
First notice that $S\in J_\dense(p)$ if and only if $S\cap\down y\neq\emptyset$ for each $y\in\down p$. Similarly, we have $\bar S\in J_\dense^X(x)$ if and only if $\bar S\cap\bar\down y\neq\emptyset$ for each $y\in\bar\down x$. Now take $K=J_\dense^X$. Then we find for each $p\in\P$:
\begin{eqnarray*}
 J_K(p) & = & \{S\in\D(\down p):\forall x\in\bar\down p(S\cap\bar\down x\in K(x))\}\\
& = & \{S\in\D(\down p):\forall x\in X\cap\down p(S\cap X\cap\down x\in J^X_\dense(x))\}\\
& = & \{S\in\D(\down p):\forall x\in X\cap\down p(\forall y\in\bar\down x(S\cap X\cap\down x\cap\bar\down y\neq\emptyset))\}\\
& = & \{S\in\D(\down p):\forall x\in X\cap\down p(\forall y\in X\cap\down x(S\cap X\cap\down x\cap X\cap\down y\neq\emptyset))\}\\
& = & \{S\in\D(\down p):\forall x\in X\cap\down p(\forall y\in X\cap\down x(S\cap X\cap\down y\neq\emptyset))\}\\
& = & \{S\in\D(\down p):\forall x\in X\cap\down p(S\cap X\cap\down x\neq\emptyset)\}\\
& = & L_X(p).
\end{eqnarray*}
\end{enumerate}
\end{proof}

So we see that $L_X$ is indeed the extension of the dense Grothendieck topology on $X$. Moreover, (vi) of the Proposition shows that $L_X$ is not a subset Grothendieck topology if $J_dense^X$ is not a subset Grothendieck topology.

\section{Sheaves and morphisms of sites}
In this section we want to explore several notions of the equivalence of sites. A possibility is to say that $(\P,J)$ and $(\QQ,K)$ are equivalent if there exists an equivalence between their categories of sites. However, it might be more natural to explore notions of morphisms of sites, so we postpone the discussion of equivalence of sites in terms of sheaves.

\begin{definition}\cite[Section VII.10]{M&M}
 Let $(\P,J)$ and $(\QQ,K)$ be sites. An order morphism $\phi:\P\to\QQ$ \emph{preserves covers}\index{cover preserving order morphism}\index{clp} if $\down\phi[S]\in K(\phi(p))$ for each $p\in\P$ and each $S\in J(p)$. An order morphism $\pi:\QQ\to\P$ has the \emph{covering lifting property}\index{covering lifting property} (abbreviated by ``clp'') if for each $q\in\QQ$ and each $S\in J(\pi(q))$ there is an $R\in K(q)$ such that $\pi[R]\subseteq S$.
\end{definition}

\begin{proposition}
 Let $\phi:(\P_1,J_1)\to(\P_2,J_2)$ and $\pi:(\P_2,J_2)\to(\P_3,J_3)$ be order morphisms. Then $\pi\circ\phi$ has the clp if $\pi$ and $\phi$ have the clp. Moreover, if $\pi$ and $\phi$ preserve covers, then $\pi\circ\phi$ preserves covers.
\end{proposition}
\begin{proof}
 Let $p\in\P_1$ and $S_3\in J_3(\pi\circ\phi(p))$. Since $\pi$ has the clp, there must be an $S_2\in J_2(\phi(p))$ such that $\pi[S_2]\subseteq S_3$. Since $\phi$ has the clp, there must be an $S_1\in J_1(p)$ such that $\phi[S_1]\subseteq S_2$, hence $\pi\circ\phi[S_1]\subseteq\pi[S_2]$. Combining both inclusions, we obtain $\pi\circ\phi[S_1]\subseteq S_3$, hence $\pi\circ\phi$ must have the clp.

Now assume that $\pi$ and $\phi$ preserve covers and let $p\in\P$ and $S\in J_1(p)$. Then $\down\phi[S]\in J_2(\phi(p))$, since $\phi$ preserves covers, hence $\down\pi[\down\phi[S]]\in J_3(\pi\circ\phi(p))$ for $\pi$ preserves covers. However, in order to show that $\pi\circ\phi$ preserves covers, we have to show that $\down\pi\circ\phi[S]\in J_3(\pi\circ\phi(p))$. Let $x\in\down\pi[\down\phi[S]]$. Then there is a $y\in\down\phi[S]$ such that $x\leq\pi(y)$. Moreover, there must be an $s\in S$ such that $y\leq\phi(s)$. Since $\pi$ is an order morphism, we find $\pi(y)\leq\pi\circ\phi(s)$, so $x\leq\pi\circ\phi(s)$. We conclude that $\down\pi[\down\phi[S]]\subseteq\down\pi\circ\phi[S]$, hence by Lemma \ref{lem:filter} it follows that $\down\pi\circ\phi[S]\in J_3(\pi\circ\phi(p))$.
\end{proof}


In order to define a correct and suitable notion of a morphism of sites, it seems like we have to choose between the cover preserving property and the clp. However, both notions are related to each other as the following lemma shows. 
\begin{lemma}\cite[Lemma VII.10.3]{M&M}
 Let $(\P,J)$ and $(\QQ,K)$ be sites and $\phi:\P\to\QQ$ the upper adjoint of $\pi:\QQ\to\P$. Then $\phi$ preserves covers if and only if $\pi$ has the clp.
\end{lemma}
\begin{proof}
 Assume that $\phi$ preserves covers. Let $q\in\QQ$ and $S\in J(\pi(q))$. Let $p=\pi(q)$, so $S\in J(p)$. Then $\pi(q)\leq p$, so using the adjunction, we obtain $q\leq\phi(p)$. Since $\phi$ preserves covers, we have $\down\phi[S]\in K(\phi(p)$, and since $q\leq\phi(p)$, it follows by stability of $K$ that $R=\down\phi[S]\cap\down q\in K(q)$. Let $x\in R$, then $\pi(x)\in\pi[R]$. Moreover, we have $R\subseteq\down\phi[S]$, so there is a $y\in S$ such that $x\leq\phi(y)$. By the adjunction, we obtain $\pi(x)\leq y$ and since $S$ is a down-set, we find $\pi(x)\in S$. Thus $\pi[R]\subseteq S$, which shows that $\pi$ has the clp.

Conversely, assume that $\pi$ has the clp and let $p\in\P$ and $S\in J(p)$. Let $q=\phi(p)$, then $q\leq\phi(p)$, so using the adjunction, we find $\pi(q)\leq p$. By stability of $J$, it follows that $S\cap\down\pi(q)\in J(\pi(q)$, and since $\pi$ has the clp, we obtain an $R\in K(q)$ such that $\pi[R]\subseteq S\cap\down\pi(q)$. Let $x\in R$. Then $s=\pi(x)\in S$, so $\pi(x)\leq s$. Using the adjunction we obtain $x\leq\phi(s)$, whence $x\in\down\phi[S]$, so $R\subseteq\down\phi[S]$. By Lemma \ref{lem:filter} it follows that $\down\phi[S]\in J(q)$ if we can show that $\down\phi[S]\in\D(\down q)$. Thus we have to show that $q$ is an upper bound of $\down\phi[S]$. Let $y\in\down\psi[S]$. Then there is an $s\in S$ such that $y\leq\phi(s)$. Using the adjunction, we find $\pi(y)\leq s\leq p$, so using the adjunction again yields $y\leq\phi(p)=q$. We conclude that $\down\phi[S]\in K(\phi(p))$, hence $\phi$ preserves covers.
\end{proof}

If we assume that we have a definition of site morphism, $\phi:(\P,J)\to(\QQ,K)$ is a site isomorphism if it has an inverse $\pi:(\QQ,K)\to(\P,J)$. On the poset part of the site, this means that $\phi$ and $\pi$ are order isomorphisms, which are each others inverses. Since this implies that $\phi$ is both the upper adjoint and lower adjoint of $\pi$, the preceding lemma implies that it does not matter whether we define an order morphism $\phi:(\P,J)\to(\QQ,K)$ to be a site morphism if it preserves covers of if it has the clp. Independent of which choice we make, we obtain the same notion of site isomorphisms, which is in the end the notion in which we are interested.

\begin{definition}
 Let $(\P,J)$ and $(\QQ,K)$ be sites. Then an order isomorphism $\phi:\P\to\QQ$ is called an \emph{isomorphism of sites}\index{isomorphism} if it satisfies one of the following equivalent conditions:
\begin{enumerate}
 \item $\phi$ preserves covers and has the clp;
 \item $\phi$ and $\phi^{-1}$ both preserve covers;
 \item $\phi$ and $\phi^{-1}$ both have the clp.
\end{enumerate}
\end{definition}

\begin{lemma}\label{lem:orderisopreservessieves}
 Let $\phi:\P\to\QQ$ be an order isomorphism. Let $p\in\P$. Then $S\in\D(\down p)$ implies $\phi[S]\in\D(\down\phi(p))$. Moreover, $\down\phi(p)=\phi[\down p]$.
\end{lemma}
\begin{proof}
 For each $s\in S\in\D(\down p)$ we have $s\leq p$, so $\phi(s)\leq\phi(p)$. If $x\in\phi[S]$ and $y\leq x$, then $\phi^{-1}(x)\in S$ and $\phi^{-1}(y)\leq\phi^{-1}(x)$, so $\phi^{-1}(y)\in S$. Hence $y\in\phi[S]$, and we conclude that $\phi[S]\in\D(\down\phi(p))$.

Since $\down p\in\D(\down p)$, we immediately obtain $\phi[\down p]\in\D(\down\phi(p))$, hence $\phi[\down p]=\down\phi[\down p]$. Now $q\in\down\phi[\down p]$ if and only if $q\leq\phi(p')$ for some $p'\leq p$ if and only if $q\leq\phi(p)$, so $\down\phi[\down p]=\down\phi(p)$.
\end{proof}

\begin{proposition}
 Let $(\P,J)$ and $(\QQ,K)$ be sites. Then an order isomorphism $\phi:\P\to\QQ$ is an isomorphism of sites $(\P,J)\to(\QQ,K)$ if and only if for each $p\in\P$ and $S\in\PP(\P)$, we have $S\in J(p)$ if and only if $\phi[S]\in K(\phi(p))$.
\end{proposition}
\begin{proof}
 Assume that $\phi$ is an isomorphism of sites and let $p\in\P$ and $S\in J(p)$. Since $\phi$ preserves covers, we have $\down\phi[S]\in K(\phi(p))$. By the preceding lemma, we have $\phi[S]\in\D(\down\phi(p))$, so $\down\phi[S]=\phi[S]$. Hence $\phi[S]\in K(\phi(p))$. Conversely, if $S\in\PP(\P)$ such that $\phi[S]\in K(\phi(p))$, then the same argument for $\phi^{-1}$ instead of $\phi$ yields $\phi^{-1}[\phi[S]]\in J(\phi^{-1}\circ\phi(p))$, so $S\in J(p)$.

Now assume that for each $p\in\P$, we have $S\in J(p)$ if and only if $\phi[S]\in K(\phi(p))$. Let $p\in\P$ and $S\in J(p)$. Then $\phi[S]\in K(\phi(p))$, so certainly $\down\phi[S]\in K(\phi(p))$. Let $q\in\QQ$ and $R\in K(q)$. By Lemma \ref{lem:orderisopreservessieves} we find $\phi^{-1}[R]\in\D(\down\phi^{-1}(q))$. Since $\phi[\phi^{-1}[R]]=R$, we find that $\phi[\phi^{-1}[R]]\in K(\phi\circ\phi^{-1}(q))$, so $\phi^{-1}[R]\in J(\phi^{-1}(q))$. We conclude that both $\phi$ and $\phi^{-1}$ preserve covers, so $\phi$ is an isomorphism of sites.
\end{proof}

\begin{corollary}
 Let $\P$ and $\QQ$ posets, and $X\subseteq\P$ and $Y\subseteq\QQ$ subsets. Then an order isomorphism $\phi:\P\to\QQ$ is an isomorphism of sites $(\P,J_X)\to(\QQ,J_Y)$ if and only if $\phi[X]=Y$.
\end{corollary}
\begin{proof}
 Assume that $\phi:(\P,J_X)\to(\QQ,J_Y)$ is an isomorphism of sites and let $x\in X$. Then $J(x)=\{\down x\}$, so $J_Y(\phi(x))=\{\phi[\down x]\}$ for $\phi$ is an isomorphism of sites. Now $J_Y(\phi(x))$ contains only one element if and only if $\phi(x)\in Y$, hence $\phi[X]\subseteq Y$. Replacing $\phi$ by $\phi^{-1}$ gives $\phi^{-1}[Y]\subseteq X$, hence $\phi[X]=Y$.

Conversely assume that $\phi[X]=Y$. Let $p\in\P$ and $S\in J_X(p)$. Then $X\cap\down p\subseteq S$, hence $\phi[X]\cap\phi[\down p]\subseteq\phi[S]$. By Lemma \ref{lem:orderisopreservessieves}, $\phi[S]\in\D(\down\phi(p))$ and $\phi[\down p]=\down\phi(p)$. Moreover, $\phi[X]=Y$, hence $Y\cap\down\phi(p)\subseteq\phi[S]$, and we conclude that $\phi[S]\in J_Y(\phi(p))$. Since $\phi$ is an order isomorphism, we have $\phi^{-1}[Y]=X$. Hence applying the same arguments to $\phi^{-1}$ gives the implication $\phi[S]\in J_Y(\phi(p))$ implies $S\in J_X(p)$. We conclude that $\phi$ is an isomorphism of sites.
\end{proof}

If $(\P,J_X)$ and $(\QQ,J_Y)$ are sites, it might be interesting as well to examine how to express the cover preserving property and the clp of an order morphism $\phi:\P\to\QQ$ in terms of $\phi$, $X$ and $Y$.

\begin{proposition}
 Let $\P$ and $\QQ$ posets and $X$ and $Y$ subsets of $\P$ and $\QQ$, respectively. Let $\phi:\P\to\QQ$ an order morphism. Then $\phi:(\P,J_X)\to(\QQ,J_Y)$ has the clp if and only if $\phi[X]\subseteq Y$.
\end{proposition}
\begin{proof}
Assume that $\phi$ has the clp and let $x\in X$. Since $\phi$ has the clp and $\down(Y\cap\down\phi(x))\in J_Y(\phi(x))$, there must be an $R\in J_X(x)$ such that $\phi[R]\subseteq\down(Y\cap\down\phi(x))$. Since $x\in X$, it follows that $J_X(x)$ contains only $\down x$, hence $\phi[\down x]\subseteq\down(Y\cap\down\phi(x))$. In particular, we must have $\phi(x)\in\down (Y\cap\down\phi(x))$. Now assume that $\phi(x)\notin Y$. Then for each $y\in Y\cap\down\phi(x)$ we have $y<\phi(x)$. If $z\in\down(Y\cap\down\phi(x))$, we must have $z\leq y$ for some $y\in Y\cap\down\phi(x)$, so $z<\phi(x)$ for each $z\in \down(Y\cap\down\phi(x)$. Thus the choice $z=\phi(x)$ gives a contradiction, so we must have $\phi(x)\in Y$. We conclude that $\phi[X]\subseteq Y$.

Assume that $\phi[X]\subseteq Y$ and let $p\in\P$ and $S\in J_Y(\phi(p))$. Then $Y\cap\down\phi(p)\subseteq S$, so $\phi[X]\cap\down \phi(p)\subseteq S$. Moreover, since $S$ is a down-set, we obtain $\down(\phi[X]\cap\down\phi(p))\subseteq S$. Let $R=\down (X\cap\down p)$, then $R\in J_X(p)$. Let $y\in R$. Then there is an $x\in X$ such that $y\leq x\leq p$. Since $\phi$ is an order morphism, we find $\phi(y)\leq\phi(x)\leq\phi(p)$. So $\phi(x)\in \phi[X]\cap\down\phi(p)$, whence $\phi(y)\in\down(\phi[X]\cap\down\phi(p))$. Thus we find that $\phi(y)\in S$, hence $\phi[R]\subseteq S$, and we conclude that $\phi$ has the clp.
\end{proof}

For $\phi:(\P,J_X)\to(\QQ,J_Y)$ preserving covers, we can state a similar statement, although we have add bijectivity of $\phi$ as extra condition.

\begin{proposition}
Let $\P$ and $\QQ$ posets and $X$ and $Y$ subsets of $\P$ and $\QQ$, respectively. Let $\phi:\P\to\QQ$ an order isomorphism. Then $\phi:(\P,J_X)\to(\QQ,J_Y)$ preserves covers if and only if $Y\subseteq\phi[X]$.
\end{proposition}
\begin{proof}
 Assume that $\phi$ preserves covers. Then $\down\phi[S]\in J_Y(\phi(p))$ for each $p\in\P$ and $S\in J_X(p)$. Notice that Lemma \ref{lem:orderisopreservessieves} assures that $\down\phi[S]=\phi[S]$ for $\phi$ is assumed to be an order isomorphism. Since $\down(X\cap\down p)$ is the least element of $J_X(p)$, Lemma \ref{lem:filter} assures that $\phi[S]\in J_Y(\phi(p))$ for each $S\in J_X(p)$ if and only if $\phi[\down(X\cap\down p))]\in J_Y(\phi(p))$, we find that $\phi$ preserves covers if and only if for
\begin{equation}\label{eq:coveringpreservingforsubsettopologies}
Y\cap\down\phi(p)\subseteq\phi[\down(X\cap\down p)]
\end{equation} for each $p\in\P$.

Let $y\in Y$. By the surjectivity of $\phi$ there is some $p\in\P$ such that $\phi(p)=y$. Since $\phi$ preserves covers, we find that $y=\phi(p)\in\phi[\down(X\cap\down p)]$. Assume that $p\notin X$. Then for each $x\in X\cap\down p$, we must have $x<p$. Then if $p'\in\down(X\cap\down p)$, there must be an $x\in X\cap\down p$ such that $p'\leq x$, hence $p'<p$. Then $\phi(p')<\phi(p)$ for $\phi$ is an order isomorphism, so $z<\phi(p)$ for each $z\in\phi[\down(X\cap\down p)$. Since the choice $z=y$ gives a contradiction, we must have $\phi(p)\in X$. We conclude that $Y\subseteq\phi[X]$.

Conversely, assume that $Y\subseteq\phi[X]$ and let $p\in\P$. We have to show that (\ref{eq:coveringpreservingforsubsettopologies}) holds, so let $y\in Y\cap\down\phi(p)$. Since $Y\subseteq\phi[X]$, there is some $x\in X$ such that $\phi(x)=y$. Hence $\phi(x)\leq\phi(p)$, whence $x\leq p$ for $\phi$ is an order isomorphism. Hence $y=\phi(x)$ with $x\in X\cap\down p$, which shows that (\ref{eq:coveringpreservingforsubsettopologies}) indeed holds.
\end{proof}


If $\P$ is a poset and $X\subseteq\P$ a subset, it might be interesting to find a description of all $J_X$-sheaves.
\begin{definition}\cite[Chapter III.4]{M&M}
Let $\P$ be a poset and $J$ a Grothendieck topology on $\P$.
Furthermore, let $F:\P^\op\to\Sets$ be a functor (a contravariant functor to $\Sets$ is also called a \emph{presheaf}\index{presheaf}). Let $p\in\P$ and $S\in J(p)$. Then a family $\langle a_x\rangle_{x\in S}\in\prod_{x\in S}F(x)$ is called a \emph{matching family}\index{matching family} for the cover $S$ with elements of $F$ if \begin{equation}\label{matchfam}
 F(y\leq x)a_x=a_y
\end{equation}
for each $x,y\in S$ such that $y\leq x$. An element $a\in F(p)$ such that $F(x\leq p)a=a_x$ for each $x\in S$ is called an \emph{amalgamation}\index{amalgamation}. We say that $F$ is a $J$-\emph{sheaf}\index{sheaf} if for each $p\in\P$ for each $S\in J(p)$ and for each matching family $\langle a_x\rangle_{x\in S}$ there is a unique
\emph{amalgamation}\index{amalgamation} $a\in F(p)$.
\end{definition}

\begin{example}\label{ex:indiscretetopologysheaves}
 Let $J$ be a topology on a poset $\P$. If we consider the indiscrete topology given by $J_{\mathrm{ind}}(p)=\{\down p\}$ on $\P$, we have $\Sh(\P,J_{\mathrm{ind}})=\Sets^{\P^\op}$. Indeed, let $F\in\Sets^{\P^\op}$ and $p\in\P$. Then there is only one cover of $p$, namely $\down p$. Hence if $\langle a_q\rangle_{q\in\down p}$ is a matching family for $\down p$, we have $p\in\down p$, so $a_q=F(q\leq p)a_p$. Thus $a_p$ is an amalgamation of the matching family. It is also unique, since if $a\in F(p)$ is another amalgamation, we have $a_p=F(p\leq p)a=a$.
\end{example}

\begin{example}\label{ex:discretetopologysheaves}
 Let $J$ be a topology on a poset $\P$. Given a $J$-sheaf $F$ and a point $p\in\P$ such that $\emptyset\in J(p)$. Then a matching family for $\emptyset$ must be a function on the empty set. There exists only one function with the empty set as domain, and it is clearly a matching family. An amalgamation must be a point in $F(p)$, but since there are no restrictions on this point except that it must be unique, it follows that $F(p)$ is a singleton set $1$. In particular, if $J$ is the discrete topology, then $\Sh(\P,J)=1$, since $\emptyset\in J(p)$ for each $p\in\P$.
\end{example}

\begin{example}\label{ex:atomsheaves}
 Let $\P$ be a poset and let $J$ be a topology on $\P$. Let $F$ be a $J$-sheaf on $\P$ and let $q\leq p$ be an element such that $\down q\in J(p)$. Notice that $F(q)$ is non-empty if $F(p)$ is non-empty, since we have a map $F(q\leq p):F(p)\to F(q)$. Then given any $b\in F(q)$, we can define $a_r:=F(r\leq q)b$ for each $r\in\down q$, which yields a matching family for $\down q\in J(p)$. Since $F$ is a sheaf, we find that there is a unique $a\in F(p)$ such that $F(r\leq p)a=a_r$. This shows not only that $F(p)$ must be non-empty if $F(q)$ is non-empty, but also that $a$ is the unique element such that $F(q\leq p)a=b$, so $F(q\leq p)$ is a bijection.

Now assume that $\P$ is downwards directed consider the atomic topology $J_\atom$. Then $\down q\in J_\atom(p)$ for each $q\leq p$. Let $p_1,p_2\in\P$. Since $\P$ is directed, we find that there is a $q\in\P$ such that $p_1,p_2\geq q$. From the preceding, we find that if $F$ is a $J_\atom$-sheaf, $F(q\leq p_1)$ and $F(q\leq p_2)$ are bijections, so $F(q\leq p_2)^{-1}\circ F(q\leq p_1)$ is a bijection from $F(p_1)$ to $F(p_2)$. Thus we find that all sheaves on $\P$ are constant up to isomorphism. Hence, $\Sh(\P,J_\atom)\cong\Sets$.
\end{example}

The indiscrete topology is the coarsest topology, but from the first example, we see that all presheaves are sheaves with respect to this topology. So $\Sh(\P,J_{\mathrm{ind}})=\Sets^{\P^\op}$. On the other hand, from the second example we see that only one presheaf is a sheaf with respect to the discrete topology, which is the finest topology possible. So $\Sh(\P,J_{\mathrm{dis}})\cong 1\cong\Sets^{\emptyset^\op}$ if we consider the empty set as a subposet of $\P$.
In the third example, we consider the atomic topology, which lies between the indiscrete and the discrete topology if we order Grothendieck topologies from coarse to fine. Recall that the atomic topology is only defined for downwards directed posets, but if $\P$ is Artinian, this is certainly the case. Then the atomic topology is induced by the least element $0$, and we see that $\Sh(\P,J_\atom)\cong\Sets\cong\Sets^{\{0\}^\op}$. In all these cases, the topology in question is induced by some subset $X$ of $\P$, and the category of sheaves with respect to this topology is equivalent to $\Sets^{X^\op}$.

It turns out to be a direct consequence of the so-called \emph{Comparison Lemma}\index{Comparison Lemma} that this is true for each subset $X$ of $\P$. The idea behind the Comparison Lemma is the following. Given a subset $X$ of a poset $\P$, the inclusion map $I:X\to\P$ induces a map $I^*:\Sets^{\P^\op}\to\Sets^{X^\op}$ given by $F\mapsto F\circ I$. If $J$ is a Grothendieck topology on $\P$, one could ask about the image of $\Sh(\P,J)$ under $I^*$. The Comparison Lemma gives a sufficient condition on $X$ for the existence of a Grothendieck topology $\bar J$ on $X$ such that $I^*$ provides an equivalence $\Sh(X,\bar J)\cong\Sh(\P,J)$, where we consider $X$ a subposet of $\P$, so if $x,y\in X$ and $x\leq y$ in $\P$, then $x\leq y$ in $X$. In categorical terms, this says that $X$ is a full subcategory of $\P$. We shall state the Comparison Lemma restricted to the posetal case as Theorem \ref{thm:ComparisonLemma} below. The proof we present is based on the proof of \cite[Theorem C.2.2.3]{Elephant2}, where the Comparison Lemma is stated for arbitrary categories, although we do not explicitely make use of Kan extensions. Another proof of the Comparison Lemma for arbitrary categories can be found in \cite{K&M}. In these references, the Lemma is also stated in a sharper version, i.e., the fullness condition is dropped. The condition that we will put on $X$ is called $J$-\emph{denseness}\index{$J$-dense}, which says that for each $p\in\P$ one should have $\down(X\cap\down p)\in J(p)$. Since $\down(X\cap\down p)$ is the smallest cover in $J_X(p)$, we find by Lemma \ref{lem:filter} that $X$ is $J$-dense if and only if $J_X\subseteq J$.

Let $X\subseteq\P$ and $I:X\embeds\P$ be the inclusion. Then we write $\bar F=I^*(F)=F\circ I$, if we want to stress when we restrict the domain of $F$ to $X$.
\begin{lemma}
 Let $J$ be a Grothendieck topology on a poset $\P$ and let $X\subseteq\P$ be a subset such that $J_X\leq J$. For each $x\in X$, let $\bar J$ be the Grothendieck topology on $X$ defined in Lemma \ref{lem:inducedtopology}. Then the functor $I^*:\Sets^{\P^\op}\to\Sets^{X^\op}$ restricts to a functor $\Sh(\P,J)\to\Sh(X,\bar J)$, which we will denote by $I^*$ as well.
\end{lemma}
\begin{proof}
  Let $F\in\Sh(\P,J)$ and $\bar F=F\circ I$ with $I:X\embeds\P$ the inclusion. Let $x\in X$ and $\bar S\in\bar J(x)$. Let $\langle a_y\rangle_{y\in\bar S}$ be a matching family for $\bar S$ of elements of $\bar F$. By Lemma \ref{lem:inducedtopology}, $S=\down\bar S\in J(x)$. Thus we are going to extend the matching family $\langle a_y\rangle_{y\in\bar S}$ to a matching family for $S$ of elements of $F$, as follows. If $z\in S$, there must be at least one $y\in\bar S$ such that $z\leq y$, hence we define $a_z=F(z\leq y)a_y$. Now, if there is another $y'\in\bar S$ such that $z\leq y'$, we should have $a_z=F(z\leq y')a_{y'}$. Call the right-hand side $b_z$. Since $$R=\down(X\cap\down z)\in J(z),$$ we have matching families $\langle F(w\leq z)a_z\rangle_{w\in R}$ and $\langle F(w\leq z)b_z\rangle_{w\in R}$ for $R$ of elements of $F$. Now, the matching families must be equal, since if $u\in\down(X\cap\down z)$, there is a $v\in X\cap\down z$ such that $u\leq v\leq z$. Notice that also $v\in\bar S$, since $v\in X\cap z$, and $z\leq y$ with $y\in\bar S\in\D(X)$. Then $F(v\leq y)a_y=F(v\leq y')a_{y'}$, since $v,y,y'\in\bar S$. So indeed
\begin{eqnarray*}
 F(u\leq z)a_z & = & F(u\leq v)F(v\leq z)a_z=F(u\leq v)F(v\leq z)F(z\leq y)a_y\\
 & = & F(u\leq v)F(v\leq y)a_y=F(u\leq v)F(v\leq y')a_{y'}\\
& = & F(u\leq v)F(v\leq z)F(z\leq y')a_{y'}=F(u\leq z)b_z.
\end{eqnarray*}
Hence both matching families must have the same amalgamation, for $F$ is a $J$-sheaf, but since both $a_z$ and $b_z$ are amalgamations, we find $a_z=b_z$. Hence $\langle a_z\rangle_{z\in\down\bar S}$ is a matching family for $\down\bar S\in J(x)$, so again since $F$ is a $J$-sheaf, it has a unique amalgamation $a\in F(x)$ such that $F(z\leq x)a=a_z$ for each $z\in\down\bar S$, so in particular for $z\in\bar S$. We conclude that $a$ is an amalgamation of $\langle a_y\rangle_{y\in\bar S}$. Now assume $b$ is another amalgamation. Then $a_y=F(y\leq x)b$ for each $y\in\bar S$, and if $z\in\down\bar S$, we have $$F(z\leq x)b=F(z\leq y)F(y\leq x)b=a_y$$ for some $y\in\bar S$ such that $z\leq y$. So $b$ is also an amalgamation for $\langle a_z\rangle_{z\in\down\bar S}$, whence $b=a$.
\end{proof}

The functor $E:\Sets^{X^\op}\to\Sets^{\P^\op}$ in the opposite direction is defined on objects by sending $F\mapsto\hat F$, where for each $p\in\P$, $\hat F(p)$ is defined as the collection of all functions $f:X\cap\down p\to\bigcup_{x\in X}F(x)$ such that $f(x)\in F(x)$ for each $x\in X\cap\down p$, and
\begin{equation}\label{eq:elFhat}
F(y\leq x)f(x)=f(y)
\end{equation}
 for each $x,y\in X\cap\down p$ such that $y\leq x$. The action of $\hat F$ on morphisms $q\leq p$ in
$\P$ is defined by
\begin{equation}\label{eq:hatFmor}
\hat F(q\leq p)(f)=\left.f\right|_{X\cap\down q},
\end{equation}
where $f\in \hat F(p)$. It is immediate that $\hat F(q\leq
p)f\in\hat F(q)$ and that $\hat F$ is a functor.

If $\alpha:F\to G$ is a natural transformation between
$F,G\in\Sets^{X^\op}$, we define the action of $E$ on $\alpha$, denoted by
$\hat\alpha:\hat F\to\hat G$, by
\begin{equation}\label{eq:hatalph}
\hat\alpha_p(f)(x)=\alpha_x\big(f(x)\big),
\end{equation}
where $p\in\P$, $f\in\hat F(p)$ and $x\in X\cap\down p$.

This action on morphisms is well defined. For this, we need the naturality of $\alpha$:
\begin{equation*}
\xymatrix{F(x)\ar[rr]^{\alpha_x}\ar[dd]_{F(y\leq x)} && G(x)\ar[dd]^{G(y\leq x)}\\
\\
F(y)\ar[rr]_{\alpha_y} && G(y)}
\end{equation*}
We have to show that $\hat\alpha_p(f)\in\hat G(p)$. So let $y\leq x$ in $X\cap\down p$. Then
\begin{align*}
 G(y\leq x)\hat\alpha_p(f)(x) &  =G(y\leq x)\alpha_x\big(f(x)\big) =\alpha_y\big(F(y\leq x)f(x)\big)\\
 & =\alpha_y\big(f(y)\big) =\hat\alpha_p(f)(y).
\end{align*}
Furthermore, we have to show that $\hat\alpha$ is a natural transformation between $\hat F$ and $\hat G$. That is,
\begin{equation*}
\xymatrix{\hat F(p)\ar[rr]^{\hat \alpha_p}\ar[dd]_{\hat F(q\leq p)} && \hat G(p)\ar[dd]^{\hat G(q\leq p)}\\
\\
\hat F(q)\ar[rr]_{\hat \alpha_q} && \hat G(q)}
\end{equation*}
commutes. Notice that given a function $f$ with domain $X\cap\down p$, both $\hat G(q\leq p)\circ\hat\alpha_p(f)$ and $\hat\alpha_q\circ\hat F(q\leq p)(f)$ are functions with domain $X\cap\down q$. Hence for arbitrary $x\in X\cap\down q$ we find
\begin{align*}
\hat G(q\leq p)\circ\hat\alpha_p(f)(x) & = \left.\hat\alpha_p(f)\right|_{X\cap\down q}(x) = \hat\alpha_p(f)(x) = \alpha_x(f(x))\\
& = \alpha_x\left(\left.f\right|_{X\cap\down q}(x)\right) =  \hat\alpha_q\left(\left.f\right|_{X\cap\down q}\right)(x)\\
&  = \hat\alpha_q\circ\hat F(q\leq p)(f)(x),
\end{align*}
where we used (\ref{eq:hatFmor}) in the first and the last equality, and used (\ref{eq:hatalph}) in the third and fifth equality.
So the
diagram indeed commutes.

\begin{lemma}
Let $J$ be a Grothendieck topology on $\P$ such that $J_X\leq J$. Then $E$ restricts to a functor $\Sh(X,\bar J)\to \Sh(\P,J)$.
\end{lemma}
\begin{proof}
We have to show that $\hat F\in\Sh(\P,J)$ for each $F\in\Sh(X,\bar J)$.
 So let $\langle
f_q\rangle_{q\in S}$ be a matching family for $S\in J(p)$. Then by definition of $f_q\in\hat F(q)$, $f_q$ is a function $X\cap\down
q\to\bigcup_{x\in X}F(x)$ with $f_q(x)\in F(x)$ for each $x\in X\cap\down q$, and such that (\ref{eq:elFhat}) holds if we substitute $f_q$ for $f$ and $q$ for $p$.
Furthermore, the condition that $\langle f_q\rangle_{q\in S}$ is a matching family for
$S$ of elements of $\hat F$ translates to $\hat F(r\leq q)f_q=f_r$. By (\ref{eq:hatFmor}), it follows that for each $q\in S$, $r\leq q$ and $x\in X\cap\down r$, we have
\begin{align}\label{eq:mfforhatF}
f_q(x)=f_r(x).
\end{align}

We have to find a unique amalgamation $f$ for $\langle f_q\rangle_{q\in S}$. This means that we have to construct a function $f\in\hat F(p)$ such that $\hat F(q\leq p)f=f_q$ for each $q\in S$, and if there is another function $g\in\hat F(p)$ such that
\begin{equation}\label{eq:uniquenessf}
\hat F(q\leq p)g=f_q,
\end{equation}
for each $q\in S$, we must have $f=g$.

Let $x\in X\cap\down p$. We would like to define $f(x)=f_x(x)$, but since it might be possible that $x\notin S$ and $f_x$ is only given for $x\in S$, we cannot do so. However, it is still possible to construct $f$. First notice that if $x\in S$, $f_x\in \hat F(x)$ is defined such that $f_x(y)\in F(y)$ for each $y\in X\cap\down x$, so $f_x(x)\in F(x)$ for each $x\in S$.

For each $x\in X\cap\down p$, we have $x\leq p$, and since $S\in J(p)$, we have $S\cap\down x\in J(x)$ by the stability axiom for Grothendieck topologies. It follows that $S\cap\down x\cap X\in\bar J(x)$. Now, $f_y$ exists for each $y\in S\cap\down x\cap X$, and $\langle f_y(y)\rangle_{y\in S\cap\down x\cap X}$ is a matching family for $S\cap\down x\cap X$ of elements of $F$. Indeed, if $z\leq y$ in $S\cap\down x\cap X\subseteq X\cap\down x$, we have
\begin{equation*}
F(z\leq y)f_y(y)=f_y(z)=f_z(z),
\end{equation*}
where, in the first equality, we used $f_y\in\hat F(y)$, and in the second equality, we used (\ref{eq:mfforhatF}).
Since $F$ is a $\bar J$-sheaf, we find that $\langle f_y(y)\rangle_{y\in S\cap\down x\cap X}$ has a unique amalgamation $f(x)\in F(x)$. Notice that if $x\in S$, then $f(x)=f_x(x)$.

 If $y\leq x$ in $X\cap\down p$, then $\langle f_z(z)\rangle_{z\in S\cap\down y\cap X}\subseteq\langle f_z(z)\rangle_{z\in S\cap\down x\cap X}$. The unique amalgamation of the left-hand side of the inclusion is $f(y)$. Since $f(x)$ is the amalgamation of the right-hand side of the inclusion, we have $$f_z(z)=F(z\leq x)f(x)=F(z\leq y)F(y\leq x)f(x)$$ for each $z\in S\cap\down y\cap X$. Thus $F(y\leq x)f(x)$ is another amalgamation of the left-hand side of the inclusion, whence $F(y\leq x)f(x)=f(y)$. So $f:X\cap\down p\to\bigcup_{x\in X}F(x)$, defined by $x\mapsto f(x)$, is an element of $\hat F(p)$.

Let $q\in S$. Then for each $x\in X\cap\down q$, we have $x\in S$, so $f(x)=f_x(x)$ as we have noted before. By (\ref{eq:mfforhatF}) we have $f_q(x)=f_x(x)$, so $\hat F(q\leq p)f=f_q$, and we see that $f$ is indeed an amalgamation for the matching family
$f_q$.

If $g\in\hat F(p)$ satisfies (\ref{eq:uniquenessf}) for each $q\in S$, then $$g(q)=\left.g\right|_{X\cap\down q}=\left(\hat F(q\leq p)g\right)(q)=f_q(q).$$
Since $g\in\hat F(p)$, we have $F(y\leq x)g(x)=g(y)$ for each $y\leq x$ in $X\cap\down p$ and since we just noticed that $f_y(y)=g(y)$ for each $y\in S\cap\down x\cap X$, which is a subset of $X\cap\down p$, it follows that also $g(x)$ is an amalgamation of $\langle f_y(y)\rangle_{y\in S\cap\down x\cap X}$. Hence $g(x)=f(x)$ for each $x\in X\cap\down p$.
It follows that $f=g$, so $f$ is the unique amalgamation of $\langle f_q\rangle_{q\in S}$.

\end{proof}

If $J_X\leq J$, we have $\down(X\cap\down p)\in J(p)$ for each $p\in\P$. It turns out that if $F\in\Sh(\P,J)$, then $F$ satisfies a sheaflike condition for the generating set $X\cap\down p$:
\begin{lemma}
Let $X$ be a subset of a poset $\P$ and $J$ a Grothendieck topology on $\P$ such that $J_X\leq J$. Let $F\in\Sh(\P,J)$. Then for each $p\in\P$, if $\langle a_x\rangle_{x\in X\cap\down p}$ such that $a_x\in F(x)$ for each $x\in X\cap\down p$ and $F(y\leq x)a_x=a_y$ for each $y\leq x$ in $X\cap\down p$, then there is a unique $a\in F(p)$ such that $F(x\leq p)a=a_x$.
\end{lemma}
Although $X\cap\down p$ generates a sieve on $p$, but is not necessarily a sieve itself, we shall also refer to $a$ as the amalgamation of $\langle a_x\rangle_{x\in X\cap\down p}$.
\begin{proof}
First we consider the case that $X\cap\down p=\emptyset$. Then there
is only one family $\langle a_x\rangle_{x\in X\cap\down p}$ such that the conditions of the lemma hold, namely the empty family, and clearly any point in $F(p)$ is an amalgamation of this empty family. So we have to show that $F(p)\cong 1$. Since $X\cap\down
p=\emptyset$, it follows that $\emptyset=\down(X\cap\down p)\in J(p)$. Now, there
is only one matching family for this sieve, namely the empty
matching family, which must have a unique amalgamation $a$, so indeed
$F(p)=\{a\}$. Notice that the same conclusion could be drawn from
Example \ref{ex:discretetopologysheaves}.

If $X\cap\down p\neq\emptyset$, for each
$z\in\down(X\cap\down p)$ we have an $x\in X\cap\down p$ such that $z\leq
x$. So given a family $\langle a_x\rangle_{x\in X\cap\down p}$ satisfying the conditions stated in the lemma, we define $a_z:=F(z\leq x)a_x$. We
have to show that $a_z$ does not depend on the point $x\geq z$. As we have seen above, we have $F(z)\cong 1$ if
$X\cap\down z=\emptyset$. Since $a_z\in F(z)$, we find that
$F(z)=\{a_z\}$, so if $y\in X\cap\down p$ such that $z\leq y$, we
also must have $a_z=F(z\leq y)a_y$.

If $X\cap\down z\neq\emptyset$, then also $\down(X\cap\down z)\in
J(z)$ is non-empty. So for each $v\in\down(X\cap\down z)$ there is
a $w\in X\cap\down z$ such that $v\leq w$. Hence
\begin{eqnarray*}
F(v\leq z)F(z\leq x)a_x & = & F(v\leq w)F(w\leq z)F(z\leq
y)a_y=F(v\leq w)F(w\leq y)a_y\\
& =& F(v\leq w)a_w,
\end{eqnarray*}
since $w,x\in X\cap\down p$. In a similar way, we find $$F(v\leq
z)F(z\leq y)a_y=F(v\leq w)a_w.$$ In other words, for each
$v\in\down(X\cap\down z)\in J(z)$, the images under $F(v\leq z)$
of $F(z\leq x)a_x$ and $F(z\leq y)a_y$ are the same, and since $F$
is a $J$-sheaf, it follows that $$F(z\leq x)a_x=F(z\leq y)a_y.$$
So $a_z$ is uniquely determined and it is clear from the way it is
defined that $\langle a_z\rangle_{z\in\down(X\cap\down p)}$ is a
matching family. It follows that there is a unique amalgamation $a\in F(p)$ for $\langle a_x\rangle_{x\in\down(X\cap\down p)}$, so $a$ is also an amalgamation for $\langle a_x\rangle_{x\in X\cap\down p}$. In order to show that $a$ is also the unique amalgamation of $\langle a_x\rangle_{x\in X\cap\down p}$, let $b\in F(p)$ such that $a_x=F(x\leq p)b$ for each $x\in X\cap\down p$. Then for each $z\in\down(X\cap\down p)$, there must be an $x\in X\cap\down p$ such that $z\leq x$, whence
\begin{equation*}
F(z\leq p)b=F(z\leq x)F(x\leq p)b=F(z\leq x)a_x=a_z.
\end{equation*}
So we see that $b$ is an amalgamation for $\langle a_z\rangle_{z\in\down(X\cap\down p)}$, whence $b=a$.
\end{proof}

\begin{lemma}
 Let $J$ be a Grothendieck topology such that $J_X\leq J$. Then for each $F\in\Sh(\P,J)$, there is a natural bijection $\phi_F:E\circ I^*(F)\to F$. Moreover, the family $\{\phi_F\}_F$ constitutes a natural isomorphism $\phi:E\circ I^*\to 1_{\Sh(\P,J)}$.
\end{lemma}
\begin{proof}
Let $F\in\Sh(\P,J)$ and write $\hat F$ instead of $E\circ I^*(F)=\widehat{F\circ I}$. Then $\hat F(p)$ is the set of functions $f:X\cap\down p\to\bigcup_{x\in X}F(x)$ such that $f(x)\in F(x)$ for each $x\in X$ and $F(y\leq x)f(x)=f(y)$ for each $x,y\in X$ such that $y\leq x$.
We define the natural bijection $\phi_F: \hat F\to F$ as follows. Since the family $\langle f(x)\rangle_{x\in X\cap\down p}$ satisfies the conditions of the previous lemma, we define $(\phi_F)_p(f)$ to be the unique amalgamation of this family, whose existence is assured by the same lemma. Clearly $(\phi_F)_p$ is a bijection, but we also have to show that it is natural, i.e., we have to show that for each $q\leq p$, in $\P$
the following diagram commutes:
\begin{equation*}
\xymatrix{\hat F(p)\ar[rr]^{(\phi_F)_p}\ar[dd]_{\hat F(q\leq p)} && F(p)\ar[dd]^{F(q\leq p)}\\
\\
\hat F(p)\ar[rr]_{(\phi_F)_q} && F(q)}
\end{equation*}
Let $f\in\hat F(p)$, so that $(\phi_F)_p(f)$ is the amalgamation of the family $\langle f(x)\rangle_{x\in X\cap\down p}$. Let $$g=\hat
F(q\leq p)f=\left.f\right|_{X\cap\down q}.$$ Then for each $x\in X\cap\down q$, we find $$g(x)=f(x)=F(x\leq p)(\phi_F)_p(f)=F(x\leq q)F(q\leq p)(\phi_F)_p(f).$$ In other words, $F(q\leq p)(\phi_F)_p(f)$ is the amalgamation of the family $\langle g(x)\rangle_{x\in X\cap\down q}$. But by definition of $\phi_F$, this is exactly $(\phi_F)_q(g)$. Hence $$F(q\leq p)(\phi_F)_p(f)=(\phi_F)_q\circ F(q\leq p)(f).$$

Finally, we have to show that $\phi$ is a natural isomorphism. Since each component $\phi_F$ is an isomorphism in $\Sets$, we only have to show that $\phi$ is natural in $F$. In other words, let $\alpha:F\to G$ be a natural transformation in $\Sh(\P,J)$. Then for each $p\in\P$,
\begin{equation*}
\xymatrix{\hat F(p)\ar[rr]^{(\phi_F)_p}\ar[dd]_{\hat\alpha_p} && F(p)\ar[dd]^{\alpha_p}\\
\\
\hat G(p)\ar[rr]_{(\phi_G)_p} && G(p)}
\end{equation*}
should commute. So let $f\in\hat F(p)$, let $a=(\phi_F)_p(f)$ be the amalgamation of $\langle f(x)\rangle_{x\in X\cap\down p}$, and let $g\in\hat G(p)$ be given by $\hat\alpha(f)$, which means that $g(x)=\alpha_x(f(x))$ for each $x\in X\cap\down p$. Then $(\phi_G)_p(g)$ is the unique amalgamation of $\langle g(x)\rangle_{x\in X\cap\down p}$, but since $\alpha:F\to G$ is natural, we find for each $x\in X\cap\down p$ that
\begin{equation*}
 G(x\leq p)\alpha_p(a)=\alpha_xF(x\leq p)a=\alpha_x(f(x))=g(x),
\end{equation*}
whence $\alpha_p(a)$ is an amalgamation of $\langle g(x)\rangle_{x\in X\cap\down p}$. In other words, we must have $(\phi_G)_p(g)=\alpha_p(a)$.
\end{proof}

\begin{lemma}
  Let $J$ be a Grothendieck topology such that $J_X\leq J$. Then for each $F\in\Sh(X,\bar J)$, there is a natural bijection $\psi_F:I^*\circ E(F)\to F$. Moreover, the $\psi_F$ constitute a natural isomorphism $\psi:I^*\circ E\to 1_{\Sh(X,\bar J)}$.
\end{lemma}
\begin{proof}
 Let $F\in\Sh(X,\bar J)$. Then for each $x\in X$ we have $$I^*\circ E(F)(x)=\hat F\circ I(x)=\hat F(x).$$ Then define $(\psi_F)_x:\hat F\circ I(x)\to F(x)$ by $f\mapsto f(x)$. This is well defined, for $x\in X$ and $f$ is a function $X\cap\down x\to\bigcup_{x\in X}$, so $x$ lies in the domain of $f$. Now, $\down(X\cap\down x)\in J(x)$, so $X\cap\down x\in\bar J(x)$, so $\langle f(y)\rangle_{y\in X\cap\down x}$ is a matching family for the cover $X\cap\down x$ of elements of $F\in\Sh(X,\bar J)$, hence there is a unique amalgamation $a\in F(x)$ such that $F(y\leq x)a=f(y)$. But since clearly $f(x)$ is also an amalgamation, we see that $a=f(x)$, whence the assignment $(\psi_F)_x:f\mapsto f(x)$ is a bijection. We have to show that $\psi_F$ is natural, so for each $x,y\in X$ such that $y\leq x$,
\begin{equation*}
\xymatrix{\hat F(x)\ar[rr]^{(\psi_F)_x}\ar[dd]_{\hat F(y\leq x)} && F(x)\ar[dd]^{F(y\leq x)}\\
\\
\hat F(x)\ar[rr]_{(\psi_F)_y} && F(y)}
\end{equation*}
should commute. Now, if $f\in\hat F(x)$, let $$g=\hat F(y\leq x)f=\left.f\right|_{X\cap\down q}.$$ Then $$(\psi_F)_y(g)=g(y)=f(y),$$ whilst $F(y\leq x)f(x)=f(y)$ by definition of $f\in\hat F(x)$. So the diagram clearly commutes.

Finally, we show that $\psi$ is a natural isomorphism. Again, since each component $\psi_F$ is a bijection, we only have to show that $\psi$ is natural in $F$, i.e., if $\alpha:F\to G$ in $\Sh(X,\bar J)$, then the following diagram
\begin{equation*}
\xymatrix{\hat F(x)\ar[rr]^{(\psi_F)_x}\ar[dd]_{\hat\alpha_x} && F(x)\ar[dd]^{\alpha_x}\\
\\
\hat G(x)\ar[rr]_{(\psi_G)_x} && G(x)}
\end{equation*}
should commute for each $x\in X$. Let $g=\hat\alpha(f)\in\hat G(x)$. Then $g(y)=\alpha_y(f(y))$, so $(\psi_G)_x(g)=g(x)=\alpha_x(f(x))$. But this is exactly $\alpha_x\circ(\psi_F)_x(f)$.
\end{proof}

\begin{theorem}[Comparison Lemma, posetal case]\label{thm:ComparisonLemma}
 Let $\P$ be a poset and $X\subseteq\P$ a subset. If $I:X\to\P$ is the inclusion, and $J$ is a Grothendieck topology on $\P$ such that $J_X\leq J$, then $I^*:\Sh(\P,J)\to\Sh(X,\bar J)$ and $E:\Sh(X,\bar J)\to\Sh(\P,J)$ form an equivalence of categories.
\end{theorem}
\begin{proof}
The previous two lemmas show that $I^*\circ E\cong 1$ and $E\circ I^*\cong 1$, so $\Sh(\P,J)\cong\Sh(X,\bar J)$.
\end{proof}

\begin{corollary}\label{cor:JXsheaves}
 Let $\P$ be a poset and $X\subseteq\P$ a subset. Then $\Sh(\P,J_X)\cong\Sets^{X^\op}$.
\end{corollary}
\begin{proof}
 For each $x\in $, we have $\bar J_X(x)=\{S\cap X:S\in J_X(x)\}$. Since $x\in X$, it follows that $$J_X(x)=\{S\in\D(\down x):X\cap\down x\subseteq S\}=\{\down x\},$$ so $\bar{J}_X(x)=\{\bar\down x\}$, the indiscrete topology on $X$. By Example \ref{ex:indiscretetopologysheaves}, we find $\Sh(X,\bar{J}_X)=\Sets^{X^\op}$.
\end{proof}

We can put an equivalence relation on the set of Grothendieck topologies of a poset by declaring two Grothendieck topologies $J_1$ and $J_2$ to be equivalent if and only if $\Sh(\P,J_1)\cong\Sh(\P,J_2)$. If $X,Y\subseteq\P$ are subsets, using the previous corollary we find that $J_X\simeq J_Y$ if and only if there is an order isomorphism between $X$ and $Y$ if we equip both sets with the order inherited from $\P$.

A question which might be interesting is which subset topologies are subcanonical. In the case of an Artinian poset, we are also able to determine the condition of canonicity. We start by introducing the notion of canonical and subcanonical topologies.
\begin{definition}
 Let $J$ be a Grothendieck topology on a category $\CC$. Then $J$ is called \emph{subcanonical}\index{subcanonical Grothendieck topology}
if all representable presheaves are $J$-sheaves. That is, if
$\y(C)=\Hom(-,C)$ is a $J$-sheaf for each $C\in\CC$. The
\emph{canonical}\index{canonical Grothendieck topology} topology on
$\CC$ is the finest subcanonical topology.
\end{definition}

It turns out to be quite easy to find a condition for subcanonicity of an arbitrary Grothendieck topology on an arbitrary poset, since every homset of each poset $\P$ contains at most one element. Therefore, given a matching family for a cover of elements of $\y(p)$ for $p\in\P$, the existence of an amalgamation automatically implies uniqueness. Another way of expressing this is saying that $\y(p)$ is a \emph{separated presheaf}. The existence of an amalgamation for a matching family for a cover of an element $q\in\P$ is assured if $q\leq p$. So if $q\nleq p$, the sheaf condition only holds if there is no matching family for every cover of $q$. The next proposition assures that this condition is necessary and sufficient.

\begin{proposition}\label{prop:representablesheaf}
 Let $J$ be a Grothendieck topology on a poset $\P$ and let $p\in\P$. Then $\y(p)$ is an $J$-sheaf if and only if $J(q)\cap\D(\down p)=\emptyset$ for each $q\nleq p$.
\end{proposition}
\begin{proof}
We first show that for $q\leq p$, the sheaf condition is always satisfied. Let $q\leq p$ and $S\in J(q)$. Then if $\langle a_x\rangle_{x\in S}$ is a matching family for $S$ of elements of $\y(p)$, we have $x\leq p$, so $a_x\in\y(p)(x)=\Hom(x,p)\cong 1$. Thus $a_x$ is fixed for each $x\in S$, so there is only one matching family for $S$.  As a consequence, if $R\in J(q)$ such that $S\subseteq R$ and $\langle b_x\rangle_{x\in R}$ is a matching family for $R$, we must have $b_x=a_x$ for each $x\in S$. In particular, we can take $R=\down q$. Now, $\langle\y(p)(x\leq q)a\rangle_{x\leq q}$ is a matching family for $R$, where $a$ is the unique point in $\y(p)(q)=\Hom(q,p)$, which exists since $q\leq p$. So $a_x=\y(p)(x\leq q)a$ for each $x\in S$, which proves that $a$ is the unique amalgamation of $\langle a_x\rangle_{x\in S}$.

Now assume that $q\nleq p$ and $J(q)\cap\D(\down p)=\emptyset$. Let $S\in J(q)$, then $S\notin\D(\down p)$, so $S\nsubseteq\down p$. Then there must be an $x\in S$ such that $x\notin\down p$, or equivalently, there is an $x\in S$ such that $x\nleq p$. Hence $\y(p)(x)=\Hom(x,p)=\emptyset$, so there is no matching family $\langle a_y\rangle_{y\in S}$ for $S$, since this requires the existence of an element $a_x\in\y(p)(x)$. We conclude that $\y(p)\in\Sh(\P,J)$.

The other direction follows by contraposition. So assume that $q\nleq p$ and furthermore, assume that $J(q)\cap\D(\down p)\neq\emptyset$. So there is an $S\in J(q)$ such that $S\subseteq\down p$. Then for each $x\in S$ we have $x\leq p$, so there is a unique element $a_x\in\y(p)(x)=\Hom(x,p)\cong 1$, and $\langle a_x\rangle_{x\in S}$ define a matching family for $S$. However, since $q\nleq p$, we have $\y(p)(q)=\Hom(q,p)=\emptyset$, so there is no amalgamation of the matching family. We conclude that $\y(p)\notin\Sh(\P,J)$.
\end{proof}

\begin{proposition}\label{prop:subcanonicalJX}
 Let $X$ be a subset of a poset $\P$, and let $p\in\P$. Then $\y(p)\in\Sh(\P,J_X)$ for each $p\in\P$ if and only if $X\to\down p=\down p$. As a consequence, $J_X$ is subcanonical if and only if $X\to\down p=\down p$ for each $p\in \P$.
\end{proposition}
\begin{proof}
Assume that $X\to\down p=\down p$ and let $q\nleq p$. Then $\down q\nsubseteq X\to\down p$, so by definition of the Heyting implication, we find that $X\cap\down q\nsubseteq \down p$. Let $S\in J_X(q)$. Then $S$ contains $X\cap\down q$, hence $S\nsubseteq\down p$. Thus $J_X(q)\cap\D(\down p)=\emptyset$.

Now let $X\to\down p\neq\down p$. Since $X\cap\down p\subseteq\down p$, we have $\down p\subseteq X\to\down p$. Hence, there must be a $q\notin\down p$, i.e. $q\nleq p$, such that $q\in X\cap\down p$. But then $\down q\subseteq X\to\down p$, so $X\cap\down q\subseteq\down p$. Hence $\down(X\cap\down q)\subseteq\down p$, and since $\down(X\cap\down q)\in J_X(q)$, we find that $J_X(q)\cap\D(\down p)\neq\emptyset$. By contraposition, we find that $J_X(q)\cap\D(\down p)=\emptyset$ for each $q\nleq p$ implies $X\to\down p=\down p$. The claim now follows directly from Proposition \ref{prop:representablesheaf}.
\end{proof}

\begin{example}
 Let $\P=\{x,y,z\}$ with the order specified by $x<y$ and $x<z$. Then $J_X$ with $X=\{y,z\}$ is subcanonical, since $X\to\down p=\down p$ for $p=x,y,z$. On the other hand, $J_Y$ with $Y=\{y\}$ is not subcanonical, since $Y\to\down x=\down z$.
\end{example}

Let $\P$ be an Artinian poset with $X,Y\subseteq\P$ such that $X\subseteq Y$. Then $S\in J_Y(p)$ if and only if $Y\cap\down p\subseteq S$, which implies that $X\cap\down p\subseteq S$, or equivalently, $S\in J_X(p)$. So $X\subseteq Y$ implies $J_Y\subseteq J_X$. The canonical topology $J_{\mathrm{canonical}}$ on $\P$ is defined as the largest subcanonical topology on $\P$. Since on Artinian posets all topologies are of the form $J_X$ for some subset $X$ of $\P$, we find that $J_{\mathrm{canonical}}=J_X$, where $X$ is the smallest subset of $\P$ satisfying $X\to\down p=\down p$ for all $p\in\P$.

\section{Other Grothendieck topologies and its sheaves}

We found that if $\P$ is downwards-directed, then the dense Grothendieck topology is actually the atomic Grothendieck topology on $\P$.
It turns out that we can find a whole new class of Grothendieck topologies on $X$ related to the atomic Grothendieck topology that are not always subset Grothendieck topologies. Since we can extend these Grothendieck topologies to Grothendieck topologies on $\P$ that are not subset Grothendieck topologies, they might be interesting to investigate. Another reason for investigating this class of Grothendieck topologies is that, together with the subset Grothendieck topologies, they exhaust all Grothendieck topologies on $\Z$, which can be found in the Appendix as Proposition \ref{prop:classificationGTofZ}.

\begin{proposition}
 Let $\P$ be a downwards directed poset and $X\subseteq\P$. Then $K_X$ defined pointwise by
\begin{equation}
 K_X(p)=J_X(p)\setminus\{\emptyset\}
\end{equation}
is a Grothendieck topology on $\P$, which we call the \emph{derived Grothendieck
topology}\index{derived Grothendieck topology} with respect to the subset $X$.
\end{proposition}
\begin{proof}
 Since $\down p\in J_X(p)$ for each $p\in\P$, we have $\down p\in K_X(p)$ for each $p\in\P$. If $S\in K_X(p)$, then $S\in J_X(p)$ so $S\cap\down q\in J_X(q)$ for each $q\leq p$. Moreover, since $S\in K_X(p)$, we must have $S\neq\emptyset$. Since $\P$ is downwards directed, this implies that $S\cap\down q\neq\emptyset$, so $S\cap\down q\in K_X(q)$ for each $q\leq p$. Finally, let $S\in K_X(p)$ and $R\in\D(\down p)$ be such that $R\cap\down q\in K_X(q)$ for each $q\in S$. So $R\neq\emptyset$, since $R\cap\down q\neq\emptyset$ if there exists a $q\in S$, which is the case since $S\in K_X(p)$ implies that $S$ is non-empty. Moreover, $S\in K_X(p)$ implies $S\in J_X(p)$, and $R\cap\down q\in K_X(q)$ for each $q\in S$ implies $R\cap\down q\in J_X(q)$ for each $q\in S$. So $R\in J_X(p)$, and since $R\neq\emptyset$, we find that $R\in K_X(p)$.
\end{proof}

 In the definition it might not be directly visible that the derived Grothendieck topologies are related to the atomic Grothendieck topology, but this follows from the next lemma.

\begin{lemma}\label{lem:derivedtopologies}
 Let $\P$ be a downward directed poset and $X\subseteq\P$ a subset of $\P$. Then
 \begin{equation}
 K_X(p)=\left\{
       \begin{array}{ll}
         J_X(p), & p\in\up X; \\
         J_{\mathrm{atom}}(p), & p\notin \up X.
       \end{array}
     \right.
\end{equation}
In particular, if $\up X=\P$, we have $K_X=J_X$.
\end{lemma}
\begin{proof}
Let $p\in\up X$. Then $X\cap\down p\neq\emptyset$, since otherwise
we must have $x\nleq p$ for each $x\in X$, contradicting $p\in\up X$.
It follows that $X\cap\down p\nsubseteq \emptyset$, so
$\emptyset\notin J_X(p)$, whence $J_X(p)=K_X(p)$. If $p\notin\up X$,
we have $X\cap\down p=\emptyset$, so $J_X(p)=\D(\down p)$. It
follows that $K_X(p)=\D(\down
p)\setminus\{\emptyset\}=J_{\mathrm{atom}}(p)$.
\end{proof}
Since the atomic topology is only defined on downwards directed
subsets, $(\up X)^c$ should be downwards directed. This is indeed
the case: if $p_1,p_2\in(\up X)^c$, then there is a $q\in\P$ such
that $q\leq p_1,p_2$, since $\P$ is downwards directed. However, we
cannot have $q\in\up X$, since this would imply that $p_1,p_2\in\up
X$. Hence $q\in(\up X)^c$.

Furthermore, we see that $K_\emptyset=J_{\mathrm{atom}}$. Even if
$X$ is not dense, there might be another subset $Y\subseteq\P$ such
that $K_X=J_Y$.

\begin{lemma}\label{lem:derivedtopologiesforposetswithleastelement}
 Let $\P$ be a downwards directed poset with a least element $0$. Then for each subset $X\subseteq\P$, we have $K_X=J_{X_0}$, where $X_0=X\cup\{0\}$.
\end{lemma}
\begin{proof}
Let $p\in\P$. Then if $S\in\D(\down p)$, we clearly have
$S\neq\emptyset$ if and only if $0\in S$, since $S$ is downwards
closed and $0$ lies below each element of $S$. Assume $p\in\up X$
and $S\in\D(\down p)$. Then $X_0\cap\down p\subseteq S$ implies
$X\cap\down p\subseteq S$, so $J_{X_0}(p)\subseteq J_X(p)$. On the
other hand, since $p\in\up X$, we have $X\cap\down p\neq\emptyset$,
so if $X\cap\down p\subseteq S$, then $S\neq\emptyset$, so $0\in S$.
Hence we find that $X_0\cap\down p\subseteq S$, whence
$J_X(p)\subseteq J_{X_0}(p)$. Thus $J_{X_0}(p)=J_X(p)$.

Now assume $p\notin\up X$. Then $X\cap\down p=\emptyset$, so
$X_0\cap\down p=\{0\}$. We find that $J_{X_0}(p)$ is exactly the set
of all $S\in\D(\down p)$ such that $0\in S$, which are exactly all
non-empty sieves on $p$. Thus $J_{X_0}(p)=J_\atom(p)$. By the
previous lemma, we conclude that $J_{X_0}(p)=K_X(p)$ for each
$p\in\P$.
\end{proof}
Notice that if $\P$ is downwards directed and Artinian, we
automatically have a least element $0$. An example of a downwards
directed poset without a least element is $\Z$. Then for any
$p\in\Z$ we have
\begin{equation*}
\bigcap K_\emptyset(p)=\bigcap\left(\D(\down
p)\setminus\{\emptyset\}\right)=\bigcap_{q\leq p}\down q=\emptyset,
\end{equation*}
so $K_\emptyset$ is not complete. Since $J_X$ is complete for each
$X\subseteq\Z$, we find that $K_\emptyset\neq J_X$ for each
$X\subseteq\P$. So we see that if $\P$ is non-Artinian, all
topologies are neither necessarily a subset Grothendieck topology, nor are they necessarily complete.

In Lemma \ref{lem:derivedtopologies} we found that $K_X$ is a
Grothendieck topology, which is constructed from two other
Grothendieck topologies, but since one of these topologies is the
atomic topology, we have to consider posets which are downwards
directed. Since the dense topology can be seen as a generalisation of the atomic topology for posets that are not necessarily downwards directed, one could try to generalize $K_X$ by defining
 \begin{equation*}
 K'_X(p)=\left\{
       \begin{array}{ll}
         J_X(p), & p\in\up X; \\
         J_{\mathrm{dense}}(p), & p\notin \up X.
       \end{array}
     \right.
\end{equation*}
However, this is not always a Grothendieck topology. For instance,
consider the poset $\P=\{x,y_1,y_2\}$ with $y_i\leq x$. Let
$X=\{y_1\}$. Then $K'_X(x)=J_X(x)=\{\down x,\down y_1\}$,
$K'_X(y_1)=J_X(y_1)=\{\down y_1\}$, whereas
$K'_X(y_2)=J_{\mathrm{dense}}(y_2)=\{\down y_2\}$. Now, since
$y_2\leq x$ and $\down y_1\in K'_X(x)$, we should have
$\emptyset=\down y_1\cap\down y_2\in K'_X(y_2)$ by the stability
axiom. Since this is not the case, $K'_X$ cannot be a Grothendieck
topology.

We shall calculate $K_X$-sheaves on a downwards directed poset $\P$. Moreover, if $\P$ is a poset, $X\subseteq\P$ is downwards directed and $Y\subseteq X$, we can consider the Grothendieck topology $L_{X,Y}$ on $\P$ obtained by extending the Grothendieck topology $K_Y^X$ on $X$ to $\P$. Since $L_{X,Y}$ is defined by the extension of a Grothendieck topology, whose sheaves are known, we can easily calculate the $L_{X,Y}$-sheaves on $\P$ as well. For the calculation of the $K_X$-sheaves we need the following definition.

\begin{definition}
Let $\P$ be a downwards directed poset. Even if $\P$ already has a
least element, we can add a new least element, which we denote by
$0$. We define $\P_0=\P\cup\{0\}$, where $0<p$ for each $p\in\P$. If
$A\subseteq\P$, then we write $A_0=A\cup\{0\}$. If $p\in\P$, then
$\down p$ with respect to $\P_0$ is exactly $(\down p)_0$, where
$\down p$ is defined with respect to $\P$. If $X\subseteq\P$, then $X\cap\down
p=X\cap(\down p)_0$, in which case the notation does not matter. If
$J$ is a topology on $\P$ and we want to define a similar topology
on $\P_0$, we denote this topology by $J^0$. For instance, $K^0_X$
is the derived topology on $\P_0$ with respect to $X$. By the
previous remark about $X\cap\down p=X\cap(\down p)_0$, we have $S\in
J_X^0(p)$ if and only if $X\cap\down p\subseteq S$.
\end{definition}

\begin{lemma}\label{lem:KX0intermsofKX}
 Let $\P$ be a downwards-directed poset. Then we have a bijection $\D(\P)\to\D(\P_0)\setminus\{\emptyset\}$ given by $A\mapsto A_0$, with inverse $B\mapsto B\setminus\{0\}$, which, for each $X\subseteq\P$ and $p\in\P$, restricts to a bijection $K_X(p)\to K^0_X(p)\setminus\{\{0\}\}$. Moreover, $\{0\}\in K_X^0(p)$ if and only if $p\in(\up X)^c$, where the complement is taken with respect to $\P_0$.
\end{lemma}
\begin{proof}
Clearly, $A_0\in\D(\P_0)\setminus\{\emptyset\}$ if $A\in\D(\P)$ and $B\setminus\{0\}\in\D(\P)$ if $B\in\D(\P_0)\setminus\{\emptyset\}$. It is also clear that $A_0\setminus\{0\}=A$ for $A\in\D(\P)$. It follows from $0\in B$ for each $B\in\D(\P_0)\setminus\{\emptyset\}$ that $(B\setminus\{0\}_0=B$.

Let $S\in K_X(p)$. If $p\in\up X$, this means that $\emptyset\neq X\cap\down p\subseteq S$. Then $X\cap\down p\subseteq S_0$, so $S_0\in K^0_X(p)$. If $p\notin\up X$, then $S\in J_\atom(p)$, so $S$ is non-empty. Then also $S_0\neq\emptyset$, so $S_0\in J^0_\atom(p)=K^0_X(p)$. In both cases, since $S\neq\emptyset$, we have $S_0\neq\{0\}$.

Let $R\in K^0(p)$ and assume that $R\neq\{0\}$. If $p\in\up X$, this means that $R\in J^0_X(p)$, so $X\cap\down p\subseteq R$. Since $0\notin X$, we find that  $X\cap\down p\subseteq R\setminus\{0\}$, hence $R\setminus\{0\}\in J_X(p)=K_X(p)$. If $p\notin\up X$, we have $R\in J^0_\atom(p)$, so $R$ is non-empty. Since also $R\neq\{0\}$, we must have $R\setminus\{0\}\neq\emptyset$, so $R\setminus\{0\}\in J_\atom(p)=K_X(p)$.

If $p\in(\up X)^c$, then for each $x\in X$ we have $x\nleq p$, so $X\cap\down p=\emptyset\subseteq\{0\}$, whence $\{0\}\in K_X^0(p)=J_X^0(p)\setminus\{\emptyset\}$. If $p\in\up X$, we have $x\leq p$ for some $x\in X$ and since $X\subseteq\P$, we have $x\neq 0$. Hence $x\in X\cap\down p$, so $X\cap\down p\nsubseteq\{0\}$. Thus $\{0\}\notin K_X^0(p)$.

\end{proof}

\begin{lemma}
 Let $\P$ be a downwards directed poset. Let $X\subseteq\P$ be a subset. If $I_0:\P\embeds\P_0$ denotes the map, then $I_0^*:\Sets^{\P_0^\op}\to\Sets^{\P^\op}$ given by $F\mapsto F\circ I_0$ restricts to a functor $\Sh(\P_0,K^0_X)\to\Sh(\P,K_X)$.
\end{lemma}
\begin{proof}
 Let $F\in\Sh(\P_0,K^0_X)$ and let $p\in\P$. Let $S\in K_X(p)$ and let $\langle a_x\rangle_{x\in S}$ be a matching family for $S$ of elements of $I_0^*(F)=F\circ I_0$. So $a_x\in F(x)$ for each $x\in S$. Since $K_X$-covers are always non-empty, we can find a $q\in S$ that allows us to define $a_0\in F(0)$ by $a_0=F(0\leq q)a_q$. This is independent of the chosen $q\in S$; if $q'\in S$, there exists an $r\leq q,q'$, for $\P$ is downwards directed, hence
\begin{eqnarray*}
F(0\leq q')a_{q'} & = & F(0\leq r)F(r\leq q')a_{q'}=F(0\leq r)a_r\\
 & = & F(0\leq r)F(r\leq q)a_q=F(0\leq q)a_q=a_0.
\end{eqnarray*}
By the previous lemma, $S_0\in K^0_X(p)$, and $a_0$ is constructed exactly in such a way that $\langle a_x\rangle_{x\in S_0}$ is a matching family for $S_0$ of elements of $F$. Since $F\in\Sh(\P_0,K^0_X)$, we find that there is a unique amalgamation $a\in F(p)$ such that $a_x=F(x\leq p)$ for each $x\in S_0$. To see that this is also a unique amalgamation for $S$, let $b\in F(p)$ such that $a_x=F(x\leq p)b$ for each $x\in S$. Then $$F(0\leq p)b=F(0\leq q)F(q\leq p)b=F(0\leq q)a_q=a_0,$$ so $b$ is also an amalgamation of $\langle a_x\rangle_{x\in S_0}$, whence $b=a$.
\end{proof}

\begin{definition}\label{def:E0}
 Let $\P$ be a downwards directed poset and take $X\subseteq\P$ such that $\up X\neq\P$. Given a fixed $p_0\in(\up X)^c$, define $E_0:\Sh(\P,K_X)\to\Sh(\P_0,K_X^0)$ as follows. $E_0$ acts on $F\in\Sh(\P,K_X)$ by $F\mapsto\bar F$, where $\bar F$ acts on objects of $\P_0$ by
 \begin{equation*}
 \bar F(p)=\left\{
       \begin{array}{ll}
         F(p), & p\neq 0; \\
         F(p_0), & p=0,
       \end{array}
     \right.
\end{equation*}
and acts on morphisms by
 \begin{equation*}
 \bar F(q\leq p)=\left\{
       \begin{array}{ll}
         F(q\leq p), & q\neq 0; \\
         F(r\leq p_0)^{-1}F(r\leq p), & q=0\mathrm{\ for\ some\ }r\leq p,p_0\mathrm{\ in\ }\P.
       \end{array}
     \right.
\end{equation*}
Furthermore, if $\alpha:F\to G$ in $\Sh(\P,K_X)$, then $E_0$ acts on $\alpha$ by $\alpha\mapsto\bar\alpha$, where
 \begin{equation*}
 \bar\alpha_p=\left\{
       \begin{array}{ll}
         \alpha_p, & p\neq 0; \\
         \alpha_{p_0}, & p=0.
       \end{array}
     \right.
\end{equation*}
\end{definition}
\begin{lemma}
$E_0$ is well defined.
\end{lemma}
\begin{proof}
Since $\up X\neq\P$, we can indeed find a $p_0\in(\up X)^c$. Let $F\in\Sh(\P,K_X)$. We have to show that $\bar F(0\leq p)$ is well defined. Since $\P$ is downwards directed, there is indeed an $r\in\P$ such that $r\leq p,p_0$. Now, $r\in\up X$ implies $p_0\in\up X$, which is not the case, so we have $r\notin\up X$. Since $p_0\notin\up X$, we have $K_X(p_0)=J_\atom(p_0)$. So $\down r\in K_X(p_0)$, and hence by Example \ref{ex:atomsheaves}, $F(r\leq p_0)$ is a bijection, so its inverse is defined. We have to show that the definition of $\bar F(0\leq p)$ is independent of the choice of $r\leq p,p_0$. That is, if $r'$ is another element of $\P$ such that $r'\leq p,p_0$, then we have to show that
\begin{equation}\label{eq:independentofr}
 F(r'\leq p_0)^{-1}F(r'\leq p)=F(r\leq p_0)^{-1}F(r\leq p).
\end{equation}
Since $\P$ is downwards directed, there is an $r''\leq r,r'$, which is necesarrily in $(\up X)^c$, so $F(r''\leq r)$ and $F(r''\leq r')$ are bijections.
Then
\begin{eqnarray*}
 F(r'\leq p_0)^{-1}F(r'\leq p) &  =  &  F(r'\leq p_0)^{-1}F(r''\leq r')^{-1}F(r''\leq r')F(r'\leq p)\\
&  = &  F(r''\leq p_0)^{-1}F(r''\leq p),
\end{eqnarray*}
and in the same way, we find $$F(r\leq p_0)^{-1}F(r\leq p)=F(r''\leq p_0)^{-1}F(r''\leq p),$$ which shows that (\ref{eq:independentofr}) indeed holds. To prove that $\bar F$ is a functor, we must show that $$\bar F(0\leq p)=\bar F(0\leq p')\bar F(p'\leq p)$$ if $p'\leq p$ in $\P$. Therefore, take an $r\leq p',p_0$, then also $r\leq p$, so
\begin{equation*}
 \bar F(0\leq p)=F(r\leq p_0)^{-1}F(r\leq p)=F(r\leq p_0)F(r\leq p')F(p'\leq p)=\bar F(0\leq p')\bar F(p'\leq p).
\end{equation*}

We have to show that $\bar F$ is a $K_X^0$-sheaf if $F\in\Sh(\P,K_X)$. Let $p\in\P_0$, $R\in K_X^0(p)$, and let $\langle a_x\rangle_{x\in R}$ be a matching family for $R$ of elements of $\bar F$. If $p=0$, then $R=\{0\}$, since there are no other covers. Then the matching family consists only of the single element $a_0\in\bar F(0)$, which is also automatically the unique amalgamation, so the sheaf condition is satisfied. Let $p\in\P$ and assume that $R\neq\{0\}$. Then by Lemma \ref{lem:KX0intermsofKX},  we have $R\setminus\{0\}\in K_X(p)$. Since $a_x\in\bar F(x)=F(x)$ for $x\in R\setminus\{0\}$, we find that $\langle a_x\rangle_{x\in R\setminus\{0\}}$ is a matching family for $R\setminus\{0\}$ of elements of $F$, which is a $K_X$-sheaf, so there is a unique amalgamation $a\in F(p)$ for this matching family. Since $R\neq\{0\}$, there is an $r_0\in\P$ such that $r_0\in R$. Now, take some $r\leq r_0,p_0$, then $r\in R$, so we find that
\begin{eqnarray*}
\bar F(0\leq p)a & = & F(r\leq p_0)^{-1}F(r\leq p)a=F(r\leq p_0)^{-1}F(r\leq r_0)F(r_0\leq p)a\\
& = & \bar F(r\leq r_0)a_{r_0}=a_0,
\end{eqnarray*}
where $F(r_0\leq p)a=a_{r_0}$ in the third equality, since $a$ is the amalgamation for $\langle a_x\rangle_{x\in R\setminus\{0\}}$ of elements of $F$. The last equality holds since $r\leq r_0$ in $R$ and $\langle a_x\rangle_{x\in R}$ is a matching family of elements of $\bar F$. Thus $a$ is also an amalgamation of the matching family for $R$ of elements of $\bar F$. If $b\in F(p)$ is another amalgamation of the family for $R$ of elements of $\bar F$, then $b$ is also an amalgamation of the family for $R\setminus\{\emptyset\}$ of elements of $F$, so $b=a$. This shows that $a$ is the unique amalgamation of $\langle a_x\rangle_{x\in R}$.

Now assume $R=\{0\}$. By Lemma \ref{lem:KX0intermsofKX}, this is only possible if $p\notin\up X$. So $K_X(p)=J_\atom(p)$, hence $\down r\in K_X(p)$ for each $r\leq p$ in $\P$. By Example \ref{ex:atomsheaves}, we find that $F(r\leq p)$ is a bijection for each $r\leq p$ in $\P$. Then if we choose $r\leq p,p_0$, we find that $\bar F(0\leq p)=F(r\leq p_0)^{-1}F(r\leq p)$ is also a bijection. So if $\langle a_x\rangle_{x\in R}$ is the matching family for $R=\{0\}$ with elements of $\bar F$, we see that this matching family consists only of a single element $a_0\in\bar F(0)$. Since $\bar F(0\leq p)$ is a bijection, it is clear that the matching family has an amalgamation $\bar F(0\leq p)^{-1}a_0$, which is necesarrily unique.

Finally, we have to show that $\bar\alpha:\bar F\to\bar G$ is natural if $\alpha:F\to G$ is a natural transformation in $\Sh(\P,K_X)$.
Now, $\bar\alpha_p=\alpha_p$, $\bar F(p)=F(p)$ and $\bar G(p)=G(p)$ if $p\in\P$, so we clearly have $\bar G(q\leq p)\bar\alpha_p=\bar\alpha_q\bar F(q\leq p)$ if $q,p\in\P$, since this is exactly the naturality of $\alpha$. If $q=0$, then
\begin{eqnarray*}
\bar G(0\leq p)\bar\alpha_p &
= & G(r\leq p_0)^{-1}G(r\leq p)\alpha_p=G(r\leq p_0)^{-1}\alpha_rF(r\leq p)\\
& = & G(r\leq p_0)^{-1}F(r\leq p_0)^{-1}F(r\leq p_0)\alpha_rF(r\leq p)\\
 & = & G(r\leq p_0)^{-1}\alpha_rF(r\leq p_0)F(r\leq p_0)^{-1}F(r\leq p)\\
& = & G(r\leq p_0)^{-1}G(r\leq p_0)\alpha_{p_0}F(r\leq p_0)^{-1}F(r\leq p) = \bar\alpha_{0}\bar F(0\leq p).
\end{eqnarray*}
So $\bar\alpha$ is indeed a natural transformation.
\end{proof}

\begin{theorem}
 Let $\P$ be a poset and take $X\subseteq\P$ such that $\up X\neq\P$. Then
\begin{eqnarray*}
I_0^*&:&\Sh(\P_0,K_X^0)\to\Sh(\P,K_X);\\
E_0&:&\Sh(\P,K_X)\to\Sh(\P_0,K_X^0)
\end{eqnarray*}
 form an equivalence of categories.
\end{theorem}
\begin{proof}
 Clearly, we have $I^*_0\circ E_0(F)=\bar F\circ I_0=F$ for each $F\in\Sh(\P,K_X)$, so $I^*_0\circ E_0=1_{\Sh(\P,K_X)}$. We aim to construct a natural isomorphism $\beta: E_0\circ I_0^*\to 1_{\Sh(\P_0,K_X^0)}$. First, given $F\in\Sh(\P_0,K_X^0)$, we have $E_0\circ I_0^*(F)=\widehat{F\circ I_0}$. If we abbreviate the right-hand side by $\bar F$, notice that $\bar F(p)=F(p)$ for each $p\in\P$ and $\bar F(q\leq p)=F(q\leq p)$ if $q\leq p$ in $\P$. Let $p_0$ the fixed point we chose in the previous lemma in order to construct $E_0$. Since $p_0\notin\up X$, we have $\{0\}=\down 0\in K_X^0(p_0)$, where we used Lemma \ref{lem:KX0intermsofKX}. Since $F\in\Sh(\P_0,K^0_X)$, we find by Example \ref{ex:atomsheaves} that $F(0\leq p_0)$ is a bijection. This allows us to define the component $\beta_F$ by
 \begin{equation*}
 (\beta_F)_p=\left\{
       \begin{array}{ll}
         1_{F(p)}, & p\neq 0; \\
         F(0\leq p_0), & p=0.
       \end{array}
     \right.
\end{equation*}
Notice that indeed $(\beta_F)_0$ maps $\bar F(0)$ into $F(0)$, since $\bar F(0)=F(p_0)$ by definition of $\bar F$. Furthermore, $(\beta_F)_p$ is clearly a bijection for each $p\in\P_0$. We have to check that it is natural, that is,
\begin{equation*}
\xymatrix{\bar F(p)\ar[rr]^{(\beta_F)_p}\ar[dd]_{\bar F(q\leq p)} &&  F(p)\ar[dd]^{ F(q\leq p)}\\
\\
\bar F(q)\ar[rr]_{(\beta_F)_q} && F(q)}
\end{equation*}
commutes for each $q\leq p$ in $\P_0$. If $q\in\P$ (so also $p\in\P$), this is exactly $$F(q\leq p)1_{F(p)}=1_{F(q)}F(q\leq p),$$ so assume $q=0$.
Choose an $r\leq p,p_0$. Then
\begin{eqnarray*}
 (\beta_F)_0\bar F(0\leq p) & = & F(0\leq p_0)F(r\leq p_0)^{-1}F(r\leq p)\\
& =& F(0\leq r)F(r\leq p_0)F(r\leq p_0)^{-1}F(r\leq p)\\
 & = & F(0\leq p)1_{F(p)} = \bar F(0\leq p)(\beta_F)_p.
\end{eqnarray*}
So we find that $\beta_F$ is a natural bijection for each $F$.
Finally, we have to show that $\beta$ is natural in $F$. So let $\alpha:F\to G$ in $\Sh(\P_0,K_X^0)$ and abbreviate $E_0\circ I_0^*(\alpha)=\widehat{\alpha I_0}$ by $\bar\alpha$. Then $\bar\alpha_p=\alpha_p$ for $p\in\P$ and $\bar\alpha_0=\alpha_{p_0}$. We have to show that for each $p\in\P_0$, the diagram
\begin{equation*}
\xymatrix{\bar F(p)\ar[rr]^{(\beta_F)_p}\ar[dd]_{\bar \alpha_p} && F(p)\ar[dd]^{ \alpha_p}\\
\\
\bar G(p)\ar[rr]_{(\beta_G)_p} && G(p)}
\end{equation*}
commutes. If $p\in\P$, this says that $1_{G(p)}\alpha_p=\alpha_p1_{F(p)}$, so assume $p=0$. Then by the naturality of $\alpha$ we obtain
\begin{equation*}
 (\beta_G)_0\bar\alpha_0=G(0\leq p_0)\alpha_{p_0}=\alpha_0F(0\leq p_0)=\alpha_0(\beta_F)_0.
\end{equation*}
So $\beta$ is a natural transformation such that each component $\beta_F$ is a natural bijection, hence $\beta:E_0\circ I_0^*\to 1$ is a natural isomorphism. We conclude that $E_0$ and $I_0^*$ constitute an equivalence of categories.
\end{proof}

\begin{corollary}\label{cor:KXsheaves}
 Let $\P$ be a downwards directed poset and $X\subseteq\P$. Then $$\Sh(\P,K_X)\cong\Sets^{Y^\op},$$ where $Y=X$ if $\up X=\P$, and $Y=X_0$ (considered as subset of $\P_0$) otherwise.
\end{corollary}
\begin{proof}
Assume that $\up X=\P$. Then by Lemma \ref{lem:derivedtopologies}, we have $K_X=J_X$, so $$\Sh(\P,K_X)=\Sh(\P,J_X)\cong\Sets^{X^\op},$$ by Corollary \ref{cor:JXsheaves}. Now assume that $\up X\neq\P$. Then $\up (X_0)=\P_0$, so $K^0_{X}=J^0_{X_0}$ by Lemma \ref{lem:derivedtopologiesforposetswithleastelement}. Using the previous theorem and Corollary \ref{cor:JXsheaves}, we find $$\Sh(\P,K_X)\cong\Sh(\P_0,K^0_X)=\Sh(\P_0,J^0_{X_0})\cong\Sets^{X_0^\op}.$$
\end{proof}

We are also interested in the subcanonicity of the derived Grothendieck topologies. It turns out that the condition for subcanonicity is exactly the same as for subset topologies.
\begin{proposition}
 Let $\P$ be a downwards directed poset and $X\subseteq\P$. Then for each $p\in\P$, $\y(p)$ is a $K_X$-sheaf if and only if $X\to\down p=\down p$. As a consequence, $K_X$ is subcanonical if and only if $X\to\down p=\down p$ for each $p\in\P$.
\end{proposition}
\begin{proof}
Let $X\to\down p=\down p$. Proposition \ref{prop:subcanonicalJX} assures that $\y(p)\in\Sh(\P,J_X)$, so by Proposition \ref{prop:representablesheaf}, we obtain $J_X(q)\cap\D(\down p)=\emptyset$ for each $q\nleq p$. Now $K_X\subseteq J_X$, so certainly $K_X(q)\cap\D(\down p)=\emptyset$ for each $q\nleq p$. Again by Proposition \ref{prop:representablesheaf}, we obtain $\y(p)\in\Sh(\P,K_X)$.

Now assume that $X\to\down p\neq\down p$. In the proof of Proposition \ref{prop:subcanonicalJX}, we showed that this implies that $X\cap\down q\subseteq\down p$ for some $q\nleq p$. If $q\in\up X$, we have $X\cap\down q\neq\emptyset$, so $\down(X\cap\down q)\in K_X(q)$, whereas also $\down(X\cap\down q)\subseteq\down p$. If $q\notin\up X$, it follows from Lemma \ref{lem:derivedtopologies} that $K_X(q)=J_\atom(q)$. Now, $\P$ is downwards directed, so there is an $r\leq p,q$, which implies that $\down r\subseteq\down p$ and $\down r\in K_X(q)$. So in both cases we find that $K_X(q)\cap\D(\down p)\neq\emptyset$, whence by Proposition \ref{prop:representablesheaf}, we obtain that $\y(p)\notin\Sh(\P,K_X)$.
\end{proof}

Finally, if $X$ is a downwards directed subset of a poset $\P$, we once again consider the map $\G(X)\to\G(\P)$ that assigns to each Grothendieck topology on $X$ its extension on $\P$ in Definition \ref{def:extensionGrothTop}.
\begin{definition}
 Let $\P$ be a poset and $X$ a downwards directed subset of $\P$. For each subset $Y$ of $X$, we define $L_{X,Y}$ to be the extension on $\P$ of the derived Grothendieck topology induced by $Y\subseteq X$ on $X$, which we denote by $K_Y^X$.
\end{definition}
Notice that $K_Y^X=J^X_\atom$, the atomic Grothendieck topology on $X$ if $Y=\emptyset$. So $L_{X,\emptyset}=L_X$.
 Recall the notion $\bar\up Z=\up Z\cap X$ in Definition \ref{def:subposetdownset}.
\begin{proposition}
 Let $\P$ be a poset and $X$ a downwards directed subset of $\P$. For each subset $Y$ of $X$ we have $$\Sh(\P,L_{X,Y})\cong\Sets^{Z^\op},$$ where $Z=Y$ if $\bar\up Z=X$, whereas $Z=Y_0$ if $\bar\up Y\neq X$.
\end{proposition}
\begin{proof}
 By Proposition \ref{prop:extensionofGrothTop}, $J_X\leq L_{X,Y}$, so $\bar L_{X,Y}$ is well defined and is equal to $K_Y^X$. Moreover, we are allowed to use Theorem \ref{thm:ComparisonLemma}, i.e., the Comparison Lemma. Hence $$\Sh(\P,L_{X,Y})\cong\Sh(X,K_X^Y)\cong\Sets^{Z^\op},$$ where we used Corollary \ref{cor:KXsheaves} in the last equivalence.
\end{proof}

\section*{Conclusion}
We have found that a poset $\P$ is Artinian if and only if every Grothendieck topology on $\P$ equals $J_X$ for some subset $X$ of $\P$.
If $\P$ is not Artinian, we used the extension of the atomic Grothendieck topology on a certain subset of $\P$ in order to show the existence of Grothendieck topologies that are not generated by some subset of $\P$. Now, one could ask the following: if we completely identify all Grothendieck topologies on downwards directed subsets of $\P$, do the extensions of these Grothendieck topologies together with the subset Grothendieck topologies exhaust all Grothendieck topologies on $\P$? And if so, how do we characterize all Grothendieck topologies on a downwards directed poset $\P$? A partial answer to the last question is given by the Grothendieck topologies of the form $K_X$ for some subset $X$ of $\P$. In case of $\P=\Z$, these Grothendieck topologies together with the $J_X$-topologies exhaust all Grothendieck topologies on $\Z$.

If there exist more Grothendieck topologies on a downwards directed poset $\P$ not equal to $\Z$, it might be possible to find the corresponding sheaf topoi using \emph{Artin glueing}. Uniqueness of sites yielding a sheaf topos is not assured, see for instance the remark below Corollary \ref{cor:JXsheaves}, hence this might not give a method of finding new Grothendieck topologies, but at least it provides some information about the corresponding Grothendieck topologies. A technique of obtaining new sheaf topoi is provided by \emph{Artin glueing}\index{Artin glueing}, which is described in \cite{CJ}. We illustrate why this technique might be useful by considering $K_X$-sheaves again.

By Lemma \ref{lem:derivedtopologies}, we know that $K_X$ is a Grothendieck topology on $\P$, which is equal to $J_X$ on $\up X$, whereas it is equal to $J_\atom$ on its complement in $\P$. By Corollary \ref{cor:JXsheaves}, $\Sets^{X^\op}$ is equivalent to $\Sh(\up X,J_X)$, with $J_X$ the subset Grothendieck topology on $\up X$ generated by $X$. Moreover, $\Sh((\up X)^c,J_\atom)\cong\Sets$ by Example \ref{ex:atomsheaves}, since each sheaf on $(\up X)^c$ is constant. Moreover, by Lemma \ref{lem:derivedtopologies}, we know that $K_X$ is a Grothendieck topology on $\P$, which is equal to $J_X$ on $\up X$, whereas it is equal to $J_\atom$ on its complement in $\P$. Therefore, in some way the $K_X$-sheaves are obtained by glueing $J_X$-sheaves on $\up X$ and $J_\atom$-sheaves on $(\up X)^c$, or equivalently, by glueing $\Sets^{X^\op}$ together with $\Sets$.

Artin glueing can be described as follows. Given two categories $\CC$ and $\mathcal{D}$ and a functor $F:\CC\to\mathcal{D}$, the category obtained by Artin glueing along $F$ is the comma category $(\mathcal{D}\down F)$ with objects $(C,D,f)$, where $C$ is an object of $\CC$, $D$ an object of $\mathcal{D}$ and $f$ a morphism $f:D\to FC$. A morphism $(C,D,f)\to (C',D',f')$ between objects of $(\mathcal{D}\down F)$ is a pair $(\alpha,\beta)$, where $\alpha:C\to C'$ in $\CC$ and $\beta:D\to D'$ in $\mathcal{D}$ such that
\begin{equation*}
\xymatrix{D\ar[rr]^{\beta}\ar[dd]_{f} &&  D'\ar[dd]^{f'}\\
\\
 FC\ar[rr]_{F\alpha} && FC'}
\end{equation*}

Let $\P$ be a downwards directed poset and take $X\subseteq\P$ such that $\up X\neq\P$. Let $\CC=\Sets$, $\mathcal{D}=\Sets^{X^\op}$, and $F=\Delta_{X^\op}$. Let the functor $\Sets\to\Sets^{X^\op}$ be defined as $\Delta_{X^\op}(A)(x)=A$ for each $x\in X$ and $A\in\Sets$. Then the category obtained by Artin glueing along $\Delta_{X^\op}$ consists of objects $(A,G,f)$, where $G\in\Sets^{X^\op}$, $A\in\Sets$ and $f$ a natural transformation $G\to\Delta_{X^\op}(A)$. This corresponds exactly to a presheaf $G_0\in\Sets^{X_0^\op}$, where $X_0$ is the set obtained by adding a least element to $X$. The correspondence is given by defining $G_0(x)=G(x)$ if $x\in X$ and $G_0(0)=A$, $G_0(x\leq y)=G(x\leq y)$ if $x\in X$ and $G_0(0\leq y)=f_y$ for each $y\in X$. Moreover, if $(\alpha,\beta):(A,G,f)\to(A',G',f')$ is a morphism, this corresponds to a natural transformation $\gamma:G_0\to G_0'$ defined by $\gamma_x=\beta_x$ if $x\in X$ and $\gamma_0=\alpha$. Since $f'\circ\beta=\Delta_{X^\op}(A)\circ f$, this is indeed a natural transformation. Conversely, given an element $H\in\Sets^{X_0^\op}$, we obtain an element $(H(0),H|_{X},f)$ in $(\Sets\down\Delta_{X^\op})$, the category obtained by Artin glueing along $\Delta_{X^\op}$, where $f:H|_X\to\Delta_{X^\op}(H(0))$ is defined as $f_x=H(0\leq x)$ for each $x\in X$. It turns out that these correspondences constitute an equivalence between $(\Sets\down\Delta_{X^\op})$ and $\Sets^{X_0^\op}$. In Corollary \ref{cor:KXsheaves}, we already found that the latter is equivalent to $\Sh(\P,K_X)$, hence Artin glueing yields exactly the $K_X$-sheaves.

It might be interesting to see what other topoi we obtain by Artin glueing along other functors, by glueing different categories together. For instance, the $K_X$ sheaves are obtained by glueing $\Sets^{X^\op}$ and $\Sets$, which corresponds to adding an element $0$ below $X$. What happens if we replace $\Sets$ by $\Sets^{Y^\op}$ for some poset $Y$? Can we glue this presheaf category together with $\Sets^{X^\op}$. Does this correspond to a topos equivalent to $\Sh(\P,J)$ for some Grothendieck topology $J$ on $\P$? If so, how can we guess what form $J$ should have? Is $J$ a mixture of two already known Grothendieck topologies $J_1$ and $J_2$ just like $K_X$ is a mixture of $J_X$ and $J_\atom$? In any case, Artin glueing might be an interesting technique for further research.

\appendix

\section{The poset of commutative unital C*-subalgebras of a C*-algebra}
In this section, we will always assume that all C*-algebras are unital and that every *-homomorphism preserves the unit.

\begin{definition}
 Let $\A$ be a C*-algebra with unit $1_\A$. We denote the set of its commutative C*-subalgebras containing $1_\A$ by $\CCC(\A)$.
\end{definition}

\begin{lemma}\cite[Proposition 14]{BH}.
Let $\A$ be a C*-algebra. Then the following statements are equivalent.
\begin{enumerate}
 \item[(i)] $\A$ is commutative;
 \item[(ii)] $\CCC(\A)$ contains a greatest element;
 \item[(iii)] $\CCC(\A)$ is a complete lattice.
\end{enumerate}
\end{lemma}
\begin{proof}
Assume that $\A$ is commutative, hence $\A$ is clearly the greatest element of $\CCC(\A)$.

Now Assume that $\CCC(\A)$ contains a greatest element. This is equivalent with saying that $\CCC(\A)$ contains the empty meet ($=\A$). Since in any case $\CCC(\A)$ contains all non-empty meets, this implies that $\CCC(\A)$ contains all meets. As a consequence, $\CCC(\A)$ is a complete lattice.

Finally, assume that $\CCC(\A)$ is a complete lattice. Denote the join operation on $\CCC(\A)$ by $\vee$. Let $a,b\in\A$. Then $a$ can be written as linear combination of two self-adjoint elements $a_1$ and $a_2$. Then $C^*(a_i,1_\A)$, the closure of all polynomials in $a_i$, is a unital C*-subalgebra of $\A$ for each $i\in\{1,2\}$. Let $C_a=C^*(a_1,1_\A)\vee C^*(a_2,1_\A)$. Then $C_a\in\CCC(\A)$ and since $C^*(a_1,1_\A), C^*(a_2,1_\A)\subseteq C_a$, it follows that $a_1,a_2\in C_a$. Hence $a\in C_a$. In the same way, there is a $C_b\in\CCC(\A)$ such that $C_b$. Now, $a,b\in C_a\vee C_b$ for $C_a,C_b\subseteq C_a\vee C_b$, and since $C_a\vee C_b$ is commutative, it follows that $ab=ba$. We conclude that $\A$ is commutative.
\end{proof}

If we denote the category of unital C*-algebras by $\mathbf{uCStar}$ and the category of posets by $\mathbf{Poset}$, then $\CC:\mathbf{uCStar}\to\mathbf{Poset}$ can be made into a functor \cite[Proposition 5.3.3]{Heunen}.
\begin{lemma}\label{lem:CAisinvariantforA}
$\CC:\mathbf{uCStar}\to\mathbf{Poset}$ becomes a functor if we define $\CC(f):\CC(\A)\to\CC(\B)$ for *-homomorphisms $f:\A\to\B$ between C*-algebras $\A$ and $\B$ by  $C\mapsto f[C]$. Moreover, $\CC(f)$ is injective if $f$ is injective, and surjective if $f$ is surjective. As a consequence, if $f$ is a *-isomorphism, then $\CC(f)$ is an order isomorphism.
\end{lemma}
\begin{proof}
Let $C\in\CC(\A)$. Then the restriction of $f$ to $C$ is a *-homomorphism with codomain $\B$. By \cite[Theorem 4.1.9]{KR1}, it follows that $f[C]$ is a  C*-subalgebra of $\B$. Since $f$ is multiplicative, it follows that $f[C]$ is commutative. Clearly $f[C]$ is a *-subalgebra of $\B$, so $f[C]\in\CC(\B)$. Moreover, we have $f[C]\subseteq f[D]$ if $C\subseteq D$, so $\CC(f)$ is an order morphism. If $f:\A\to\B$ and $g:\B\to\mathcal{D}$ are *-homomorphisms, then $\CC(g\circ f)(C)=g\circ f[C]=g[f[C]]=\CC(g)\circ\CC(f)$, and if $\mathrm{Id}_\A:\A\to\A$ is the identity morphism, then $\CC(\mathrm{Id_\A})=\mathrm{Id}_{\CC(\A)}$, the identity morphism of $\CC(\A)$. Thus $\CC$ is indeed a functor.

Assume that $f$ is injective. Then $f^{-1}[f[C]]=C$ for each $C\in\CC(\A)$. If $\CC(f)(C)=\CC(f)(D)$, then $f[C]=f[D]$, hence $C=D$. Now assume that $f$ is surjective and let $C\in\CC(\B)$. By the linearity and multiplicativity of $f$ it follows that $f^{-1}[C]$ is a commutative subalgebra of $\A$. Since $C$ is closed in $\B$ and $f$ is a *-homomorphism, hence continuous, it follows that $f^{-1}[C]$ is closed in $\A$, so $f^{-1}[C]$ is a commutative C*-subalgebra of $\A$. Hence $f^{-1}[C]\in\CC(\A)$, and by the surjectivity of $f$, we have $f[f^{-1}[C]]=C$. We conclude that $\CC(f)$ is surjective.
\end{proof}

Notice that a maximal commutative C*-subalgebra of $\A$ is precisely a maximal element of $\CCC(\A)$.
\begin{lemma}\label{lem:finitemaximalimpliesfinitedimension}
 Let $\A$ be a C*-algebra. Then $\CCC(\A)$ contains a maximal element. Moreover, $\A$ is finite dimensional if and only if there is a maximal element of $\CCC(\A)$ that is finite dimensional.
\end{lemma}
\begin{proof}
If $\CCC(\A)$ is Noetherian, it automatically contains a maximal element. If $\CCC(\A)$ is not Noetherian, we need the Lemma of Zorn. Let $\{C_i\}_{i\in I}$ be a chain of commutative C*-subalgebras of $\A$ containing $1_\A$. Let $C=\bigcup_{i\in I}C_i$. Then $C$ is a subalgebra of $\A$ containing $1_\A$, which is commutative, since if $x,y\in C$, there is a $j\in I$ such that $x,y\in C_j$. Hence $xy=yx$ by the commutativity of $C_j$. Then $\overline{C}$ is a C*-subalgebra of $\A$, which is commutative. Indeed, let $x,y\in\overline{C}$. Then there are sequences $\{x_n\}_{n\in\N},\{y_n\}_{n\in\N}$ in $C$ converging to $x$ and $y$, respectively. Then $\{x_n\}_n$ is bounded, since it is Cauchy. Indeed, if $\epsilon>0$, there is an $N\in\N$ such that $\|x_n-x_m\|<\epsilon$ for each $n,m\geq N$. Then for each $n\geq N$, we find $$\|x_n\|=\|x_n-x_N+x_N\|\leq\|x_n-x_N\|+\|x_N\|<\epsilon+\|x_N\|.$$ Let $K=\max_{1\leq k\leq N}\{\epsilon+\|x_k\|\}$. Then clearly $\|x_n\|\leq K$ for each $n\in\N$. It follows now that $x_ny_n\to xy$ if $n\to\infty$, since
\begin{eqnarray*}
\|x_ny_n-xy\| & = & \|x_ny_n-x_ny+x_ny-xy\|\leq\|x_ny_n-x_ny\|+\|x_ny-xy\|\\
 & \leq & \|x_n\|\|y_n-y\|+\|x_n-x\|\|y\|\leq K\|y_n-y\|+\|x_n-x\|\|y\|\to 0
\end{eqnarray*}
if $n\to\infty$. But since $x_ny_n=y_nx_n$ for each $n\in\N$, it follows that $x_ny_n\to yx$ as well. Since $\overline{C}$ is Hausdorff, it follows that $xy=yx$.
It follows that for every chain in $\CCC(\A)$, there is an upper bound for the chain in $\CCC(\A)$. By Zorn's Lemma it follows that $\CCC(\A)$ contains a maximal element.

Assume that $\A$ is finite dimensional. Then all subalgebras of $\A$ must be finite dimensional as well, hence all maximal elements of $\CCC(\A)$ are finite dimensional. Conversely, assume that there is a maximal element of $\CCC(\A)$ of finite dimension. By \cite[Exercise 4.6.12]{KR1}, whose solution can be found in \cite{KR3}, it follows that $\A$ must have finite dimension.
\end{proof}

\begin{definition}
 Let $X$ be a topological space with topology $\O(X)$. Then $X$ is called \emph{Noetherian}\index{Noetherian topological space} if $\O(X)$ ordered by inclusion is Noetherian.
\end{definition}

\begin{lemma}\cite[Exercise I.1.7]{Hartshorne}\label{lem:Noetheriancharacterization}
 Let $X$ be a topological space. Then $X$ is Noetherian if and only if every subset of $X$ is compact. Moreover, if $X$ is Noetherian and Hausdorff, then $X$ must be finite.
\end{lemma}
\begin{proof}
 Assume that $X$ is Noetherian, let $Y\subseteq X$ and $\mathcal{U}$ a cover of $Y$. Notice that $\mathcal{U}$ cannot be empty. Now replace the cover by $\mathcal{V}$, where each element $V\in\mathcal{V}$ is the union of finite elements of $\mathcal{U}$. Since $\mathcal{V}$ contains $\mathcal{U}$, it must also cover $Y$. Moreover, $\mathcal{V}$ must be upwards directed. Indeed, let $V_1,V_2\in\mathcal{V}$. Then both $V_1$ and $V_2$ can be written as union of finite elements of $\mathcal{U}$, so $V_1\cup V_2$ can be written as union of finite elements of $\mathcal{U}$ as well. Hence $V_1\cup V_2\in\mathcal{V}$, and since $V_1,V_2\subseteq V_1\cup V_2$, it follows that $\mathcal{V}$ is upwards directed. Since $\mathcal{V}$ is a non-empty upwards directed subset of $\O(X)$, which is a Noetherian poset, it follows that $\mathcal{V}$ contains a greatest element $V$. Since all other elements of $\mathcal{V}$ are contained in $V$, it follows that $Y\subseteq\bigcup\mathcal{V}=V$. Since $V$ as an element of $\mathcal{V}$ can be written as as union of finite elements of $\mathcal{U}$, it follows that $\mathcal{U}$ has a finite subcover of $Y$. So $Y$ is compact.

Conversely, assume that every subset of $X$ is compact. Let $\mathcal{V}$ be a non-empty upwards directed subset of $\O(X)$. Then $V=\bigcup\mathcal{V}$ is a subset of $X$, which is compact and covered by $\mathcal{V}$. Hence there must be $V_1,\ldots V_n\in\mathcal{V}$ for some $n\in\N$ such that $V=\bigcup_{i=1}^nV_i$. Since $\mathcal{V}$ is upwards directed, it follows that there is a $V'\in\mathcal V$ such that $V_1,\ldots,V_n\subseteq V'$. Hence $V\subseteq V'$, but since $V$ is the union of $\mathcal V$, we must also have $V'\subseteq V$. So $V=V'$ and it follows that $V\in\mathcal V$. We conclude that $\mathcal V$ has a greatest element, namely $V$, so $X$ must be Noetherian.

Finally, assume that $X$ is Noetherian and Hausdorff. For each $x\in X$, it follows that $X\setminus\{x\}$ is compact, and since compact subsets of Hausdorff spaces are closed, it follows that $\{x\}$ is open. So $X$ is discrete. As a Noetherian space, $X$ must be compact itself. Since infinite discrete spaces cannot be compact, it follows that $X$ must be finite.
\end{proof}

\begin{lemma}\label{lem:infinitedimensionalcommutativealgebra}
 Let $\A$ be a commutative C*-algebra with infinite dimension. Then $\CCC(\A)$ contains both a descending chain that does not terminate and an ascending chain that does not terminate.
\end{lemma}
\begin{proof}
By Gel'fand duality, there is a compact Hausdorff space $X$ such that $\A\cong C(X)$. Then $X$ cannot have a finite number of points, since if $X$ would have a finite number of points, say $n$, then the Hausdorff property implies that $X$ must be discrete, so $\A\cong C(X)\cong\C^n$.

We construct a descending chain in $\CCC(A)$ as follows. By the Axiom of Dependent Choice, we can find $\{x_1,x_2,x_3,\ldots\}\subseteq X$. Let $C_n=\{f\in C(X):f(x_1)=\ldots=f(x_n)\}$ for each $n\in\N$. Since $C(X)$ seperates points of $X$, we find that $C_1\supseteq C_2\supseteq C_3\supseteq\ldots$, but $C_n\neq C_m$ if $n\neq m$. Hence $\CCC(C(X))$ contains a descending chain that does not stabilize.

We construct an ascending chain in $\CCC(\A)$ as follows. First we notice that since $X$ is infinite and Hausdorff, Lemma \ref{lem:Noetheriancharacterization} implies that $X$ is not Noetherian. So there is an ascending chain $O_1\subseteq O_2\subseteq\ldots$ in of open subsets of $X$ that does not stabilize. For each $i\in\N$, let $F_i=X\setminus O_i$. Then $F_1\supseteq F_2\supseteq\ldots$ is a descending chain of closed subsets of $X$, which does not stabilize. For each $i\in I$ define
 $$C_i=\{f\in C(X):f\ \mathrm{is\ constant\ on\ }F_i\}.$$ Then $C_i$ is a C*-subalgebra of $C(X)$ and if $i\leq j$, we have $F_i\supseteq F_j$, so $C_i\subseteq C_j$. Moreover, if $i<j$ and $F_i\neq F_j$, then there is some $x\in F_i$ such that $x\notin F_j$. By Urysohn's Lemma, there is an $f\in C(X)$ such that $f(x)=0$ and $f(y)=1$ for each $y\in F_j$. Hence $f\in C_j$, but $f\notin C_i$. It follows that $C_1\subseteq C_2\subseteq\ldots$ is an ascending chain that does not stabilize, since the descending chain of the $F_i$ does not stabilize.

So $\CCC(C(X)$ contains an ascending chain and a descending chain, which both do not stabilize. Since $\A\cong C(X)$, it follows by Lemma \ref{lem:CAisinvariantforA} that $\CCC(\A)$ as well contains an ascending chain and a descending chain, which do not stabilize.
\end{proof}

\begin{proposition}\label{prop:ArtinianandNoetherianisfinitedimensional}
 Let $\A$ be a C*-algebra. Then $\A$ is finite dimensional if and only if $\CCC(\A)$ is Artinian if and only if $\CCC(\A)$ is Noetherian.
\end{proposition}
\begin{proof}
 Assume that $\A$ is finite dimensional with dimension equal to $n$. Then each C*-subalgebra of $\A$ must be finite dimensional as well. Assume that $C_1\supseteq C_2\supseteq C_3\supseteq\cdots $ is a chain in $\CCC(\A)$ that does not stabilize. Then the dimension of each $C_n$ must be finite and $\dim(C_1)\geq\dim(C_2)\geq\dim(C_3)\geq\ldots$ must be a chain in $\N$, which does not stabilize. However, since $\N$ is Artinian, this is not possible, so $\CCC(\A)$ must be Artinian as well.

Since $\A$ has finite dimension $n$, the dimension of all its subalgebras is bounded by $n$. Assume that $C_1\subseteq C_2\subseteq C_3\subseteq\ldots$ is an ascending chain in $\CCC(\A)$. Then $\dim(C_1)\leq\dim(C_2)\leq\ldots$ must be an ascending chain in $\N$ bounded by $n$. So there must be a $k\leq n$ such that $\dim(C_i)=k$ for each $i\geq m$ for some $m\in\N$. From dimensional reasons it follows that $C_i=C_{i+1}$ for each $i\geq m$. We conclude that the chain stabilizes, so $\CCC(\A)$ must be Noetherian.

Assume that $\CCC(\A)$ is Noetherian or Artinian and that $\A$ has infinite dimension. By Lemma \ref{lem:finitemaximalimpliesfinitedimension}, $\A$ contains a maximal commutative C*-subalgebra $\M$, and we repeat that we do not have to use Zorn's Lemma if $\CCC(\A)$ is Noetherian. By the same lemma, $\M$ cannot have finite dimension. By Lemma \ref{lem:infinitedimensionalcommutativealgebra}, it follows that $\CCC(\M)$ contains a descending chain and an ascending chain that both do not stabilize. Since $\M\subseteq\A$, Lemma \ref{lem:CAisinvariantforA} assures that $\CCC(\M)\subseteq\CCC(\A)$. Hence $\CCC(\A)$ contains a descending chain and an ascending chain, which both do not stabilize. However, if $\CCC(\A)$ is Noetherian, the existence of the ascending chain yields a contradiction, whereas if $\CCC(\A)$ is Artinian, the existence of the descending chain yields a contradiction. So if $\CCC(\A)$ is either Noetherian or Artinian, $\A$ must be finite dimensional.
\end{proof}
Even in the case of $\CCC(\A)$ is Noetherian, we still have to use a weaker version of the Axiom of Choice, namely the Axiom of Dependent Choice. This is because we used different characterizations of Noetherian posets (see Lemma \ref{lem:equivalentdefinitionsNoetherian}), and in order the show that these characterizations are equivalent, the Axiom of Dependent Choice has to be used.

\section{Grothendieck topologies and Locale Theory}
In this section we aim to describe Grothendieck topologies on a small poset category $\P$ completely in terms of down-sets of $\P$. The set $\D(\P)$ of down-sets of $\P$ is a locale (see the Appendix), where the meets are intersections, and the joins are unions. Later on, we will describe Grothendieck topologies completely in terms of the locale $\D(\P)$. We begin by reviewing the localic concepts in the case of the locale $\D(\P)$, which will turn out to be equivalent to the notion of a Grothendieck topology. These results were first presented in \cite[Proposition III.4.2]{J&T}. We also refer to \cite[Chapter IX]{M&M} and \cite{EGP} for other sources of the material in this section.

\begin{warning}
Since the objects of the category of locales are frames (see Appendix \ref{Order Theory}), we will sometimes refer to locales as frames.
\end{warning}

\begin{definition}
 Let $\P$ be a poset. Then the collection $\D(\P)$ of down-sets of $\P$ is a topology on $\P$, called the \emph{Alexandrov topology}\index{Alexandrov topology} on $\P$.
\end{definition}
Since $\D(\P)$ is closed under arbitrary intersections, the closed sets, which are exactly the up-sets of $\P$, are closed under arbitrary unions. In other words, the up-sets also form a topology, which is sometimes called the upper Alexandrov topology, whereas one sometimes refers to the topology defined in the lemma as the lower Alexandrov topology. Since the collection of open subsets of a topological space is an example of a locale, it follows that $\D(\P)$ is indeed a locale.

A nice property of the Alexandrov topology is that order-theoretic notions can be translated into topological ones.
\begin{lemma}\label{lem:ordermorphismiscontinuity}
 Let $\P_1$ and $\P_2$ posets and $f:\P_1\to\P_2$ a function. Then $f$ is an order morphism if and only if $f$ is continuous with respect to the Alexandrov topology on both $\P_1$ and $\P_2$.
\end{lemma}

\begin{lemma}
 Let $\P$ be a poset equipped with the Alexandrov topology. Then the closed sets are exactly the up-sets $\U(\P)$. Moreover, the closure of a subset $X\subseteq\P$ is exactly $\up X$.
\end{lemma}
\begin{proof}
 We observe that a subset $A$ is a down-set if and only if $A^c$ is an up-set. Let $A$ be a down-set and let $x\in A^c$. If $x\leq y$, we cannot have $y\in A$, since this would imply that $x\in A$, for $A$ is downwards closed. So $y\in A^c$, proving that $A^c$ an up-set. One shows in a similar way that $A$ must be a down-set if $A^c$ is an up-set. Now, let $X\subseteq\P$ be a subset. Let $C$ be a closed set containing $X$, so $C$ is upward closed. Let $y\in\up X$. Then there is an $x\in X$ such that $x\leq y$. Since $x\in C$, and $C$ upward closed, we find that $y\in C$. So $\up X\subseteq C$ for each closed set containing $X$. Clearly $\up X$ is upward closed and contains $X$, so $\up X$ is the smallest closed subset containing $X$, so it is the closure of $X$.
\end{proof}
Also, the property of a poset being Artinian can be described in topological terms. First we have to introduce more topological notions.

\begin{definition}
 Let $X$ be a topological space. Then a non-empty closed subset $F$ of $X$ is called \emph{irreducible}\index{irreducible closed subset} if for each closed sets $F_1$ and $F_2$ such that $F=F_1\cup F_2$, we have $F_1=F$ or $F_2=F$. Then $X$ is called \emph{sober}\index{sober} when each irreducible closed subset is the closure of a unique point of $X$.
\end{definition}
Sobriety turns out to be a property lying between the $T_0$ and Hausdorff separation axioms: each sober space is $T_0$, whereas each Hausdorff space $X$ is sober\cite[Theorem IV.3.3]{M&M}. 

\begin{proposition}\label{prop:equivalencesoberartinian}
Let $\P$ be a poset equipped with the Alexandrov topology. Then $\P$ is Artinian if and only if $\P$ is sober.
\end{proposition}
\begin{proof}
Let $\P$ be Artinian and let $F\subseteq\P$ be an irreducible closed set. Since $F$ is non-empty and $\P$ is Artinian, it should contain a minimal element $p$, so $\up p\subseteq F$. Assume that $F\neq\up p$. Then $F\setminus\up p\neq\emptyset$, so assume $q\in F\setminus\up p$. We cannot have $q=p$, but if $q<p$, we obtain another contradiction since $p\in\min(F)$. Thus $q\in(\down p)^c$. We find that $F=\up p\cup (F\cap(\down p)^c)$, which contradicts the irreducibility of $F$. So we must have $F=\up p$. Since the assignment $\P\to\U(\P)$ given by $p\mapsto\up p$ is clearly injective, it follows that $p$ is unique.

Now assume that $\P$ is sober, so each irreducible closed subset equals $\up p$ for some $p\in\P$. Let $D\subseteq\P$ be non-empty and downwards directed and consider $F=\up D$. Clearly $F$ is closed, but it is also irreducible. Indeed, let $F=F_1\cup F_2$ with $F_1,F_2$ closed. If $F_1\neq F$ and $F_2\neq F$, then $F_1$ and $F_2$ cannot both be empty, and there must be $p_1,p_2\in F$ such that $p_1\notin F_1$ and $p_2\notin F_2$. Since $F=\up D$, there are $d_1,d_2\in D$ such that $p_i\geq d_i$ ($i=1,2$). Since $D$ is downwards directed there is a $d\in D$ such that $d\leq d_1,d_2$. Now, $D\subseteq F_1\cap F_2$, hence $d\in F_i$ for either $i=1$ or $i=2$. Without loss of generality, assume that $d\in F_1$. Since $F_1$ is closed, it is an up-set, so $d\in F_1$ and $d\leq d_1\leq p_1$ implies $p_1\in F_1$, which is a contradiction. Thus $F$ must be irreducible. Since $\P$ is sober, it follows that $F=\up p$ for some $p\in\P$, and since $F=\up D$, we must have that $p$ is the least element of $D$. So $\P$ is Artinian.
\end{proof}

The connection between Grothendieck topologies on a poset $\P$ and properties of the locale $\D(\P)$ is provided by the concept of a \emph{Lawvere-Tierney topology}\index{Lawvere-Tierney topology} on a topos. In case of the topos $\Sets^{\CC^\op}$, where $\CC$ is an arbitrary category, the notion of a Lawvere-Tierney topology turns out to be equivalent to that of a Grothendieck topology on $\CC$. For a definition, we refer to we refer to \cite[Chapter V.1]{M&M}, since it requires some notions from topos theory, which are not necessary for what follows. However, the definition resembles the definition of a nucleus on a locale $L$, which concept turns out to be interesting for us.

\begin{definition}
A {\it nucleus}\index{nucleus} on a locale (or more general, a meet-semilattice) $L$ is a
function $j : L \to L$ that satisfies, for all $a, b \in L$:
 \begin{enumerate}
  \item[(i)] $a \leq j(a)$;
 \item[(ii)] $j\circ j(a) = j(a)$
\item[(iii)]  $j(a \wedge b) = j(a) \wedge j(b)$.
 \end{enumerate}
If $j$ also satisfies the stronger condition
\begin{enumerate}
\item[(iii')]  $j\left(\bigwedge_{i\in I}a_i\right) = \bigwedge_{i\in I}j(a_i)$,
\end{enumerate}
for each family $\{a_i\}_{i\in I}\subseteq L$ with $I$ an index set, then we say that $j$ is \emph{complete}\index{complete nucleus}.
The set of all nuclei on a locale $L$ is denoted by $\Nuc(L)$, which can be ordered as follows. If $j,k\in\Nuc(L)$, then we define $j\leq k$ if $j(a)\leq k(a)$ for each $a\in L$.
\end{definition}
Notice that for each $j,k\in\Nuc(L)$ we have
\begin{equation}\label{eq:orderingNuc}
 j\leq k\implies k=k\circ j.
\end{equation}
Indeed, if $j\leq k$, then $$a\leq j(a)\leq k(a)$$ for each $a\in L$, so $$k(a)\leq k\circ j(a)\leq k\circ k(a),$$ whence $k(a)=k\circ j(a)$.

Since $j$ preverves meets, it is an order morphism.
\begin{lemma}\label{lem:jmonotone}
Let $j$ be a nucleus on a locale $L$. Then $a\leq b$ implies $j(a)\leq j(b)$ for each $a,b\in L$.
\end{lemma}

Just as for each category $\CC$ a Lawvere-Tierney topology on $\Sets^{\CC^\op}$ corresponds bijectively to a Grothendieck topology on $\CC$, Grothendieck topologies on a poset $\P$ are equivalent to nuclei on $\D(\P)$. More general, Lawvere-Tierney topologies on arbitrary toposes are actually nuclei on the subobject classifier $\Omega$ of the topos regarded as frame \cite[p. 481]{Elephant2}. The next proposition is the posetal version of \cite[Theorem V.4.1]{M&M}.

\begin{proposition}\label{prop:GT&Nuc}
 Let $\P$ be a poset and $J$ a Grothendieck topology on $\P$.  Define the function $j_J :
\D(\P) \to \D(\P)$ by
\begin{equation}
j_J(U) := \{p\in \P  : U \cap \down p \in J(p)\}.
\end{equation}
Then the map $J\mapsto j_J$ defines an order isomorphism $\G(\P)\to\Nuc(\D(\P))$, where the inverse $j\mapsto J_j$ is defined pointwise for each $p\in\P$ by
\begin{equation}
J_j(p) := \{S \in \D(\down p): p \in j(S)\}.
\end{equation}
\end{proposition}

\begin{proof}
Let $J$ be a Grothendieck topology on $\P$. We shall show that $j_J$ is a nucleus on $\D(\P)$. It follows from the stability property of the Grothendieck topology
that $j_J(U)$ is a down-set. If $p\in U$, then $\down p\subseteq
U$, so $\down p\cap U=\down p\in J(p)$, whence $p\in j_J(U)$. So
property (i) of a nucleus holds. It directly follows that
$j_J(U)\subseteq j_J\circ j_J(U)$. To show the other inclusion, let $p\in
j_J\circ j_J(U)$, so $j_J(U)\cap\down p\in J(p)$. Now, since
\begin{eqnarray*}
 j_J(U)\cap\down p\ & = & \{q\leq p: q\in j_J(U)\}\\
 & = & \{q\leq p:U\cap\down q\in J(q)\}\\
 & = & \{q\leq p:(U\cap\down p)\cap\down q\in J(q)\},
\end{eqnarray*}
we find by the transitivity property of Grothendieck topologies
(take $S=j_J(U)\cap \down p$ and $R=U\cap\down p$) that $U\cap\down
p\in J(p)$, so $p\in j_J(U)$. Hence $j_J$ satisfies property (ii) of a
nucleus. It follows from Lemma \ref{lem:filter} that $j_J$ satisfies
property (iii) of a nucleus. Indeed, since $J(p)$ is closed under
finite intersections, $U\cap\down p\in J(p)$ and $V\cap\down
p\in J(p)$ imply $U\cap V\cap\down p\in J(p)$, and conversely, since
$U\cap\down p$ and $V\cap\down p$ clearly contain $U\cap V\cap\down
p$, it follows that $U\cap V\cap\down p\in J(p)$ implies that
$U\cap\down p\in J(p)$ and $V\cap\down p\in J(p)$. So we have $p\in
j_J(U)\cap j_J(V)$ if and only if $p\in j_J(U\cap V)$. The assignment $J\mapsto j_J$ is an order morphism: let $K$ be another Grothendieck topology on $\P$ such that $J\leq K$. Thus $J(p)\subseteq K(p)$ for each $p\in\P$. Then for each $A\in\D(\P)$ we find
$$j_J(A)=\{p\in P:A\cap\down p\in J(p)\}\subseteq\{p\in P:A\cap\down p\in K(p)\}=j_K(A),$$
so $j_J\leq j_K$.

Now, let $j$ be a nucleus on $\D(\P)$. We check that $J_j$ is a Grothendieck topology. First note that $p \in
{\down}p \subseteq j({\down} p)$ by property (i) of a nucleus, so
${\down} p \in J_j(p)$. If $S \in J_j(p)$ and $q \leq p$, we have $p \in
j(S)$, so $q \in j(S)$, since $j(S)$ is a down-set. Since $\down q\subseteq j(\down q)$, we find $$q\in j(S)\cap j(\down q)=j(S\cap\down q),$$ so $S\cap\down q\in J_j(q)$, which shows that the stability axiom holds. For
transitivity, suppose that $S \in J_j(p)$, and $R \in \D({\down}p)$ is
such that for all $q \in S$, $R \cap \down q \in J_j(q)$. By
definition of $J_j$, we obtain $p \in j(S)$, and also find that for all $q \in
S$ we have $q \in j(R \cap {\down} q)$. So in particular $q \in
j(R)$, since $j$ is order-preserving (for $j$ preserves meets).
Hence $S \subseteq j(R)$, so $$p \in j(S) \subseteq j(j(R)) = j(R).$$
This shows that $R \in J_j(p)$. Hence, $J_j$ is a well-defined
Grothendieck topology on $\P$. The assignment $j\mapsto J_j$ is an order morphism. Let $k$ be another nucleus on $\D(\P)$ such that $j\leq k$. Thus $j(A)\subseteq k(A)$ for each $A\in\D(\P)$. Then for each $p\in\P$ we obtain
$$J_j(p)=\{S\in\D(\down p):p\in j(S)\}\subseteq\{S\in\D(\down p):p\in k(S)\}=J_k(p),$$
hence $J_j\leq J_k$.

Now, if $J$ is a Grothendieck topology on $\P$, let $J_{j_J}$ be the
Grothendieck topology associated to the nucleus $j_J$. Then, for $S \in \D({\down}p)$, we have $S \in J_{j_J}(p)$ if and
only if $p \in j_J(S)$ if and only if $S \cap {\down} p \in J(p)$ if
and only if $S \in J(p)$.

For $j$ a nucleus on $\D(\P)$, let $j_{J_j}$ the nucleus associated to
the Grothendieck topology $J_j$. Then $p \in
j_{J_j}(U)$ if and only if $U \cap {\down}p \in J_j(p)$ if and only if $$p
\in j(U \cap {\down}p) = j(U) \cap j({\down} p).$$ Since we always
have $p \in j({\down}p)$, this is in turn equivalent to $p \in
j(U)$.
\end{proof}

Under this bijection, the following nuclei on $\D(\P)$ turn out to be equivalent to the Grothendieck topologies in Example \ref{ex:examplesGT}.
\begin{example}
Let $\P$ be a poset. If, for $A\in\D(\P)$, we define
\begin{itemize}
 \item $j_{\mathrm{ind}}(A)=A$;
 \item $j_{\mathrm{dis}}(A)=\P$;
\item $j_{\mathrm{atom}}(A)= \left\{
       \begin{array}{ll}
         \emptyset, & A=\emptyset; \\
         \P, &  A\neq\emptyset,
       \end{array}
     \right.$
\end{itemize}
then $j_{\mathrm{ind}},j_{\mathrm{dis}}:\D(\P)\to\D(\P)$ are nuclei
on $\D(\P)$ and $j_{\mathrm{atom}}:\D(\P)\to\D(\P)$ is a nucleus on
$\D(\P)$ if $\P$ is downwards directed. If we consider the poset
$\P_3$ from the same example, we see that $j_{\mathrm{atom}}$ is not
a nucleus, since
\begin{equation*}
 j_{\mathrm{atom}}(\down y\cap\down z)=j_{\mathrm{atom}}(\emptyset)=\emptyset\neq\P=j_{\mathrm{atom}}(\down y)\cap j_{\mathrm{atom}}(\down z).
\end{equation*}
On the other hand, it is easy to see that $j_{\mathrm{atom}}(A\cap B)=j_{\mathrm{atom}}(A)\cap j_{\mathrm{atom}}(B)$ for each $A,B\in\D(\P)$ if $\P$ is downwards directed.
\end{example}

Lawvere-Tierney topologies are closely related to subtopoi (see \cite[Corollary VII.4.7]{M&M}). In the same way, nuclei on a locale turn out to correspond to sublocales of that locale. Since Grothendieck topologies on $\P$ and nuclei on $\D(\P)$ are equivalent notions, it might as well be interesting to explore these sublocales. Since $\mathbf{Loc}$, the category of locales, is dual to $\mathbf{Frm}$, the category of frames, sublocales correspond to frame quotients. Hence, the duality $\mathbf{Loc}^\op=\mathbf{Frm}$ yields another notion, equivalent to Grothendieck topologies on posets.

For a justification of the definition of sublocales, it is useful to work with frames rather than locales for the following reason. $\mathbf{Frm}$ is a category, which can be defined by a \emph{variety of algebras} \cite[p. 124]{Mac Lane}. Roughly speaking, this means that $\mathbf{Frm}$ is a category that can be described in terms of Universal Algebra, where quotients are quite easy to characterize, namely in terms of congruences, kernels and surjective morphisms. For a treatment of these concepts in the setting of Universal Algebra, we refer to \cite[Chapter II]{BS}. We shall see that the definition of sublocales that we will give below corresponds with these concepts.

There is also a categorical justification for using these concepts. A morphism $m:B\to A$ in a category is called a \emph{monomorphism} if $m\circ f=m\circ g$ implies $f=g$ for any two morphism $f,g:C\to B$. Dually, a morphism $e:A\to B$ is called an \emph{epimorphism} if $f\circ e=g\circ e$ implies $f=g$ for each $f,g:B\to C$. A monomorphism $m$ is called \emph{extremal} if $m=m'\circ e$ with $e$ an epimorphism implies that $e$ is an isomorphism. Dually, an epimorphism $e$ is called extremal if $e=m\circ e'$ with $m$ a monomorphism implies that $m$ is an isomorphism. We call a category \emph{balanced} if every morphism that is both a monomorphism and an epimorphism is an isomorphism. In balanced categories one usually describes subobjects and quotient objects in terms of monomorphisms and epimorphisms, respectively. It is easy to see that if either all monomorphisms or all epimorphisms are extremal, then the category is balanced.

One easily shows that if the category is concrete, i.e., a category equipped with a faithful functor to the category $\Sets$, a morphism that is injective regarded as function in $\Sets$ is automatically a monomorphism, and likewise all surjective morphisms are epimorphisms. Indeed, if $e:A\to B$ is a surjection, and $f,g:B\to C$ are morphisms such that $f\circ e=g\circ e$, then each $b\in B$ is equal to $e(a)$ for some $a\in A$. Hence $$f(b)=f\circ e(a)=g\circ e(a)=g(b),$$ and it follows that $f=g$.

 In $\mathbf{Frm}$, the converse is true: all monomorphisms are injective. However, only the extremal epimorphisms are surjective \cite[Chapter III.1]{PicPul}. Since $\mathbf{Frm}$ is a category, which can be described in terms of Universal Algebra, the isomorphism are exactly the morphisms that are both injective and surjective. Hence a morphism that is both a monomorphism and an epimorphism but not a surjection fails to be an isomorphism. Since there are actually epimorphism that are not extremal \cite[Example IV.6.1.1]{PicPul}, it follows that $\mathbf{Frm}$ is not balanced. For this reason it is natural to restrict to extremal frame epimorphisms (so surjective frame morphisms) in order to describe frame quotients.

The next definition makes use of the Heyting implication of a frame, which is denoted by $\to$, and whose definition can be found in the appendix.
\begin{definition}\label{def:sublocale}
 Let $L$ be a locale. A subset $M\subseteq L$ is called a \emph{sublocale}\index{sublocale} if it is closed under all meets and if $x\to m\in M$ for each $x\in L$ and each $m\in M$\footnote{If we had defined sublocales of a locale $L$ to be subsets that are closed under finite meets and arbitrary joins, we would have obtained the wrong definition, since these subsets describe subframes rather than sublocales.}. We denote the set of all sublocales a locale $L$ by $\Sub(L)$, which can be ordered by inclusion.
\end{definition}
Notice that $1\in M$ if $M$ is a sublocale of $L$, since $1$ is equal to the empty meet. Since $M$ inherits the meet operation from $L$, it also inherits the order from $L$. This follows from $x\leq y$ if and only if $x=x\wedge y$.

\begin{warning}
 We denote sublocales of $\D(\P)$ by the symbol $\M$ instead of the symbol $M$, since we already use capitals in order to denote elements of $\D(\P)$.
\end{warning}

\begin{lemma}
 Let $L$ be a locale and $M\subseteq L$ a sublocale of $L$. Then $M$ becomes a locale if we equip it with the meet operation inherited from $L$ and the join operation $\bigsqcup$ defined by
\begin{equation*}
 \bigsqcup_{i\in I}a_i=\bigwedge\{m\in M:a_i\leq m\ \forall i\in I\},
\end{equation*}
where $\{a_i\}_{i\in I}$ is a family of elements of $M$, with $I$ some index set.
\end{lemma}
\begin{proof}
Since $M$ is closed under arbitrary meets, the meet operation inherited from $L$ is well defined on $M$. By Lemma \ref{lem:completelattice}, the operation $\bigsqcup$ is the join on $M$. We will use the fact that $M$ is closed under the Heyting implication to show that the distributivity law holds. So let $\{a_i\}_{i\in I}\subseteq M$ and $b\in M$. First we observe that for each $i\in I$, $$\{m\in M:a_i\wedge b\leq m\}\supseteq\{m\in M:a_i\leq m\}.$$ Thus if we take the meet of both sets, we obtain
\begin{equation}\label{eq:inequality}
\bigsqcup_{i\in I}(a_i\wedge b)\leq\bigsqcup_{i\in I}a_i.
\end{equation}
Since $a_i\wedge b\leq b$ for each $i\in I$, we also find $\bigsqcup_{i\in I}(a_i\wedge b)\leq b$. Combining this inequality with (\ref{eq:inequality}) gives $$\bigsqcup_{i\in I}(a_i\wedge b)\leq\left(\bigsqcup_{i\in I}a_i\right)\wedge b.$$
For the inequality in the other direction, we first notice that for each $k\in I$ we have $a_k\wedge b\leq\bigsqcup_{i\in I}(a_i\wedge b)$. By definition of the Heyting implication, this is equivalent to $a_k\leq b\to\bigsqcup_{i\in I}(a_i\wedge b)$ for each $k\in I$. Therefore, we obtain $\bigsqcup_{i\in I}a_i\leq b\to\bigsqcup_{i\in I}(a_i\wedge b)$, which is equivalent to $$\left(\bigsqcup_{i\in I}a_i\right)\wedge b\leq\bigsqcup_{i\in I}(a_i\wedge b).$$ We conclude that the distributive law holds, so $M$ is a locale.
\end{proof}

The next proposition connects the notion of a sublocale of a locale $L$ with that of a nucleus on $L$. The proof may be found in \cite[Prop. III.5.3.2]{PicPul}.

\begin{proposition}\label{prop:nucleiHI}
 Let $L$ be a locale. For each nucleus $j$ on $L$, we define the subset $M_j$ of $L$ by
\begin{equation*}
 M_j  =  \{x\in L:j(x)=x\}.
\end{equation*}
Then $M_j=j[L]$. Moreover, the map $\Nuc(L)\to\Sub(L)^\op$ given by $j\mapsto M_j$ is an order isomorphism with inverse $M\mapsto j_M$, where $j_M(a)$ is defined by
\begin{equation*}
 j_M(a)  =  \bigwedge\{m\in M:a\leq m\},
\end{equation*}
where $a\in L$.
\end{proposition}
\begin{proof}
Let $j:L\to L$ be a nucleus. First notice that if $a\in j[L]$, then $a=j(b)$ for some $b\in L$, so $$j(a)=j\circ j(b)=j(b)=a,$$ hence $a\in M_j$. Conversely, if $a\in M_j$, then $a=j(a)$, so $a\in j[L]$. In order to show that $M_j$ is closed under arbitrary meet, let $a=\bigwedge_{i\in I}a_i$, where $a_i\in M_j$ for some index set $I$. By Lemma \ref{lem:jmonotone}, it follows from $a\leq a_k$ for each $k\in I$ that $j(a)\leq j(a_k)$ for each $k\in I$. So $j(a)\leq\bigwedge_{i\in I}j(a_i)$ and since the $a_i\in M_j$, we obtain $$j(a)\leq\bigwedge_{i\in I}a_i=a.$$ The inequality in the other direction always hold by definition of a nucleus, whence $j(a)=a$. We conclude that $a\in M_j$.

Let $b\in M_j$ and $a\in L$. Then $x\leq(a\to b)$ if and only if $x\wedge a\leq b$. By definition of a nucleus, $j(x\wedge a)\leq j(b)$, so $j(x)\wedge j(a)\leq j(b)$. Now, $j(b)=b$, and $a\leq j(a)$, so $j(x)\wedge a\leq b$, which is equivalent with $j(x)\leq(a\to b)$. So if we choose $x=a\to b$, we find $j(x)\leq x$. Since the inequality in the other direction automatically holds by definition of a nucleus, we find $j(x)=x$, so $(a\to b)\in M_j$. To show that the assignment $j\mapsto M_j$ is an order morphism, let $k$ be another nucleus on $L$ such that $j\leq k$, i.e. $j(a)\subseteq k(a)$ for each $a\in L$. If $a\in M_k$, we have $k(a)=a$. Thus $$a\leq j(a)\leq k(a)=a,$$
whence $a=j(a)$. So $M_k\subseteq M_j$.

Let $M$ be a sublocale of $L$. Notice that $M\neq\emptyset$, since the meet of the empty family equals $1$, so $1\in M$. Thus the set $\{m\in M:a\leq m\}$ is non-empty, for it always contain $1$, and since all elements in this set are greater than $a$, its meet is greater than $a$. Hence $a\leq j_M(a)$. If $x\in M$ such that $a\leq x$, then $$\bigwedge\{m\in M:a\leq m\}\leq x,$$ so $a\leq x$ implies $j_M(a)\leq x$. Thus $$\{m\in M:a\leq m\}\subseteq\{m\in M:j_M(a)\leq m\},$$ whence $j_M\circ j_M(a)\leq j_M(a)$. The inclusion in the other direction follows by definition of a nucleus, so $j_M\circ j_M=j_M$.

Now, let $a,b\in L$, then $$\{m\in M:a\leq m\}\subseteq\{m\in M:a\wedge b\leq m\},$$ so we find that $j_M(a\wedge b)\leq j_M(a)$, and similarly, $j_M(a\wedge b)\leq j_M(b)$. Thus $j_M(a\wedge b)\leq j_M(a)\wedge j_M(b)$. For the inequality in the other direction, we remark that the image of $j_M$ is in $M$, since $M$ is closed under intersections. Then $a\wedge b\leq j_M(x)$ for some $x\in L$ is equivalent to $a\leq\big(b\to j_M(x)\big)$. Now $$j_M(a)=\bigwedge\{m\in M: a\leq m\},$$ so $j_M(a)\leq \big(b\to j(x)\big)$, since $b\to j_M(x)\in M$ by the fact that $j_M(x)\in M$ and $M$ is a sublocale. We conclude that $j_M(a)\wedge b\leq j_M(x)$, and since $j_M(a)\wedge b=b\wedge j_M(a)$, we can repeat this argument in order to find $j_M(a)\wedge j_M(b)\leq j_M(x)$. Thus $a\wedge b\leq j_M(a\wedge b)$ implies $j_M(a)\wedge j_M(b)\leq j_M(a\wedge b)$. We show that the assignment $M\mapsto j_M$ defines an order morhism. Let $N$ be another sublocale of $L$ such that $N\subseteq M$. Then for each $a\in L$ we have
$$j_M(a)=\bigwedge\{m\in M:a\leq m\}\leq \bigwedge\{n\in N:a\leq n\}=j_N(a).$$
We conclude that $j_M\leq j_N$.

Finally, we show that $M\mapsto j_M$ and $j\mapsto M_j$ are mutual inverses. We already remarked that the image of $j_M$ is in $M$. So if $x\in M_j$, then by $x=j(x)$, we find $x\in M$.  If $x\in M$, then $$j_M(x)=\bigwedge\{m\in M:m\leq x\}=x,$$ so $x\in M_j$, whence $M_{j_M}=M$. If $a\in L$, then $$j_{M_j}(a)=\bigwedge\{m\in M_j:a\leq m\}=\bigwedge\{m:a\leq m\mathrm{\ and\ }j(m)=m\}.$$ In other words, $j_{M_j}(a)$ is the smallest element which is both a fixed point of $j$ and a upper bound of $a$. Let $a\leq b$ with $b=j(b)$. Then $j(a)\leq j(b)=b$, and since both $a\leq j(a)$ and $j(a)=j\circ j(a)$, we find that $j(a)$ is this smallest fixed point of $j$ which is also an upper bound of $a$. In other words, $j_{M_j}(a)=j(a)$.
\end{proof}

Using this proposition, we can find the sublocales on $\D(\P)$ corresponding to the Grothendieck topologies in Example \ref{ex:examplesGT}.

\begin{example}
Let $\P$ be a poset. Then the following subsets of $\D(\P)$ are sublocales.
\begin{itemize}
 \item $\M_{\mathrm{ind}}=\D(\P)$;
 \item $\M_{\mathrm{dis}}=\{\P\}$;
\item $\M_{\mathrm{atom}}=\{\emptyset,\P\}$, which is only defined if $\P$ is downwards directed.
\end{itemize}
If we consider the poset $\P_3$ from the same example, we see that
$\M_{\mathrm{atom}}$ is not a sublocale. Indeed, if it were, we should have
$A\to\emptyset\in\M_{\mathrm{atom}}$ for each $A\in\D(\P_3)$.
However, we have $\down y\cap\down z\subseteq\emptyset$, so
$y\in\down z\to\emptyset$, whereas $\down y\cap\down
x\nsubseteq\emptyset$, so $x\notin\down z\to\emptyset$. In other
words, $\P\neq\down z\to\emptyset\neq\emptyset$, so
$\M_{\mathrm{atom}}$ cannot be a sublocale. On the other hand,
if $\P$ is downwards directed, it is easy to see that
$A\to\emptyset$ equals $\emptyset$ if $A$ is non-empty and equals $\P$ if
$A$ is empty, so $A\to\emptyset\in\M_{\mathrm{atom}}$ for each
$A\in\D(\P)$.
\end{example}

In Proposition \ref{prop:GTisCompleteLattice} we have seen that the set $\G(\P)$ of Grothendieck topologies on a poset $\P$ is a complete lattice. Since we found an order isomorphism between $\G(\P)$ and $\Nuc(\D(\P))$, we find that the latter set is also a complete lattice. More generally, it turns out that $\Nuc(L)$ is a complete lattice for each locale $L$. We shall show this by proving that $\Sub(L)$ is a complete lattice for each sublocale $L$. Then $\Sub(L)^\op$ is a complete lattice by the remark below Lemma \ref{lem:completelattice}, so $\Nuc(L)$ is complete by the order isomorphism between $\Nuc(L)$ and $\Sub(L)^\op$.

The next proposition is the sublocale analogue of \cite[Proposition II.2.5]{Stone}.
\begin{proposition}\label{prop:subLiscompletelattice}
 Let $L$ be a locale. Then the set $\Sub(L)$ of sublocales of $L$ is a complete lattice with meet operation equal to the intersection operator.
\end{proposition}
\begin{proof}
Let $\{M_i\}_{i\in I}$ be a collection of sublocales of $L$. We show that $M=\bigcap_{i\in I}M_i$ is a sublocale of $L$. Firstly, let $\{a_k\}_{k\in K}$ be a collection of elements of $M$. So $a_k\in M_i$ for each $k\in K$ and each $i\in I$. But then $\bigwedge_{k\in K}a_k\in M_i$ for each $i\in I$. So $\bigwedge_{k\in K}a_k\in M$. Now, let $x\in L$ and $m\in M$. So $m\in M_i$ for each $i\in I$, whence $x\to m\in M_i$ for each $i\in I$. We conclude that $x\to m\in M$, so $M$ is indeed a sublocale of $L$. Now, the intersection indeed defines a meet operation on $\Sub(L)$. Clearly $\bigcap_{i\in I}M_i\subset M_k$ for each $k\in I$, and if $N$ is a sublocale such that $N\subseteq M_i$ for each $i\in I$, then $N\subseteq\bigcap_{i\in I}M_i$, so the intersection is indeed a meet operation on $\Sub(L)$.
\end{proof}

We proceed by introducting the frame-theoretical concept that corresponds with nuclei and sublocales.

\begin{definition}
 Let $F$ be a frame. A \emph{congruence}\index{congruence} $\theta$ on $F$ is an equivalence
relation on $F$ that preserves finite meets and arbitrary joins. That is, if $a_1,a_2,b_1,b_2\in F$ such that $a_1\theta b_1$ and $a_2\theta b_2$, then
$(a_1\wedge a_2)\theta (b_1\wedge b_2)$, and if $\{a_i\}_{i\in I},\{b_i\}_{i\in I}$ are families in $F$ for some index set $I$ such that $a_i\theta b_i$
for each $i\in I$, then we have $\left(\bigvee_{i\in I}
a_i\right)\theta\left(\bigvee_{i\in I}b_i\right)$.

If $\theta$ also preserves arbitrary meets, i.e. if given the same families $\{a_i\}$ and $\{b_i\}$ as above, then
$\left(\bigwedge_{i\in I} a_i\right)\theta\left(\bigwedge_{i\in
I}b_i\right)$, we say that $\theta$ is a \emph{complete frame
congruence}\index{congruence:complete}. The set of all congruences on a frame $F$ is denoted by $\Con(F)$, which can be ordered by inclusion: $\theta\leq\phi$ if and only if $\theta\subseteq\phi$.
\end{definition}

A more general version of the next lemma may be found in \cite[Theorem II.6.8, Theorem II.6.10]{BS}.
\begin{lemma}\label{lem:prehomomorphismtheorem}
Let $g:F\to G$ be a surjective frame morphism. Then $$\ker g=\{(a,b)\in F^2:g(a)=g(b)\}$$ is a congruence. If $\theta$ is a congruence on a frame $F$, then $F/\theta$ is a frame with meet defined by $[a_1]_\theta\wedge[a_2]_\theta:=[a_1\wedge a_2]_\theta$ and join defined by $\bigvee_{i\in I}[a_i]=\left[\bigvee_{i\in I}a_i\right]$, where $a_1,a_2\in F$ and $\{a_i\}_{i\in I}$ is a family of elements of $F$ for some index set $I$. Moreover, the quotient map $q_\theta:F\to F/\theta$ is a surjective frame morphism with kernel $\theta$.
\end{lemma}
\begin{proof}
Let $g$ be a surjective frame morphism and write $\theta_g=\ker g$. Let $a_1,a_2,b_1,b_2\in F$ such that $a_i\theta_g b_i$ for $i=1,2$, i.e., $g(a_i)=g(b_i)$. Then $$g(a_1\wedge a_2)=g(a_1)\wedge g(a_2)=g(b_1)\wedge g(b_2)=g(b_1\wedge b_2),$$
so $(a_1\wedge a_2)\theta_g(b_1\wedge b_2)$. Now let $\{a_i\}_{i\in I},\{b_i\}_{i\in I}$ families of elements of $F$ for some index set $I$. If $a_i\theta_g b_i$, i.e., $g(a_i)=g(b_i)$ for each $i\in I$, we find
$$g\left(\bigvee_{i\in I}a_i\right)=\bigvee_{i\in I}g(a_i)=\bigvee_{i\in I}g(b_i)=g\left(\bigvee_{i\in I}b_i\right),$$
so $\left(\bigvee_{i\in I}a_i\right)\theta_g\left(\bigvee_{i\in I}b_i\right)$. So $\theta_g$ is indeed a congruence. Now let $\theta$ be a congruence on $F$.
 First we check whether the frame operations on $F/\theta$ are well defined. Assume $a_1\theta b_1$ and $a_2\theta b_2$ for some $b_1,b_2\in F$. Then $(a_1\wedge a_2)\theta(b_1\wedge b_2)$, hence
\begin{equation*}
[b_1]_\theta\wedge[b_2]_\theta=[b_1\wedge b_2]_\theta=[a_1\wedge a_2]\theta=[a_1]_\theta\wedge[a_2]_\theta.
\end{equation*}
Now assume that $\{b_i\}_{i\in I}$ is a family in $F$ such that $a_i\theta b_i$ for each $i\in I$. Hence $\left(\bigvee a_i\right)\theta\left(\bigvee b_i\right)$. Hence
\begin{equation*}
\bigvee_{i\in I}[b_i]_\theta=\left[\bigvee_{i\in I}b_i\right]_\theta=\left[\bigvee_{i\in I}a_i\right]_\theta=\bigvee_{\in I}[a_i]_\theta.
\end{equation*}
We conclude that the meet and join operations are well defined on $F/\theta$. We have to check whether they are distributive. So let $b\in F$ and $\{a_i\}_{i\in I}$ a family of elements of $F$ for some index set $I$. Then
\begin{equation*}
 [b]_\theta\wedge\bigvee_{i\in I}[a_i]_\theta=[b]_\theta\wedge\left[\bigvee_{i\in I} a_i\right]_\theta=\left[b\wedge\bigvee_{i\in I} a_i\right]_\theta=\left[\bigvee_{i\in I}(b\wedge a_i)\right]_\theta=\bigvee_{i\in I}[b\wedge a_i]_\theta.
\end{equation*}
We show that the quotient map $q_\theta:F\to F/\theta$ is a frame morphism. Let $a_1,a_2\in F$ and $\{a_i\}_{i\in I}$ a family of elements of $F$. Then
\begin{eqnarray*}
 q_\theta(a_1\wedge a_2) & = & [a_1\wedge a_2]_\theta=[a_1]_\theta\wedge[a_2]_\theta=q_\theta(a_1)\wedge q_\theta(a_2);\\
q_\theta\left(\bigvee_{i\in I}a_i\right) & = & \left[\bigvee_{i\in I} a_i\right]_\theta = \bigvee_{i\in I}[a_i]_\theta=\bigvee_{i\in I}q_\theta(a_i).
\end{eqnarray*}
The quotient map is clearly surjective and since $q_\theta(a)=q_\theta(b)$ if and only if $a\theta b$, we find its kernel is $\theta$.
\end{proof}

Different surjective morphisms can induce the same congruence, therefore we define an equivalence relation on the set of surjective morphisms with domain $F$.
\begin{definition}
 Let $F$ be a frame. Then we say that two surjective frame morphisms $g:F\to G$ and $h:F\to H$ are equivalent if and only if there is an isomorphism $k:G\to H$ such that $k\circ g=h$. We denote the equivalence class of a surjection $g$ by $[g]_E$, and the set of equivalence classes of surjective frame morphisms with domain $F$ by $\E(F)$. The latter can be ordered in a following way. If $g,h\in\E(F)$, we say that $[g]_E\leq [h]_E$ if there is a frame morphism $k$ such that $h=k\circ g$.
\end{definition}
We note that this definition is the opposite of the one that is usual within category theory. Notice that $[g]_E\leq [h]_E$ and $[h]_E\leq[g]_E$ imply that $[g]_E=[h]_E$. Indeed, if there are frame morphisms $k_1,k_2$ such that $h=k_1\circ g$ and $g=k_2\circ h$, we have $$k_1\circ k_2\circ h=k_1\circ g=h,$$
so $(k_1\circ k_2)\circ h=1\circ h$. But since $h$ is a surjection, so an epimorphism, we find that $k_1\circ k_2=1$. In the same way, $k_2\circ k_1=1$. Hence $k_1$ and $k_2$ are each other's inverses, so frame isomorphisms. Thus $[g]_E=[h]_E$.

\begin{lemma}\label{lem:qisordermorphism}
 Let $\theta$ and $\phi$ be congruences on a frame $F$ such that $\theta\leq\phi$. Then $[q_\theta]_E\leq[q_\phi]_E$.
\end{lemma}
\begin{proof}
 Define $k:F/\theta\to F/\phi$ by $[a]_\theta\mapsto[a]_\phi$. This map is well defined, since if $[a]_\theta=[b]_\theta$, then $(a,b)\in\theta\subseteq\phi$, so $[a]_\phi=[b]_\phi$. Let $a,b\in F$. Then $$k([a]_\theta\wedge[b]_\theta)=k([a\wedge b]_\theta)=[a\wedge b]_\phi=[a]_\phi\wedge[b]_\phi=k([a]_\theta)\wedge k([b]_\theta),$$ so $k$ preserves finite meets. If $\{a_i\}_{i\in I}$ is a family of elements of $F$, we find $$k\left(\bigvee_{i\in I}[a_i]\right)=k\left(\left[\bigvee_{i\in I}a_i\right]\right)=\left[\bigvee_{i\in I}a_i\right]_\phi=\bigvee_{i\in I}[a_i]_\phi=\bigvee_{i\in I}k([a_i]_\theta),$$ so $k$ is a frame morphism. By definition of $k$, we have $q_\phi=k\circ q_\theta$, so $[q_\theta]_E\leq[q_\phi]_E$.
\end{proof}

It follows from the First Homomorphism Theorem for frames that two surjections are equivalent if and only if they define the same congruence. We shall state this theorem and give a proof. For a more general version of this theorem in the setting of Universal Algebra we refer to \cite[Theorem II.6.12]{BS}. The proof of the Homomorphism Theorem for lattices is very similar to the Homomorphism Theorem for frames and can be found in \cite[Theorem 6.9]{DP}.

\begin{theorem}[First Homomorphism Theorem for Frames]\label{thm:homomorphismtheoremframes}
 Let $g:F\to G$ be a surjective frame morphism with kernel $\theta$. Then there is an isomorphism of frames $k_g:F/\theta\to G$ such that $g=k_g\circ q_\theta$.
\end{theorem}
\begin{proof}
 We define $k_g$ by $k_g([a]_\theta]=g(a)$ for each $[a]_\theta$ and show that this definition is independent of the choice of the representative of $[a]_\theta$. So assume that $a\theta b$ for some $b\in F$. Since $\theta$ is the kernel of $g$, this is $g(a)=g(b)$, so $k_g([b]_\theta)=g(a)=g(b)=k_g([b]_\theta)$. Moreover, $k_g$ is a frame morphism; let $[a_1]_\theta,[a_2]_\theta\in F/\theta$ and $\{[a_i]_\theta\}_{i\in I}$ a family of elements in $F/\theta$ with index set $I$. Then
\begin{eqnarray*}
 k_g([a_1]_\theta\wedge[a_2]_\theta) & = & k_g([a_1\wedge a_2]_\theta)=g(a_1\wedge a_2)=g(a_1)\wedge g(a_2)=k_g([a_1]_\theta)\wedge k_g([a_2]_\theta);\\
k_g\left(\bigvee_{i\in I}[a_i]_\theta\right) & = & k_g\left(\left[\bigvee_{i\in I}a_i\right]_\theta\right)=g\left(\bigvee_{i\in I}a_i\right)=\bigvee_{i\in I}g(a_i)=\bigvee_{i\in I}k_g([a_i]_\theta).
\end{eqnarray*}
Finally, $k_g$ is bijective. Let $c\in G$, then there is some $a\in F$ such that $c=g(a)$, hence $c=k_g([a]_\theta)$, so $k_g$ is surjective. Furthermore, by definition of $k_g$ and $\theta$ we have $k_g([a]_\theta)=k_g([b]_\theta)$ if and only if $g(a)=g(b)$ if and only if $a\theta b$, which exactly says that $[a]_\theta=[b]_\theta$. Thus $k_g$ is injective.
\end{proof}

\begin{corollary}\label{cor:bijectieEFConF}
Let $F$ be a frame. The map $\E(F)\to\Con(F)$ given by $[g]_E\mapsto\ker g$ is an order isomorphism with inverse $\theta\mapsto[q_\theta]_E$.
\end{corollary}
\begin{proof}
Lemma \ref{lem:qisordermorphism} exactly expresses that $\theta\mapsto[q_\theta]_E$ is an order morphism.
We have to show that $[g]_E\mapsto\ker g$ is well defined. So assume $g:F\to G$ and $h:F\to H$ are equivalent frame morphism. Thus there is an isomorphism $k:G\to H$ such that $k\circ g=h$. Then $g(a)=g(b)$ if and only if $k\circ g(a)=k\circ g(b)$, which is exactly $h(a)=h(b)$. Hence $\ker g=\ker h$. Now, the assignment $[g]_E\mapsto \ker g$ is also order-preserving. Let $[g]_E\leq[h]_E$, so there is a frame morphism $k$ such that $k\circ g=h$. Then $g(a)=g(b)$ implies $h(a)=h(b)$ for each $(a,b)\in F^2$, so $\ker g\subseteq\ker h$.

Let $[g]_E\in\E(F)$. By the Homomorphism Theorem, there is an isomorphism $k_g:F\to F/\ker g$ such that $g=k_g\circ q_{\ker g}$, hence $[g]_E=[q_{\ker g}]_E$. Conversely, let $\theta\in\Con(F)$. Lemma \ref{lem:prehomomorphismtheorem} exactly expresses that $\ker q_\theta=\theta$, which concludes the proof that both maps are each other's inverses.
\end{proof}

The next propositions connect the notions of a frame quotient and a sublocale, and hence they justify the definition of a sublocale. The first step is showing that every sublocale, or equivalently, every nucleus, induces a surjective frame morphism. It turns out that every nucleus is, in fact, already a surjective frame morphism.

\begin{proposition}\cite[Proposition IX.4.3]{M&M}\label{prop:nucleusisframesurjection}
Let $j:L\to L$ a nucleus. If we restrict the codomain of $j$ to $M_j$ and regard $L$ and $M_j$ as frames, then $j:L\to M_j$ becomes a surjective frame morphism.
\end{proposition}
\begin{proof}
Since $M_j=j[L]$, it follows immediately that $j$ is surjective.
The meet operation on $M_j$ is inherited from the meet operation on $L$, and $j$ satisfies $j(a\wedge b)=j(a)\wedge j(b)$, so we only have to show that $j$ preserves arbitrary joins. This is done by introducing another join operation on $M_j$, which should be equal to $\bigsqcup$, since joins on frames are unique.
Let $\{a_i\}_{i\in I}$ be a family of elements of $M_j$ with $I$ some index set. Then $j\left(\bigvee_{i\in I}a_i\right)$ is the join of $\{a_i\}_{i\in I}$. Indeed, since it lies in the image of $j$, it is an element of $M_j$. By Lemma \ref{lem:jmonotone} we have for each $k\in I$ $$a_k\leq j(a_k)\leq j\left(\bigvee_{i\in I}a_i\right),$$ whereas if $b\in M_j$ such that $a_k\leq b$ for each $k\in I$, we obtain $\bigvee_{i\in I}a_i\leq b$, so by the same lemma, we obtain
\begin{equation*}
j\left(\bigvee_{i\in I}a_i\right)\leq j(b)=b,
\end{equation*}
Thus for each $\{a_i\}_{i\in I}\subseteq M_j$ we obtain
\begin{equation}\label{eq:multiplejoinsonM}
j\left(\bigvee_{i\in I}a_i\right)=\bigsqcup_{i\in I}a_i.
\end{equation}
Now let $\{a_i\}_{\in I}$ be a family of elements of $L$ for some index set $I$. Then we have $a_k\leq\bigvee_{i\in I} a_i$ for each $k\in I$, hence by Lemma \ref{lem:jmonotone} it follows that
\begin{equation}\label{eq:monotonejoin}
j(a_k)\leq j\left(\bigvee_{i\in I} a_i\right).
\end{equation}
Thus
\begin{equation*}
 \bigvee_{i\in I}a_i\leq\bigvee_{i\in I}j(a_i)\leq j\left(\bigvee_{i\in I} a_i\right),
\end{equation*}
where the second inequality follows from (\ref{eq:monotonejoin}) and the first since $a_k\leq j(a_k)$ for each $k\in I$. Again by Lemma \ref{lem:jmonotone}, the order is preserved if we let act $j$ on the string of inequalities above. Furthermore, $j\circ j=j$, hence we obtain
\begin{equation*}
 j\left(\bigvee_{i\in I}a_i\right)=j\left(\bigvee_{i\in I} j(a_i)\right)=\bigsqcup_{i\in I}j(a_i),
\end{equation*}
where we used (\ref{eq:multiplejoinsonM}) in the last equality. Thus $j$ preserves joins, so it is a frame morphism.
\end{proof}

\begin{proposition}\label{prop:bijectieEFNucF}
 Let $L$ be a locale. If we regard $L$ as a frame and $g:L\to G$ is a surjective frame morphism, we define $j_g: L\to L$ by $j_g=g_*\circ g$. Then the map $\E(L)\to\Nuc(L)$ given by $[g]_E\mapsto j_g$  is an order isomorphism with inverse $j\mapsto[j]_E$. In particular, for each nucleus $j\in\Nuc(L)$ we have $j_*j=j$.
\end{proposition}
\begin{proof}
By Proposition \ref{prop:nucleusisframesurjection}, $j$ is a frame surjection with domain $L$, so $[j]_E\in\E(L)$. Moreover, (\ref{eq:orderingNuc}) exactly expresses that $j\mapsto [j]_E$ is an order morphism.
In order to show that $[g]_E\mapsto g_*\circ g$ is well defined, let $h$ a frame surjection with domain $L$ such that $[g]_E=[h]_E$. So there is an isomorphism $k$ such that $k\circ g=h$. Then by Lemma \ref{lem:isomorphismisadjoint}, we find $$h_*\circ h=(k\circ g)_*\circ k\circ g=g_*\circ k^{-1}\circ k\circ g=g_*\circ g.$$ Let $g:L\to G$ be a frame morphism (not necessarily surjective). We show that $j_g=g_*\circ g$ is a nucleus. Since $g$ as a frame morphism preserves all joins, it has an upper adjoint by Lemma \ref{lem:adjoint}. By the same lemma, $g_*$ preserves all meets. Since $g$ preserves as a frame morphism finite meets, we obtain $$j_g(a\wedge b)=j_g(a)\wedge j_g(b),$$ for each $a,b\in F$. By Lemma \ref{lem:adjointequivalent}, we find that $$a\leq g_*\circ g=j_g(a),$$ for each $a\in F$. Moreover, by the same lemma, we find
\begin{equation*}
j_g\circ j_g=g_*\circ g\circ g_*\circ g=g_*\circ g=j_g,
\end{equation*}
so $j_g$ is a nucleus. We conclude that the map $[g]_E\mapsto j_g$ is well defined. Now let $[h]_E\in\E(L)$ be such that $[g]_E\leq[h]_E$. So $g:L\to G$ and $h:L\to H$ are frame surjections and there is a frame morphism $k:G\to H$ such that $h=k\circ g$. Then $k_*\circ k:G\to G$ is a nucleus, so for each $a\in L$ we have $g(a)\leq k_*\circ k(g(a)).$ By Lemma \ref{lem:adjointequivalent}, we have $g\circ g_*\circ g=g$, so we obtain $g\circ g_*\circ g(a)\leq k_*\circ k(g(a))$ for each $a\in L$. Then, since $g_*$ is the upper adjoint of $g$, this is equivalent to $g_*\circ g(a)\leq g_*\circ k_*\circ k\circ g(a)$. Hence $j_g(a)\leq j_{k\circ g}(a)$, so $j_g\leq j_h$ and we conclude that the map $[g]_E\mapsto j_g$ is an order morphism.

Finally, we have to show that $[g]_E\mapsto j_g$ and $j\mapsto [j]_E$ are each other's inverses. Let $g:L\to G$ be a frame surjection, which is a representative of $[g]_E$, then by Lemma \ref{lem:adjointequivalent} we find that that $g_*$ is an injection. Hence $(a,b)\in\ker j_g$ if and only if $j_g(a)=j_g(b)$ if and only if $g_*\circ g(a)=g_*\circ g(b)$ if and only if $g(a)=g(b)$ if and only $(a,b)\in\ker g$. So $\ker g=\ker j$, which we will denote by $\theta$. It follows now by the First Homomorphism Theorem that there are isomorphisms $k_g$ and $k_{j_g}$ such that $g=k_g\circ q_\theta$ and $j=k_{j_g}\circ q_\theta$. Then $k=k_{j_g}\circ k_g^{-1}$ is an isomorphism such that $k\circ g=j$, hence $[g]_E=[j_g]_E$.

For the converse, let $j$ be a nucleus. Then $j:M_j\to L$, the restriction of $j$ to $M_j$, is the upper adjoint of $j:L\mapsto M_j$. Indeed, if $a\in L$ and $b\in M_j$, we find that by Lemma \ref{lem:jmonotone} that $a\leq j(b)$ implies $$j(a)\leq j\circ j(b)=j(b)=b.$$ On the other hand, if $j(a)\leq b$, then $a\leq j(b)$, since $a\leq j(a)$ and $b\in M_j$. Thus for each $a\in L$ we find $$j_*\circ j(a)=j\circ j(a)=j(a),$$ hence $j_*j=j$ and the map $j\mapsto[j]_E\mapsto j_*\circ j$ is the identity.
\end{proof}

So we see that different nuclei represent different equivalent classes of $\Eq(L)$. The next proposition assures that each equivalent class can be represented by a unique nucleus.
\begin{proposition}\cite[Chapter III.5]{PicPul}\label{prop:bijectieNucFConF}
 Let $L$ be a locale. Then the map $\Nuc(L)\to\Con(L)$ given by
\begin{equation*}
j\mapsto\theta_j:=\{(a,b)\in L^2:j(a)=j(b)\},
\end{equation*}
is an order isomorphism with inverse $\theta\mapsto j_\theta$ given by
\begin{equation*}
j_\theta(a)=\bigvee\{b\in L:a\theta b\}.
\end{equation*}
\end{proposition}
\begin{proof}
If we compose the map $j\mapsto[j]_E$ of Proposition \ref{prop:bijectieEFNucF} with the map $[g]_E\mapsto \ker g$ of Corollary \ref{cor:bijectieEFConF} we obtain an order isomorhism $\Nuc(L)\to\Con(L)$ by $j\mapsto\ker j$. Clearly, $\ker j$ is equal to $\theta_j$. For its inverse, we compose the inverses of the maps in Proposition \ref{prop:bijectieEFNucF} and Corollary \ref{cor:bijectieEFConF}, which are the maps $\theta\mapsto[q_\theta]_E$ and $[g]_E\mapsto g_*\circ g$. The inverse of $j\mapsto\ker j$ is given by $\theta\mapsto j_\theta$, where $j_\theta=(q_\theta)_*\circ q_\theta$.
We will explicitly calculate $j_\theta$. By Lemma \ref{lem:frameadjoint}, we have for each $[b]_\theta\in L/\theta$
\begin{equation*}
(q_\theta)_*([b]_\theta)=\bigvee\{c\in L:q_\theta(c)\leq[b]_\theta\},
\end{equation*}
hence for each $a\in L$
\begin{eqnarray*}
 j_\theta(a) & = & (q_\theta)_*\circ q_\theta(a)\\
 & = & \bigvee\{c\in L:q_\theta(c)\leq q_\theta(a)\}\\
 & = & \bigvee\{c\in L:q_\theta(a)\wedge q_\theta(c)=q_\theta(c)\}\\
 & = & \bigvee\{c\in L: (a\wedge c)\theta c\},
\end{eqnarray*}
where the last equality follows since $q_\theta$ is a frame morphism, so $q_\theta(a)\wedge q_\theta(c)=q_\theta(a\wedge c)$.
This can be simplified, hence let $x=\bigvee\{b\in L:a\theta b\}$. Since $a\theta b$
implies $a=(a\wedge a)\theta( a\wedge b)$, we find $\{b:a\theta
b\}\subseteq\{b:b\theta(a\wedge b)\}$, so $x\leq j_\theta(a)$. On
the other hand, we have
\begin{equation*}
j_\theta(a)=\bigvee\{b\in L:b\wedge(a\wedge
b)\}\theta\bigvee\{a\wedge b:b\wedge(a\wedge b)\}=a\wedge
j_\theta(a)=a,
\end{equation*}
since $a\leq j_\theta(a)$. But this implies that
$j_\theta(a)\in\{b\in L:a\theta b\}$, so $j_\theta(a)\leq x$. We
conclude that
\begin{equation*}
j_\theta(a)=\bigvee\{b\in L:a\theta b\}.
\end{equation*}
\end{proof}

Since we have obtained an equivalence between congruences and nuclei, we can calculate the congruences on $\D(\P)$, which in turn correspond to the Grothendieck topologies of Example \ref{ex:examplesGT}.
\begin{example}
Let $\P$ be a poset. If we define
\begin{itemize}
 \item $\theta_{\mathrm{ind}}=\{(A,A):A\in\D(\P)\}$;
 \item $\theta_{\mathrm{dis}}=\D(\P)^2$;
\item $\theta_{\mathrm{atom}}=\{(\emptyset,\emptyset)\}\cup(\D(\P)\setminus\{\emptyset\})^2$,
\end{itemize}
then $\theta_{\mathrm{ind}}$ and $\theta_{\mathrm{dis}}$ are frame
congruences on $\D(\P)$, and $\theta_{\mathrm{atom}}$ is a frame
congruence if $\P$ is downwards directed. If we consider $\P_3$ from
Example \ref{ex:examplesGT}, which is not downwards directed, we
see that $\down y\theta_{\mathrm{atom}}\P_3$ and $\down
z\theta_{\mathrm{atom}}\P_3$, which should imply that
$$\emptyset=(\down y\cap\down z)\theta_{\mathrm{atom}}\P_3,$$ which is
not true, so $\theta_{\mathrm{atom}}$ cannot be a frame congruence on $\D(\P_3)$.
\end{example}

Finally, we collect our results and connect the various frame-theoretical notions with Grothendieck topologies.
\begin{theorem}\label{thm:mainthm}
Let $\P$ be a poset. Then
\begin{align*}
J_\M(p) & = \{S\in\D(\down p):\forall M\in\M(S\subseteq M\implies
p\in M)\}; & \M\in\Sub(\D(\P))\\
J_j(p) & = \{S\in\D(\down p):p\in j(S)\}; & j\in\Nuc(\D(\P))\\
J_\theta(p) & = \{S\in\D(\down p):S\theta\down p\}; & \theta\in\Con(\D(\P))
\end{align*}
with $p\in\P$ are well-defined Grothendieck topologies on $\P$.
Also,
\begin{align*}
j_J(A) & =  \{p\in\P:A\cap\down p\in J(p)\}; &J\in\G(\P)\\
j_\theta(A)  & =  \bigcup\{B\in\D(\P):B\theta A\}; & \theta\in\Con(\D(\P))\\
j_\M(A) & =  \bigcap\{A\in\M:A\subseteq M\}; & \M\in\Sub(\D(\P))
\end{align*}
with $A\in\D(\P)$ are well-defined nuclei on $\D(\P)$.
Moreover,
\begin{align*}
\theta_j& =  \ker j=\{(A,B)\in\D(\P)^2:j(A)=j(B)\}; & j\in\Nuc(\D(\P))\\
\theta_J & = \{(A,B)\in\D(\P)^2:\forall p\in\P(A\cap\down p\in J(p)\Longleftrightarrow B\cap\down p\in J(p))\}; &J\in\G(\P)
\end{align*}
are well-defined congruences on $\D(\P)$. Furthermore,
\begin{align*}
\M_j & =  \{A\in\D(\P):j(A)=A\}=j[\D(\P)]; & j\in\Nuc(\D(\P))\\
\M_J & =  \{A\in\D(\P):\forall p\in\P(A\cap\down p\in J(p)\implies
p\in A)\}; &J\in\G(\P)
\end{align*}
are well-defined sublocales of $\D(\P)$.

Finally, the diagram
\begin{equation*}
\xymatrix{\Nuc(\D(\P))\ar[rr]^{j\mapsto\M_j}\ar[dd]_{j\mapsto \theta_j}\ar[ddrr]^{j\mapsto J_j} && \Sub(\D(\P))^\op\ar[dd]^{\M\mapsto J_\M}\\
\\
\Con(\D(\P))\ar[rr]_{\theta\mapsto J_\theta} && \G(\P)}
\end{equation*}
commutes, and all maps in the diagram are order isomorphisms,
with inverses
\begin{eqnarray*}
(j\mapsto J_j)^{-1} & = & J\mapsto j_J;\\
(j\mapsto \theta_j)^{-1} & = & \theta\mapsto j_\theta;\\
(j\mapsto\M_j)^{-1} & = & \M\mapsto j_\M;\\
(\theta\mapsto J_\theta)^{-1} & = & J\mapsto \theta_J;\\
(\M\mapsto J_\M)^{-1} & = & J\mapsto \M_J.
\end{eqnarray*}
\end{theorem}

\begin{proof}
In Proposition \ref{prop:GT&Nuc}, we found that the assignments $j\mapsto J_j$ and $J\mapsto j_J$ are well defined and constitute an order isomorphism. By Proposition \ref{prop:bijectieNucFConF}, we know that the assignments $\theta\mapsto j_\theta$ and $j\mapsto \theta_j$ are well-defined order morphisms, which are each other's inverses.
Hence if we show that $J_{j_\theta}=J_\theta$ for each congruence $\theta$ and $\theta_{j_J}=\theta_J$ for each Grothendieck topology $J$, we prove that the lower triangle of the diagram commutes and consists of order isomorphisms. Moreover, we automatically obtain that $\theta_J$ is a well-defined congruence for each Grothendieck topology and $J_\theta$ is a well-defined Grothendieck topology for each congruence $\theta$.

So let $S\in\D(\down p)$ for $p\in\P$. Then $S\in J_{j_\theta}$ if and only if $p\in j_\theta(S)$ if and only if $\down p\subseteq j_\theta(S)$. Now, this implies $$j_\theta(\down p)\subseteq j_\theta\circ j_\theta(S)=j_\theta(S).$$ Conversely, $\down p\subseteq j_\theta(S)$ implies $j_\theta(\down p)\subseteq j_\theta(S)$, since $\down p\subseteq j_\theta(\down p)$ and $j_\theta\circ j_\theta=j_\theta$. Hence, $S\in J_{j_\theta}(p)$ is equivalent to $j(\down p)\subseteq j(S)$. Since $S\subseteq\down p$, we always have $j(S)\subseteq j(\down p)$, hence $S\in J_{j_\theta}(p)$ if and only if $j_\theta(S)=j_\theta(\down p)$. But this last equality is equivalent with $(S,\down p)\in\theta_{j_\theta}=\theta$, which says that $S\in J_\theta(p)$ by definition of $J_\theta$. Thus $J_{j_\theta}= J_\theta$.

Now let $A,B\in\D(\P)$. Then $(A,B)\in \theta_{j_J}$ if and only if $j_J(A)=j_J(B)$. By definition of $j_J$, we find that the last equality holds if and only if $A\cap\down p\in J(p)$ if and only if $B\cap\down p\in J(p)$ for each $p\in\P$. But this says exactly that $(A,B)\in\theta_J$. So $\theta_{j_J}=\theta_J$.

Proposition \ref{prop:nucleiHI} assures that the maps $j\mapsto \M_j$ and $\M\mapsto j_\M$ are well-defined order morphisms, and are inverses of each other. So we show that the upper triangle of the diagram commutes and consists only of bijections if we show that $J_{j_\M}=J_\M$ for each sublocale $\M$ of $\D(\P)$, which also proves that $J_\M$ is a well-defined topology for each sublocale $\M$. Moreover, if we show that $\M_{j_J}=\M_J$ for each Grothendieck topology $J$, we prove that $\M_J$ is a well-defined sublocale for each Grothendieck topology $J$.

Thus let $S\in\D(\down p)$ for $p\in\P$. Then $S\in J_{j_\M}(p)$ if and only if $$p\in j_\M(S)=\bigcap\{M\in\M:S\subseteq M\}.$$ So $S\in J_{j_\M}(p)$ if and only if $p\in M$ for each $M\in\M$ such that $S\subseteq M$. But this says exactly that $S\in J_\M(p)$, hence $J_{j_\M}=J_\M$.

Finally, let $M\in\D(\P)$. Then $M\in\M_{j_J}$ if and only if $$M=j_J(M)=\{p\in P:M\cap\down p\in J(p)\}.$$ Since we automatically have $p\in M\implies M\cap\down p\in J(p)$ for $p\in M$ implies $M\cap\down p=\down p$, we find $M\in\M_{j_J}$ if and only if $M\cap\down p\in J(p)\Longleftrightarrow p\in M$ for each $p\in\P$. But this exactly says that $\M_{j_J}=\M_J$.
\end{proof}

\begin{corollary}\label{cor:completestructures}
 Let $\P$ be a poset and $J$ a Grothendieck topology on $\P$. If $\theta$ is a congruence on $\D(\P)$ and $j:\D(\P)\to\D(\P)$ a nucleus on $\D(\P)$ which are related to $J$ under the bijections of Theorem \ref{thm:mainthm}, then $J$ is complete as a Grothendieck topology if and only if $j$ is complete as a nucleus if and only if $\theta$ is a complete congruence.
\end{corollary}
\begin{proof}
 We have $J=J_\theta$, $\theta=\theta_j$ and $j=j_J$. Assume $\theta$ is complete and let $p\in\P$ and $\{S_i\}_{i\in I}\subseteq J_{\theta}(p)$ a family of covers of $p$. This means that $S_i\theta\down p$ for each $i\in I$, so $\left(\bigcap_{i\in I}S_i\right)\theta\down p$, since $\theta$ is complete. Thus $\bigcap_{i\in I}S_i\in J_\theta(p)$, so $J$ is complete. Now assume $J$ is complete, and let $\{A_i\}_{i\in I}\subseteq\D(\P)$. We obtain
\begin{equation*}
 j_J\left(\bigcap_{i\in I}A_i\right)=\left\{p\in\P:\bigcap_{i\in I}A_i\cap\down p\in J(p)\right\}.
\end{equation*}
By Lemma \ref{lem:filter}, we have $A_i\cap\down p\in J(p)$ for each $i\in I$ if $\bigcap_{i\in I}A_i\cap\down p\in J(p)$. On the other hand, we have $\bigcap_{i\in I}A_i\cap\down p\in J(p)$ if $A_i\cap\down p\in J(p)$ for each $i\in I$, since $J$ is complete. Hence
\begin{eqnarray*}
 j_J\left(\bigcap_{i\in I}A_i\right) & = & \left\{p\in\P:A_i\cap\down p\in J(p)\ \mathrm{for\ each\ }i\in I\right\}\\
 & = & \bigcap_{i\in I}\{p\in\P:A_i\cap\down p\in J(p)\}\\
 & = & \bigcap_{i\in I}j_J(A_i),
\end{eqnarray*}
and we conclude that $j=j_J$ is a complete nucleus. Finally, let $j$ be a complete nucleus. Assume that $I$ is an index set and $\{A_i\}_{i\in I},\{B_i\}_{i\in I}\subseteq\D(\P)$ such that $A_i\theta_j B_i$ for each $i\in I$. This means that $j(A_i)=j(B_i)$ for each $i\in I$, and since $j$ is complete, we find
\begin{equation*}
 j\left(\bigcap_{i\in I}A_i\right)=\bigcap_{i\in I}j(A_i)=\bigcap_{i\in I}j(B_i)=j\left(\bigcap_{i\in I}B_i\right).
\end{equation*}
In other words, $\left(\bigcap_{i\in I}A_i\right)\theta_j\left(\bigcap_{i\in I}B_i\right)$, so $\theta=\theta_j$ is complete.
\end{proof}

As an example we explore how the dense Grothendieck topology translates to under the order isomorphisms of Theorem \ref{thm:mainthm}.
\begin{proposition}\cite[Corollary VI.1.5]{M&M}
 Let $\P$ be a poset, then $j_{\neg\neg}:\D(\P)\to\D(\P)$ defined by $A\mapsto\neg\neg A$ is a nucleus. Moreover, $J_{j_{\neg\neg}}=J_\dense$.
\end{proposition}
\begin{proof}
It follows directly from Lemma \ref{lem:negationidentities} that $j_{\neg\neg}$ is a nucleus.  Recall that $J_j(p)=\{S\in\D(\down p):p\in j(S)\}$ given a nucleus $j:\D(\P)\to\D(\P)$. So assume that $S\in J_{j_{\neg\neg}}(p)$, i.e. $p\in\neg\neg S$. Then $\down p\subseteq (\neg S\to\emptyset)$, so $\down p\cap\neg S=\emptyset$. We want to show that $\down p\subseteq\up S$. So let $q\in\down p$ and assume that $q\notin\up S$. Then $S\cap\down q=\emptyset$, so $q\in\down q\subseteq S\to\emptyset=\neg S$. But this contradicts $\down p\cap\neg S=\emptyset$. So indeed $\down p\subseteq\up S$, hence $S\in J_\dense(p)$. Now assume that $S\in J_\dense(p)$. We have to show that $p\in\neg\neg S$. Assume the converse, then $\down p\nsubseteq(\neg S\to\emptyset)$, so $\down p\cap\neg S\neq\emptyset$. Hence, there is a $q\leq p$ such that $q\in\neg S=S\to\emptyset$. So $\down q\cap S=\emptyset$. But this implies that $q\notin\up S$ contradicting $\down p\subseteq\up S$. So we must have $p\in\neg\neg S$, or equivalently, $S\in J_{j_{\neg\neg}}(p)$.
\end{proof}

\begin{corollary}
Let $\P$ a poset. Then the sublocale corresponding to $J_\dense$ is given by $\M_\dense=\{A\in\D(\P):\neg\neg A=A\}$. Moreover, the corresponding congruence is given by $\theta_\dense=\{(A,B)\in\D(\P)^2:\neg\neg A=\neg\neg B\}$.
\end{corollary}
\begin{proof}
This follows directly from Theorem \ref{thm:mainthm}.
\end{proof}

As an application, we consider the case when $\P$ is a linearly ordered poset. We shall see that the completion of $\P$ corresponds with a Grothendieck topology, and moreover, if $\P=\Z$, we can find all Grothendieck topologies via the route of sublocales. The crucial fact is that $\D(\P)$ is also linearly ordered if $\P$ is linearly ordered. Indeed, if $A\nsubseteq B$, there is a $p\in A$ such that $p\notin B$. Since $B$ is a down-set, we cannot have $p\leq q$ for each $q\in B$. Since $\leq$ is a linear order, we must have $q<p$ for each $q\in B$, so $B\subseteq\down p\subseteq A$.

Knowing that $\D(\P)$ is linearly ordered makes it quite easy to compute the Heyting implication between different down-sets.
\begin{lemma}\label{lem:linordersublocale}
 Let $\P$ be a linearly ordered set. Then a subset $\M\subseteq\D(\P)$ is a sublocale of $\D(\P)$ if and only if it is closed under intersections.
\end{lemma}
\begin{proof}
Let $\M$ be a sublocale of $\D(\P)$. Then by definition of a sublocale it is closed under intersections. Conversely, assume that $\M$ is closed under intersections, let $M\in\M$, and $A\in\D(\P)$. First assume that $A\subseteq M$. Then $p\in A\to M$ if and only if $\down p\subseteq A\to M$ if and only if $\down p\cap A\subseteq M$. Hence
$A\to M=\P$, which is an element of $\M$ by the remark below Definition \ref{def:sublocale}. Now assume that $A\nsubseteq M$. We always have $A\cap M\subseteq M$, so $M\subseteq A\to M$. Assume $p\notin M$. If $p\in A$, then $\down p\subseteq A$, hence $\down p\cap A=\down p$. But $p\in\down p$, so $\down p\cap A\nsubseteq M$. On the other hand, if $p\notin A$, then we have $q<p$ for each $q\in A$, so $A\subseteq\down p$. Thus $\down p\cap A=A\nsubseteq M$. So if $p\notin M$, then $p\notin A\to M$. We conclude that $A\to M=M\in\M$.
\end{proof}

A well-known construction is the \emph{Dedekind-MacNeille completion}\index{Dedekind-MacNeille completion} of a poset $\P$, which is defined as the set $\M_{dm}$ of fixed points of the function $c:\D(\P)\to\D(\P)$, which is defined as follows. If $A\in\D(\P)$, we define $A^u=\{p\in P:a\leq p\ \forall a\in A\}$, the set of upper bounds of $A$, whereas $A^l=\{p\in\P:p\leq a\ \forall a\in A\}$, the set of all lower bounds of $A$. Then we define $c(A)=A^{ul}$. One can show that $(\M_\mathrm{dm},\subseteq)$ is a complete lattice, and the map $\phi:\P\to\M_\mathrm{dm}$ given by $p\mapsto\down p$ is an embedding that preserves all existing joins and meets in $\P$. For details, we refer to \cite[7.38]{DP}.
One can easily show that $c$ is a so-called \emph{closure operator}\index{closure operator} on $\D(\P)$, which means that for each $A,B\in\D(\P)$
\begin{enumerate}
  \item[(i)]  $A\subseteq c(A)$;
 \item[(ii)] $c\circ c(A)=c(A)$;
 \item [(iii)] $A\subseteq B$ implies $c(A)\subseteq c(B)$.
\end{enumerate}
So $c$ is `almost' a nucleus. We note that literature nonetheless refers to nuclei as closure operators. Now, if $\P$ is linearly ordered, $c$ is a nucleus, since we found that $\D(\P)$ is also linearly ordered. Then, if $A,B\in\D(\P)$, we have either $A\subseteq B$ or $B\subseteq A$. Assume without loss of generality that $A\subseteq B$. Then $A=A\cap B$, but from (iii) in the definition of a closure operator, we also have $c(A)=c(A)\cap c(B)$. It follows that $c(A\cap B)=c(A)=c(A)\cap c(B)$.

Since $\M_\mathrm{dm}=\{A\in\D(\P):c(A)=A\}$, and $c$ is a nucleus, we see
that the Dedekind-MacNeille completion is a sublocale of
$\D(\P)$. Under the order isomorphisms of Theorem \ref{thm:mainthm}, we see that the nucleus $c$ and the sublocale $\M_\mathrm{dm}$ correspond with the Grothendieck topology $$J_c(p)=\{S\in\D(\down
p):p\in S^{ul}\}$$ and the congruence $$\theta_c=\{(A,B)\in\D(\P)^2:A^{ul}=B^{ul}\}.$$

If we consider $\P=\Z$ with the usual order, notice that $\D(\Z)=\{\Z,\emptyset\}\cup\{\down n:n\in\Z\}$, which is ordered by inclusion. Moreover, we have $n\leq m$ if and only if $\down n\subseteq\down m$ if and only if $\down n\cap\down m=\down n$. Let $Y$ be a non-empty subset of $\Z$, then $Y$ contains a minimum if and only if it is bounded from below. Hence
\begin{equation}\label{eq:intersectionsubsofDZ}
 \bigcap\{\down y:y\in Y\} =  \left\{
       \begin{array}{ll}
         \down\min(Y), & Y\ $is\ bounded\ from\ below$; \\
         \emptyset, &  $otherwise$.
       \end{array}
     \right.
\end{equation}

\begin{lemma}
Let $Y$ be a subset of $\Z$ and define
\begin{eqnarray*}
 \M_Y & = &\{\Z,\emptyset\}\cup\{\down p:p\in Y\}\\
\mathcal{N}_Y & = & \{\Z\}\cup\{\down y:y\in Y\}.
\end{eqnarray*}
Then $\M_Y$ is a sublocale of $\D(\P)$ and $\mathcal{N}_Y$ is a sublocale if and only if $Y$ is bounded from below. Conversely, every sublocale of $\D(\P)$ is either equal to $\M_Y$ or to $\mathcal{N}_Y$ for some subset $Y$ of $\Z$.
\end{lemma}
\begin{proof}
 Clearly, each sublocale $\M$ of $\D(\Z)$ shoud be of the form $\M_Y$ or $\mathcal N_Y$ for some subset $Y$ of $\Z$, since every sublocale always contains $\Z$ by the remark below Definition \ref{def:sublocale} and by Lemma \ref{lem:linordersublocale}, we only have to check that $\bigcap\mathcal A\in\M$ for each subset $\mathcal A$ of $\M$. Consider $\M_Y$. If $\mathcal A\subseteq\M_Y$ contains the empty set, then clearly $\bigcap\mathcal A=\emptyset\in\M_Y$. If $\mathcal{A}=\{\Z\}$, then $\bigcap\mathcal A=\Z\in\M_Y$. If $\mathcal A\neq\{\Z\}$ and does not contain the empty set, we have $\mathcal A\setminus\{\Z\}=\{\down y:y\in X\}$ for some non-empty subset $X$ of $Y$. By (\ref{eq:intersectionsubsofDZ}), we find that $\bigcap\mathcal\A$ either equals the empty set or, if $X$ is bounded from below, $\down\min(X)$. Since $\min(X)\in X\subseteq Y$ if $X$ is bounded from below, we find that in both cases the intersection is an element of $\M_Y$.

Now consider $\mathcal N_Y$ with $Y$ bounded from below. Let $\mathcal A\neq\{\Z\}$ be a subset of $\mathcal N_Y$ (the case $\mathcal A=\{\Z\}$ is handled in a similar way as for $\M_Y$). Then $\mathcal A\setminus\{Z\}=\{\down x:x\in X\}$ for some non-empty subset $X$ of $Y$, which is bounded from below, since $Y$ is bounded from below. By (\ref{eq:intersectionsubsofDZ}), we find that $\bigcap\mathcal A=\down\min X$, and since $X$ is bounded from below, we see that $\min X\in X\subseteq Y$. So $\bigcap\mathcal{A}\in\mathcal N_Y$. If $Y$ is not bounded from below, then $\bigcap\mathcal{N}_Y=\emptyset$ by (\ref{eq:intersectionsubsofDZ}), which is not contained in $\mathcal{N}_Y$, so $\mathcal N_Y$ cannot be a sublocale.
\end{proof}

 We can find the corresponding topologies on $\Z$ as follows. The fact that $\emptyset\in\M_Y$ says that $\emptyset\notin J_{\M_Y}(p)$ for each $p\in\Z$. By Theorem \ref{thm:mainthm}, we have $S\in J_{\M_Y}(p)$ if and only if for each $M\in\M_Y$ we have $S\subseteq M$ implies $p\in M$. Since each $S\in\D(\down p)$ is of the form $\down q$ for some $q\leq p$, and each $M\in\M_Y$ is of the form $\down y$ for some $y\in Y$, we find that $\down q\in J_{\M_Y}(p)$ if and only if for each $\down y\in \M_Y$ we have $\down q\subseteq\down y$ implies $p\in\down y$. In other words, $\down q\in J_{\M_Y}(p)$ if and only if $q\leq y\implies p\leq y$ for each $y\in Y$. Since $\Z$ is a linearly ordered set, this is equivalent to $\down q\in J_{\M_Y}(p)$ if and only if $y<p\implies y<q$ for each $y\in Y$. But this says exactly that $\down q\in J_{\M_Y}(p)$ if and only if $(Y+1)\cap\down p\subseteq\down q$, where $Y+1=\{y+1:y\in Y\}$.

 For $\mathcal{N}_Y$ the same analysis applies, except that $\emptyset\notin\mathcal{N}_Y$. So there might be a $p\in\P$ such that $\emptyset\in J(p)$. Now, we have $\emptyset\in J_{\mathcal{N}_Y}(p)$ if and only if $p\in\bigcap\mathcal{N}_Y$ if and only if $p\in\down\min(Y)$, which is the case if and only if $\down p\cap (Y+1)=\emptyset$. Thus
\begin{eqnarray*}
J_{\M_Y}(p) & = & \{S\in\D(\down p):(Y+1)\cap\down p\subseteq S\}\setminus\{\emptyset\}\\
J_{\mathcal{N}_Y}(p) & = &  \{S\in\D(\down p):(Y+1)\cap\down p\subseteq S\}.
\end{eqnarray*}
So we have $J_{\mathcal{N}_Y}=J_X$ and $J_{\M_Y}(p)=J_X(p)\setminus\{\emptyset\}$ for each $p\in\Z$, where  $X=Y+1$, which was assumed to be bounded from below in the first case. Since $J_X$ is defined for each subset $X$ of $\Z$, we conclude:
\begin{proposition}\label{prop:classificationGTofZ}
Let $X$ be a subset of $\Z$. Define $K_X$ by $K_X(p)=J_X(p)\setminus\{\emptyset\}$. Then Grothendieck topologies of the form $J_X$ and $K_X$ for some subset $X$ of $\P$ exhaust all Grothendieck topologies on $\Z$.
\end{proposition}

\section{Structures induced by a subset of a poset}
In the first section we found a special class of Grothendieck
topologies on a poset, namely those topologies induced by a subset of
the poset. Since there is a bijection between the set of Grothendieck
topologies of some poset $\P$ and the set of sublocales of $\D(\P)$,
there must be a map which assigns a sublocale to a subset of $\P$ . Since we also have bijections with the set of nuclei and the set of congruences on $\D(\P)$, there must be a
description of congruences and nuclei induced by a subset, too. The aim
of this section is to find these maps. Moreover, we shall prove that for Artinian posets
these maps are bijections.

The next theorem is the analogue of Proposition \ref{prop:correspsubsetsandtopologies} for equivalence classes of surjective maps instead of Grothendieckt topologies. For the definition of (lower) adjoints, we refer to Appendix \ref{Order Theory}.

\begin{theorem}\label{thm:ArtEmbMainThm}
Let $\P$ be a poset and $Y\subseteq\P$. Then the embedding $i_Y:Y\embeds\P$ induces a frame surjection $i_Y^{-1}:\D(\P)\to\D(Y)$, which maps each $A\in\D(\P)$ to $A\cap Y$. If we order $\PP(\P)$ by inclusion, then the assignment $G:\PP(\P)^\op\to\E(\D(\P))$ given by $Y\mapsto\left[i_Y^{-1}\right]_E$ is an embedding of posets with lower adjoint $F:\E(\D(\P))\to\PP(\P)^\op$ given by $[f]_E\mapsto X_f$, where $$X_f=\big\{p\in\P:f(\down p)\neq f(\down p\setminus\{p\})\big\}.$$
Moreover, if $\P$ is Artinian, then $G$ is an order isomorphism with inverse $F$.
\end{theorem}
Notice that since $F$ is the lower adjoint of $G$, Lemma \ref{lem:reflection} implies that $F$ is the left inverse of $G$.
\begin{proof}
Equip $\P$ with the Alexandrov topology, and regard $Y$ as a poset with ordering inherited from $\P$. Then $i_Y:Y\embeds\P$ is an order morphism, so by Lemma \ref{lem:ordermorphismiscontinuity}, it is continuous. Hence $i_Y^{-1}:\D(\P)\to\D(Y)$ is a frame morphism. Thus $G:\PP(\P)^\op\to\E(\D(\P))$ is well defined.
Let $Y\subseteq Z$. Then $\ker i_Z^{-1}\subseteq\ker i_Y^{-1}$. Indeed, let $(A,B)\in\D(\P)^2$ such that $(A,B)\in\ker i^{-1}_Z$. This is equivalent with $A\cap Z=B\cap Z$, hence $A\cap Y=B\cap Y$, since $Y\subseteq Z$. In other words, $(A,B)\in\ker i_Y^{-1}$. Then $\left[i^{-1}_Z\right]_E\leq\left[i_Y^{-1}\right]_E$ by Corollary \ref{cor:bijectieEFConF}, so $G:\PP(\P)^\op\to\E(\D(\P))$ is an order morphism.

In order to show that $F$ is well defined, let $[f]_E=[h]_E$. So there is an isomorphism $k$ such that $h=k\circ f$, whence $f(\down p)\neq f(\down p\setminus\{p\})$ if and only if $$h(\down p)=k\circ f(\down p)\neq k\circ f(\down p\setminus\{p\})=h(\down p\setminus\{p\}).$$ We conclude that $X_f=X_h$.
Now assume $[f]_E\leq[h]_E$. So there is a frame morphism $k$ such that $k\circ f=h$. Now, let $p\in X_h$. Then $$k\circ f(\down p)=h(\down p)\neq h(\down p\setminus\{p\})=k\circ f(\down p\setminus\{p\}),$$ which implies that $f(\down p)\neq f(\down p\setminus\{p\})$. Thus $p\in X_p$, and we conclude that $X_h\subseteq X_f$, so $F:\E(\D(\P))\to\PP(\P)^\op$ is an order morphism.

We show that $G$ is an embedding as follows. We have $p\in X_{i_Y^{-1}}$ if and only if $$i_Y^{-1}(\down p)\neq i_Y^{-1}(\down p\setminus\{p\})$$ if and only if $$Y\cap\down p\neq Y\cap(\down p\setminus\{p\})$$ if and only if $p\in Y$. So $Y=X_{i_Y^{-1}}$ for each $Y\subseteq\P$, or equivalently $F\circ G=1_{\PP(\P)}$. Now, let $Y,Z\subseteq\P$ such that $G(Z)\leq G(Y)$. Then $$Y=F\circ G(Y)\subseteq F\circ G(Z)=Z,$$  so $G:\PP(\P)^\op\to\E(\D(\P))$ is indeed an embedding of posets.

In order to show that $F$ is the lower adjoint of $G$, let $f:\D(\P)\to H$ be a frame surjection. We show that $\ker f\subseteq\ker i^{-1}_{X_f}$.
 Assume $(A,B)\in\ker f$, so $f(A)=f(B)$. Moreover, let $p\in A\cap X_f=i^{-1}_{X_f}(A)$. Then $f(\down p)\neq f(\down p\setminus\{p\})$. It follows from $p\in A$ that $\down p\setminus\{p\}\subseteq\down p\subseteq A,$ whence
\begin{eqnarray}
 \down p &  = & A\cap\down p;\label{eq:Acapp}\\
 \down p\setminus\{p\} & = & A\cap(\down p\setminus\{p\}).\label{eq:Acappsetminp}
\end{eqnarray}
Then (\ref{eq:Acapp}) implies
\begin{equation}\label{eq:fBcapp}
f(\down p)=f(A\cap\down p)=f(A)\wedge f(\down p)=f(B)\wedge f(\down p)=f(B\cap\down p).
\end{equation}
Now assume $p\notin B$. Then $B\cap\down p=B\cap(\down p\setminus\{p\})$, hence
\begin{eqnarray*}
f(\down p\setminus\{p\}) & = & f(A\cap(\down p\setminus\{p\}))=f(A)\wedge f(\down p\setminus\{p\})=f(B)\wedge f(\down p\setminus\{p\})\\
& = & f(B\cap(\down p\setminus\{p\}))= f(B\cap\down p)=f(\down p),
\end{eqnarray*}
where we used (\ref{eq:Acappsetminp}) in the first and used (\ref{eq:fBcapp}) in the last equality. But this contradicts $p\in X_f$, so we find that $A\cap X_f\subseteq B$. By interchanging $A$ and $B$ in this argument, we find $B\cap X_f\subseteq A$, so $f(A)=f(B)$ implies $A\cap X_f=B\cap X_f$. Thus $\ker f\subseteq\ker i^{-1}_{X_f}$, so $[f]_E\leq [i^{-1}_{X_f}]_E$ by Corollary \ref{cor:bijectieEFConF}. But this exactly expresses that $1_{\E(\D(\P))}\leq G\circ F$. We already found that $F\circ G=1_{\PP(\P)}$, which implies $F\circ G\leq 1_{\PP(\P)}$, so by Lemma \ref{lem:adjointequivalent} we find that $F$ is the lower adjoint of $G$.

Finally, assume that $\P$ is Artinian. We already showed that $\ker f\subseteq\ker i^{-1}_{X_f}$. For the inclusion in the other direction, first note that as a frame morphism $f$ preserves the order. So if $A\subseteq B$ in $\D(\P)$, we find $f(A)\leq f(B)$. Now assume that $(A,B)\in\ker i^{-1}_{X_f}$, i.e., $A\cap X_f=B\cap X_f$. We shall show that $p\in A$ implies $f(\down p)\leq f(B)$. From our assumption it immediately follows that $p\in B\cap X_f\subseteq B$ if $p\in A\cap X_f$. Thus $\down p\subseteq B$, so $f(\down p)\leq f(B)$. Assume $p\in A\cap X_f^c$. Since $p\notin X_f$, we find $f(\down p)=f(\down p\setminus\{p\})$. If $p\in\min(A)$, then $\down p\setminus\{p\}=\emptyset$, since $A$ is a down-set, so if $\down p\setminus\{x\}$, we obtain a contradiction with the minimality of $p$. Hence we obtain $$f(\down p)=f(\down p\setminus\{p\})=f(\emptyset)\leq f(B),$$ since $\emptyset\subseteq B$. So $p\in\min(A)$ implies $f(\down p)\leq f(B)$. Now assume that $f(\down q)\subseteq f(B)$ for each $q<p$. Then $$f(\down p)=f(\down p\setminus\{p\})=f\left(\bigcup_{q<p}\down q\right)=\bigvee_{q<p}f(\down q)\leq f(B),$$ the latter inequality by the assumption on $q<p$. By Artinian induction (see Appendix \ref{Order Theory}), we find that $f(\down p)\leq f(B)$ for each $p\in A\cap X_f^c$, and so for each $p\in A$. We find $$f(A)=f\left(\bigcup_{p\in A}\down p\right)=\bigvee_{p\in A}f(\down p)\leq f(B),$$
and by repeating the whole argument with $A$ and $B$ interchanged, we obtain $f(B)\leq f(A)$. Hence $i^{-1}_{X_f}(A)=i^{-1}_{X_f}(B)$ implies $f(A)=f(B)$, so $\ker i^{-1}_{X_f}\subseteq\ker f$.

Combining both inclusions, we find $\ker f=\ker i^{-1}_{X_f}$, hence Corollary \ref{cor:bijectieEFConF} implies that $\left[i^{-1}_{X_f}\right]_E=[f]_E$ for each $[f]_E\in\E(\D(\P))$, or equivalently, $G\circ F=1_{\E(\D(\P))}$. We conclude that $G$ is an order isomorphism with inverse $F$.
\end{proof}

We now introduce the following notation, which will be very helpful in what follows, since it extends the usual defintion of the Heyting implication.

\begin{definition}
 Let $\P$ be a poset and $X,Y$ subsets of $\P$. Then we define the \emph{Heyting implication}\index{Heyting implication}
\begin{equation}\label{eq:defheyimp}
 X\to Y=\bigcup\{A\in\D(\P):A\cap X\subseteq Y\}.
\end{equation}
\end{definition}
The next lemma shows that indeed we can see this operation on $\PP(\P)$ as an extension of the usual Heyting implication in the frame $\D(\P)$.

\begin{lemma}\label{lem:extHeyImp}
 Let $\P$ be a poset and $X\subseteq\P$. Then:
\begin{enumerate}
 \item The map $\PP(\P)\to\D(\P)$ given by $Y\mapsto X\to Y$ is the upper adjoint of the map $\D(\P)\to\PP(\P)$, $A\mapsto A\cap X$. That is, for each $Y\in\PP(\P)$ and $A\in\D(\P)$ we have
 \begin{equation}\label{eq:heyimpadjunction}
  A\cap X\subseteq Y\ \ \Longleftrightarrow\ \ A\subseteq X\to Y;
 \end{equation}
 \item The map $\D(X)\to\D(\P)$ given by $Y\mapsto X\to Y$ is the upper adjoint of the map $i_X^{-1}:\D(\P)\to\D(X)$, $A\mapsto A\cap X$. That is, (\ref{eq:heyimpadjunction}) holds for each $Y\in\D(X)$ and $A\in\D(\P)$.
 \item The map $\D(\P)\to\D(\P)$ given by $A\mapsto X\to A$ is the usual Heyting implication in the frame $\D(\P)$ if we assume that $X\in\D(\P)$;
\item For each $X,Y,Z\in\PP(\P)$ and $\{X_i:i\in I\},\{Y_i:i\in I\}\subseteq\PP(\P)$, we have
\begin{enumerate}
 \item[(i)] $X\to (Y\to Z)=Y\to(X\to Z)$;
 \item[(ii)] $\bigcap_{i\in I}(X\to Y_i)=X\to\bigcap_{i\in I}Y_i$;
 \item[(iii)] $\bigcap_{i\in I}(X_i\to Y)=\left(\bigcup_{i\in I}X_i\right)\to Y$.
\end{enumerate}
\end{enumerate}
\end{lemma}
\begin{proof}\
\begin{enumerate}
 \item By the distributivity law, the map $A\mapsto A\cap X$ preserves joins, and it clearly preserves all intersections. Hence it is a frame morphism, so by Lemma \ref{lem:frameadjoint}, it has an upper adjoint exacly given by $Y\mapsto X\to Y$.
 \item The image of $A\mapsto A\cap X$ is exactly $\D(X)$, which is a subframe of $\PP(\P)$, so if we restrict the codomain of the map $A\mapsto A\cap X$ to $\D(X)$, we obtain exactly $i_X^{-1}:\D(\P)\to\D(X)$. Now, the restriction of $Y\mapsto X\to Y$ to the domain $\D(X)$ clearly satisfies (\ref{eq:heyimpadjunction}) for each $Y\in\D(X)$ and each $A\in\D(\P)$, so it is the upper adjoint of $i_X^{-1}$, since adjoints are unique by Lemma \ref{lem:adjointequivalent}.
 \item If $X\in\D(\P)$, then the image of $A\mapsto A\cap X$ lies in $\D(\P)$. Just as in the proof of the second statement, we can restrict the codomain of $A\mapsto A\cap X$ to $\D(\P)$, and obtain a map $\D(\P)\to\D(\P)$, whose upper adjoint is the restriction of $Y\mapsto X\to Y$ to a map $\D(\P)\to\D(\P)$. But by definition of the Heyting implication in a frame (Definition \ref{def:heyimplicationinframes}), this upper adjoint is exactly the Heyting implication in $\D(\P)$.
\item
For (i), let $W\in\D(\P)$. Then $W\subseteq X\to(Y\to Z)$ if and only if $W\cap X\subseteq(Y\to Z)$ if and only if $W\cap X\cap Y\subseteq Z$, which in a similar way is equivalent to $W\subseteq Y\to(X\to Z)$. Hence (i) holds.

We note that the intersection is exactly the meet operation in both $\D(\P)$ and $\PP(\P)$. Since $Y\mapsto(X\to Y)$, is an upper adjoint, it preserves all meets (Lemma \ref{lem:adjoint}), which shows that property (ii) holds.
We can also show this in a direct way. We have $Z\subseteq\bigcap_{i\in I}(X\to Y_i)$ if and only if $Z\subseteq X\to Y_i$ for each $i\in I$ if and only if $Z\cap X\subseteq Y_i$ for each $i\in I$ if and only if $Z\cap X\subseteq\bigcap_{i\in I}Y_i$ if and only if $Z\subseteq X\to\bigcap_{i\in I}Y_i$.

Finally, (iii) follows in a similar way. We have $Z\subseteq\bigcap_{i\in I}(X_i\to Y)$ if and only if $Z\subseteq X_i\to Y$ for each $i\in I$ if and only if $Z\cap X_i\subseteq Y$ for each $i\in I$ if and only if $Z\cap\left(\bigcup_{i\in I}X_i\right)\subseteq Y$ for each $i\in I$ if and only if $Z\subseteq\left(\bigcup_{i\in I}X_i\right)\to Y$.
\end{enumerate}
\end{proof}

\begin{theorem}\label{thm:ArtMainThm}
 Let $\P$ be a poset and $X\subseteq\P$. Let $i_X:X\embeds\P$ be the embedding and $i_X^{-1}:\D(\P)\to\D(X)$ the induced frame map. Then:
\begin{enumerate}
 \item The subset Grothendieck topology $J_X$ on $\P$, which is complete, satisfies
\begin{equation}\label{eq:inducedtopology2}
J_X(p)  =  \{S\in\D(\down p):p\in X\to S\}
\end{equation}
for each $p\in\P$;
\item The map $j_X:\D(\P)\to\D(\P)$ defined by
\begin{equation}
 j_X(A)  =  X\to A
\end{equation}
for each $A\in\D(\P)$ is a complete nucleus on $\D(\P)$, and is equal to $(i_X^{-1})_*i_X^{-1}$. Moreover, $[j_X]_E=[i_X^{-1}]_E$;
\item The object $\theta_X$ defined by
\begin{equation}
 \theta_X = \ker i_X^{-1}= \{(A,B)\in\D(\P)^2:A\cap X=B\cap X\}
\end{equation}
  is a complete congruence on $\D(\P)$;
\item The object $\M_X$ defined by
\begin{equation}
\M_X  =  \{A\in\D(\P):A=X\to A\}=\{X\to A:A\in\D(\P)\}
\end{equation}
is a sublocale of $\D(\P)$.
\end{enumerate}
Moreover, the diagram
\begin{equation*}
\xymatrix{\Nuc(\D(\P))\ar[rr]^{j\mapsto\M_j}\ar[dd]_{j\mapsto \theta_j} && \Sub(\D(\P))^\op\ar[dd]^{\M\mapsto J_\M}\\
&\PP(\P)^\op\ar[ul]^{X\mapsto j_X}\ar[ur]_{X\mapsto\M_X}\ar[dl]_{X\mapsto\theta_X}\ar[dr]^{X\mapsto J_X}&\\
\Con(\D(\P))\ar[rr]_{\theta\mapsto J_\theta} && \G(\P)}
\end{equation*}
commutes and the inner four maps of the diagram are embeddings of posets, which are bijections if $\P$ is Artinian.
\end{theorem}
\begin{proof}
We have defined $J_X$ in Proposition \ref{prop:defJX}, where we also showed that $J_X$ is complete. Now, by the adjunction associated to the Heyting implication, we have $X\cap\down p\subseteq S$ if and only $\down p\subseteq X\to S$, which is the case if and only if $p\in X\to S$.

In Theorem \ref{thm:ArtEmbMainThm}, we found an order embedding $\PP(\P)^\op\to\E(\D(\P))$ given by $X\mapsto\left[i_X^{-1}\right]_E$. If we compose this embedding with the order isomorphism in Proposition \ref{prop:bijectieEFNucF} between $\E(\D(\P))$ and $\Nuc(\D(\P))$, we obtain an order embedding $\PP(\P)^\op\to\Nuc(\D(\P))$ given by $X\mapsto (i_X^{-1})_*i_X^{-1}$, which is an order isomorphism if $\P$ is Artinian. Now, by Lemma \ref{lem:extHeyImp}, the upper adjoint $(i_X^{-1})_*:\D(X)\to\D(\P)$ of $i_X^{-1}$ is given by $Y\mapsto X\to Y$, so the nucleus associated to $X$ is given by the map $$A\mapsto X\to(A\cap X).$$ In Lemma \ref{lem:extHeyImp} we found that $(i_X^{-1})_*:\D(X)\to\D(\P)$ is the restriction of $(X\to\cdot):\PP(\P)\to\D(\P)$, which is also an upper adjoint, and hence preserves the order given by the inclusion. Thus we have $$X\to (A\cap X)\subseteq X\to A$$ for each $A\in\D(\P)$. On the other hand, if $A\in\D(\P)$ and $Y\subseteq X\to A$ for some $Y\in\PP(\P)$, we find $Y\cap X\subseteq A$ by the adjunction associated to the Heyting implication. But this implies $Y\cap X\subseteq A\cap X$, hence $Y\subseteq X\to(A\cap X)$. Thus $$X\to A=X\to(A\cap X)$$ for each $A\in\D(\P)$, so the nucleus $(i_X^{-1})_*i_X^{-1}$ associated to $X$ is $j_X$. Moreover, by Proposition \ref{prop:bijectieEFNucF}, we have $[j_X]_E=[i_X^{-1}]_E$. We conclude that the map $X\mapsto j_X$ is an embedding of the poset $\PP(\P)^\op$ into $\Nuc(\D(\P))$, which is an order isomorphism if $\P$ is Artinian.

We can complete the proof of this theorem if we can show that $j_X$ corresponds to $J_X$, $\theta_X$ and $\M_X$, respectively under the various order isomorphisms of Theorem \ref{thm:mainthm}. That is, $J_X=J_{j_X}$, $\theta_X=\theta_{j_X}$ and $\M_X=\M_{j_X}$. Then by Corrolary \ref{cor:completestructures} we find that $j_X$, $\theta_X$ and $\M_X$ are complete, since $J_X$ is complete. Let $p\in\P$. Then $$J_{j_X}(p)=\{S\in\D(\down p):p\in j_X(S)\}=\{S\in\D(\down p):p\in X\to S\}=J_X.$$

Let $(A,B)\in\D(\P)^2$. Then $(A,B)\in\theta_{j_X}$ if and only if $j_X(A)=j_X(B)$ if and only if $X\to A=X\to B$. On the other hand, we have $A\theta_X B$ if and only if $A\cap X=B\cap X$. So assume $X\to A=X\to B$. Since $A\cap X\subseteq A$, we have $$A\subseteq X\to A=X\to B.$$ This implies $A\cap X\subseteq B$, so $A\cap X\subseteq B\cap X$. Reversing $A$ and $B$ gives $A\cap X=B\cap X$. So $A\theta_{j_X}B$ implies $A\theta_X B$. Now assume $A\cap X=B\cap X$. Then $A\cap X\subseteq B$, so $A\subseteq X\to B$. Since $Y\mapsto(X\to Y)$ is order preversing as an upper adjoint, we find that $$X\to A\subseteq X\to(X\to B)=j_X\circ j_X(B)=j_X(B)=X\to B.$$ Reversing $A$ and $B$ gives $X\to A=X\to B$, so indeed $\theta_{j_X}=\theta_X$.

Finally, we have $$M_{j_X}=\{A\in\D(\P):j_X(A)=A\}=\{A\in\D(\P):X\to A=A\}=\M_X.$$
\end{proof}

\begin{corollary}\label{cor:sublocaleisomorphictoDX}
Let $\P$ be a poset and $X\subseteq\P$. Then $\M_X$ is isomorphic (as a frame) to $\D(X)$.
\end{corollary}
\begin{proof}
Assume $\M=\M_X$. By the previous theorem and Theorem \ref{thm:mainthm}, $\M_X$ is the image of $j_X$, and $$\ker j_X= \ker j_{\theta_X}=\theta_X=\ker i_X^{-1}.$$ Since $\D(X)$ is the image of $i_X^{-1}$, we find by Theorem \ref{thm:homomorphismtheoremframes}: $$\M_X=j_X[\D(\P)]\cong \D(\P)/\ker j_X=\D(\P)/\ker i_X^{-1}\cong\D(X).$$
\end{proof}

We can now state a number of statements all equivalent with the Artinian property. Properties (i-iv) and (vii,viii) are also listed in \cite[Theorem 4.12]{EGP}.
\begin{theorem}
Let $\P$ be a poset. Then the following statements are equivalent:
\begin{enumerate}
\item[(i)] $\P$ is Artinian;
 \item[(ii)] Every non-empty downwards directed subset of $\P$ contains a least element;
 \item[(iii)] $\P$ satisfies the descending chain condition;
\item[(iv)] $\P$ equipped with the lower Alexandrov topology is sober;
\item[(v)] Each Grothendieck topology $J$ on $\P$ is a subset Grothendieck topology: there is a subset $X$ of $\P$ such that $$J(p)=\{S\in\D(\down p):X\cap\down p\subseteq S\}$$ for each $p\in\P$.
\item[(vi)] All Grothendieck topologies on $\P$ are complete;
 \item[(vii)] Each congruence $\theta$ on $\D(\P)$ is of the form $$\theta=\{(A,B)\in\D(\P)^2:A\cap X=B\cap X\}$$ for some subset $X$ of $\P$;
\item[(viii)] All congruences on $\D(\P)$ are complete;
 \item[(ix)] Each nucleus $j$ on $\D(\P)$ is of the form $j(A)=X\to A$ for each $A\in\D(\P)$ for some subset $X$ of $\P$;
\item[(x)] All nuclei on $\D(\P)$ are complete;
 \item[(xi)] Every sublocale of $\D(\P)$ is isomorphic to $\D(X)$ for some subset $X$ of $\P$.
 \end{enumerate}
\end{theorem}
\begin{proof}
The equivalence between (i), (ii) and (iii) is proven in Lemma \ref{lem:equivalentdefinitionsArtinian}. The equivalence between (i) and (iv) is proven in Proposition \ref{prop:equivalencesoberartinian}. By Proposition \ref{prop:correspsubsetsandtopologies}, (i) implies (v). By Proposition \ref{prop:defJX}, (v) implies (vi). In order to show that (vi) implies (i), assume that $\P$ is not Artinian. Then by (iii), $\P$ contains a non-empty downwards directed subset $X$ without a least element. By Proposition \ref{prop:densetopologydef} it follows that $K=J_\atom^X$, the atomic Grothendieck topology on $X$ is not complete. By Proposition \ref{prop:extensionofGrothTop}(vii), it follows that $J_K$ is a Grothendieck topology on $\P$, which is not complete.
By Theorem \ref{thm:ArtMainThm} we find that (v) is equivalent to (vii) and (ix), as well as to the statement that each sublocale of $\D(\P)$ is of the form $\M_X$ for some subset $X$ of $\P$. By Corollary \ref{cor:sublocaleisomorphictoDX}, this is equivalent to (xi). Finally, the equivalence between (vi), (viii) and (x) is assured by Corollary \ref{cor:completestructures}.
\end{proof}

\end{document}